\documentclass[12pt]{amsart}
\usepackage{amsfonts,amssymb,amscd}
\usepackage{oldgerm}
\usepackage{amsfonts}
\usepackage{amssymb}
\usepackage{graphics}
\newtheorem{teo}{Theorem}[section]
\newtheorem{prop}[teo]{Proposition}
\newtheorem{lema}[teo]{Lemma}
\newtheorem{obs}[teo]{Remark}
\newtheorem{defnc}[teo]{Definition}
\newtheorem{coro}[teo]{Corollary}

\newcommand{\C}{{\mathbb C}}
\newcommand{\CP}{{\mathbb C \mathbb P}}
\newcommand{\R}{{\mathbb R}}
\newcommand{\Q}{{\mathbb Q}}
\newcommand{\Z}{{\mathbb Z}}
\newcommand{\N}{{\mathbb N}}
\newcommand{\fol}{{\mathcal F}}
\newcommand{\tilf}{{\widetilde{\mathcal F}}}
\newcommand{\tXX}{{\widetilde{X}}}

\newcommand{\wdc}{{\widetilde{\mathbb C}}}

\newcommand{\da}{\alpha}
\newcommand{\db}{\beta}
\newcommand{\dl}{\lambda}
\newcommand{\de}{\varepsilon}
\newcommand{\calh}{\mathcal{H}}
\newcommand{\cald}{\mathcal{D}}
\newcommand{\cale}{\mathcal{E}}
\newcommand{\Ff}{\mathbb{F}}

\newcommand{\calp}{{\mathcal P}}

\newcommand{\partz}{\partial /\partial z}

\begin{document}

\title{Uniformizing complex ODEs and applications}
\author{Julio C. Rebelo \; \;  \& \; \; Helena Reis}

\begin{abstract}
We introduce a method to estimate the size of the domain of definition of the solutions of a meromorphic vector field on a
neighborhood of its pole divisor. The corresponding techniques are, in a certain sense, quantitative versions of some
well-known phenomena related to the presence of metrics with positive curvature. Several applications of these ideas are
provided including a type of ``confinement theorem'' for solutions of complete polynomial vector fields on $\C^n$  and obstructions
for certain (germs of) vector fields to be realized by a global holomorphic vector field on a compact K\"ahler manifold. As a
complement, a new approach to certain classical equations is proposed and detailed in the case of Halphen equations.
\end{abstract}

\maketitle

\noindent \hspace{1.2cm} {\small AMS-Classification  37F75, 34M05, 34M55}

\bigskip

\section{Introduction}
The object of this paper is a method to investigate the domain of
definition of solutions for meromorphic vector fields.
The method is quite general in that it applies to arbitrarily high
dimensions. Yet, it provides new results already in dimension~$3$. This work is then essentially
constituted by two parts, the first one corresponding to a general setup along with the basic
estimates/results. These are then exploited in some applications that are carried out in the second part.
To greater or lesser extent, the
applications given here arise from following the solution of a (complex) polynomial/rational vector field over
``special real paths going off to infinity''.

This Introduction is aimed only at stating the main applications considered in this text. These were chosen not with a purpose of
being the sharpest possible but rather of indicating ways of exploiting the basic phenomena on which our analysis relies.
In Section~2, we shall provide a more detailed discussion explaining our point of view and underlining the
common structures behind the theorems below. It is also to be noted these applications concern very special types
of vector fields (or of differential equations) such as complete vector fields and Halphen equations. Nonetheless the setting is
also well-adapted to investigating differential equations having meromorphic solutions including several classical equations
appearing in Mathematics and Physics. For example equations concerned with the works of Painlev\'e and Chazy as well as
certain Lorenz systems \cite{lorenz}. Through the work of Ablowitz, Segur and others, some of statements, e.g. Theorems~A
and~A', can also be adapted to solutions of certain non-linear evolution equations and/or to solutions of linear integral
equations in Gelfand-Levitan-Marcenko theory, see \cite{Ablowitz-1}, \cite{Ablowitz-2}. More details on these issues can be
found in Section~2. Whereas these connections will not be developed here,
they provide a clear indication that further applications of our techniques may be found in the future.

Let $X$ be a polynomial vector field of degree at least~$2$ on $\C^n$. Suppose that $X$ is complete i.e.
its complex solutions are defined for all $T \in \C$. When $X$ happens to be {\it completely integrable}\, i.e. when it admits $n-1$
independent first integrals, its orbits can be compactified into
rational curves by adding to them some ``singular points at infinity''. This fact can be interpreted as a type of confinement
phenomenon for the corresponding solutions. Our first result is a sharp, whereas weaker, generalization of this confinement phenomenon
to every complete polynomial vector field.
To state it, we proceed as follows. Being polynomial, $X$ defines a singular holomorphic foliation $\cald$ on $\C P(n) = \C^n \cup \Delta_{\infty}$
viewed as a compactification of $\C^n$. Consider a leaf $L$ of $\cald$
(details on the definition of ``leaf'' in the singular context
can be found in Section~2.2). On $L$ two (singular) oriented real one-dimensional foliations $\calh, \, \calh^{\perp}$
are going to be defined. They will depend on the leaf $L$ of $\cald$ in a regular way as it will be apparent from
their definitions (cf. Sections~3 and~6). More importantly, $\calh, \, \calh^{\perp}$
are mutually orthogonal with respect to the conformal structure of $L$, in fact, they agree respectively with the
real foliation and the purely imaginary foliation induced on $L$ by a certain Abelian form. Since $L$ is endowed
with a conformal structure, it makes sense to define also foliations $\calh^{\theta}$ whose (oriented) trajectories
makes an angle~$\theta$ with the oriented trajectories of $\calh$ ($\theta \in [-\pi /2, \pi /2]$). The trajectories of
these foliations define the ``directions of confinement'' for $L$ as will follow from Theorem~A below. In the sequel
$\Phi : \C \times \C^n \rightarrow \C^n$ stands for the holomorphic flow generated by $X$ whereas ${\rm Sing}\, (\cald)
\subset \C P(n)$ denotes the singular set of $\cald$.

\vspace{0.1cm}

\noindent {\bf Theorem~A}. {\sl Suppose that $X$ is a complete polynomial vector field of degree at least~$2$ on $\C^n$.
Fix an arbitrarily small neighborhood $V$ of $({\rm Sing}\, (\cald) \cap \Delta_{\infty}) \cup
{\rm Sing}\, (X)$ in $\C P(n)$ and suppose we are given a point $p \in \C^n$, $X (p) \neq 0$, and an angle $\theta \in
(-\pi/2 ,\pi/2)$. Denote by $L_p$ (resp. $l_p^{+,\theta}$) the leaf of $\cald$ through $p$ $($resp. the semi-trajectory of
$\calh^{\theta}$ initiated at $p)$ and consider the lift $c:[0, \infty) \rightarrow \C$ of
$l_p^{+,\theta}$ by $\Phi$, i.e. $t \in [0, \infty) \mapsto
\Phi (c(t), p)$ is a one-to-one parametrization of $l_p^{+,\theta}$ $(c(0)=0)$. Then there is a constant $C$ such that
$$
{\rm meas} \, \left( \{ t \in [0, \infty) \; ; \;  \Phi (c(t), p) \not\in V \} \right) < C \, ,
$$
where ${\rm meas}$ stands for the usual Lebesgue measure on $\R$.
}

\vspace{0.1cm}

The preceding theorem states that the trajectory $l_p^{+,\theta}$ spend most of its ``life'' in the
neighborhood $V$ and hence arbitrarily close to the singular points of $\cald$ or of $X$. Furthermore the
constant $C$ varies continuously with $\theta$. In particular, if we consider a compact interval $[-\pi/2 + \delta ,
\pi /2 - \delta] \subset (-\pi/2 ,\pi/2)$ then $C$ can be chosen so that the above estimate holds
for every $\theta \in [-\pi/2 + \delta , \pi /2 - \delta]$. The existence of $C$ uniform for
$\theta \in [-\pi/2 + \delta , \pi /2 - \delta]$ allows us to generalize the statement to paths
$c \subset L_p$ more general than the trajectories of $\calh^{\theta}$, for example we may consider
paths $c$ as before such that the angle made at the point $c(t)$ by the speed vector $c'(t)$ and the
foliation $\calh$ lies in $[-\pi/2 + \delta , \pi /2 - \delta]$ for all $t$. The interested reader will have
no difficulty in adapting the statement of Theorem~A to these more general situations.

Confinement phenomena are in stark contrast with ergodicity so that it is natural to search for a variant of Theorem~A
focusing on the ``area'' defined in $\C$ by those values of $T$ for which $\Phi (T,p) \in V$. This variant might be viewed,
in particular, as a ``super non-ergodic'' phenomenon for complete vector field. To
state it, let $B_r \subset \C$ denote the disc of radius $r$ about $0 \in \C$. A continuous path properly
embedded $c: (-\infty, \infty) \rightarrow \C$ is a {\it separating curve} if it is of class $C^{\infty}$
with possible exception of a discrete set and it is either periodic or it satisfies the condition
$\lim_{t \rightarrow -\infty} c(t) = \infty$, $\lim_{t \rightarrow \infty} c(t) = \infty$ (where the last
conditions means that the curve eventually leaves every compact subset of $\C$). A separating curve divides
$\C$ in at least two connected component and at least one of these components is unbounded. Then we have:

\vspace{0.2cm}

\noindent {\bf Theorem~A'}. {\sl Let $X$, $V$, $L_p$ and $p \in \C^n$ be as in the statement of Theorem~A. Consider the
parametrization of $L_p$ by $\C$ (possibly as a covering map) which is given by $\Phi_p (T) = \Phi (T, p)$. Then there
exists a separating curve $c: (-\infty, \infty) \rightarrow \C$, $\Phi_p(c(0)) =p$, and an unbounded component $\mathcal{U}^+$
of $\C \setminus c(t)$ such that the following holds: the set $\mathcal{T}_V \subset \mathcal{U}^+ \subset \C$
defined by
$$
\mathcal{T}_V = \{ T \in \mathcal{U}^+ \subset \C \; ; \; \Phi (T,p) \in V \} \,
$$
satisfies
$$
\lim_{r \rightarrow \infty} \frac{ {\rm Meas}\, (\mathcal{T}_V \cap B_r)}
{{\rm Meas}\, (\mathcal{U}^+ \cap B_r)} =1 \, ,
$$
where ${\rm Meas}$ stands for the usual Lebesgue measure of $\C$ ($\simeq \R^2$).}

\vspace{0.1cm}

The above statement contrasts markedly with various equidistribution results obtained by Fornaess-Sibony and
also studied by X. Gomez-Mont and his collaborators, see \cite{sibony-1}, \cite{sibony-2}, \cite{xavier1}.
The reader will note that these authors  work in a generic setting having ``hyperbolic nature'' whereas the
previous statements are closer to the non-generic ``parabolic'' case.

Unlike most standard averaging theorems, the statement above holds for every point
$p \in \C^n$ and not only for almost every point. Besides, it is easy to conclude from the proof
given in Section~6.1 that for almost all points $p$ the corresponding separating curve is smooth.

In fact, the above mentioned separating curve has a natural interpretation as a geodesic
for a suitable singular flat structure on $\C$.
Furthermore this (singular) flat structure on $\C$ has ``bounded geometry'' in a natural sense despite the lack
of compactness of $\C$. The nature of this ``bounded geometry'' issue deserves some additional comments (the
reader is referred to Section~6.1 for details). The notion of bounded geometry is, indeed, related to
the analogous statement concerning the leaves of a (regular) foliation defined on a compact manifold, cf. for example \cite{psullivan}.
More precisely, consider for a moment a regular foliation defined on some compact
manifold $M$ and some geometric object, for example a Riemannian metric, defined on $M$. The restriction of
this metric to a (possibly non-compact) leaf $L$ of the foliation in question must have ``bounded geometry'' (in the
case of Riemannian metrics this means that
the injectivity radius and the various curvatures are bounded) regardless of whether or not $L$ is compact.
This classical observation boils down to the fact that the coefficients of the metric are, ultimately,
defined on the compact manifold $M$ and therefore are ``bounded'' in a natural sense. Since only the compact
nature of the ambient manifold $M$ plays a role in the discussion, the same argument
also applies for ``foliated objects'', such as Riemannian metrics defined only on the tangent bundle of
the foliation or, more directly, defined on the corresponding leaves (provided that
they vary ``continuously'' from leaf to leaf).
Clearly, none of this needs to hold if the ambient manifold $M$ is not compact. Having recalled these simple
facts, let us go back to our complete polynomial vector field defined on $\C^n$. Whereas $\C^n$ is not compact, it can
be compactified into $\C P (n)$ and, as already pointed out, the foliation $\cald$ associated to $X$ extends to
$\C P(n)$. As far as ``bounded geometry'' for the leaves of $\cald$ is concerned, the issue is then to decide
whether or not the ``geometric object'' in question can be extended to all of $\C P(n)$ as well. For example,
considering the standard setting where the leaf $L$ is contained in $\C^n$, this leaf is endowed with a flat structure,
or equivalently, with a (singular) abelian form (called time-form) induced by duality with the restriction of $X$
to $L$ itself. This abelian form, however, does not extended to $\C P(n)$ since $X$ has poles on the hyperplane at
infinity $\Delta_{\infty}$ and, as a matter of fact,
the geometry arising from the abelian form in question is not ``bounded'' in general.
Nonetheless, the flat structure for which the above mentioned separating curve happens to be a geodesic does have
an extension to $\Delta_{\infty}$, cf. Remark~\ref{remarkboundedgeometry} in Section~6. Furthermore the ``extended'' flat structure
still varies ``continuously'' with the leaves which, in turn, guarantees the existence of ``bounds'' for the
corresponding geometry.

Let us now go back to the statement of Theorem~A'. The ``bounded geometry'' nature the preceding flat structure
implies, in particular, that ${\rm Meas}\, (\mathcal{U}^+ \cap B_r)$
is, in fact, comparable to the euclidean measure of large discs $B_r$. Since ``large discs'' are also used
in the construction of Ahlfors currents, the previous statement
may look a bit surprising since these currents do not charge singular points. Explanation for this difference is however easy
since the construction of Ahlfors currents is based on the ``global volume'' of a leaf and this may have little relation with
the asymptotic behavior of an actual solution. To be more precise, fix a diffeomorphism  between $\C$ and a leaf $L$,
for example a time-$t$ diffeomorphism $\Phi_t$ induced by the corresponding vector field. To construct Ahlfors current
the ambient manifold is equipped with a Hermitian metric which is then pulled back
by $\Phi_t$ to yield a metric $d_{\C}$ on $\C$. The desired current is then constructed by choosing a suitable
sequence of discs $B_{r_i}$ whose radii $r_i$ are measured with
respect to $d_{\C}$ and satisfy $r_i \rightarrow \infty$. Clearly a ``small'' neighborhood $V$ of a singular point in $M$ has
small diameter for the fixed Hermitian metric and so does a connected component $(L \cap V)_0$ of $L \cap V$. The diameter
(resp. the ``area''), of $\phi_t^{-1} ((L \cap V)_0)$ w.r.t. $d_{\C}$ is therefore small
as well. Now, it should be noted that $d_{\C}$ may differ markedly from the euclidean metric on $\C$ so that
the euclidean area of $\phi_t^{-1} ((L \cap V)_0)$ might be ``large''. The proofs of the preceding theorems
will make it clear that this phenomenon is precisely what happens in these cases.
As a conclusion, whereas Ahlfors currents are among the most efficient tools for studying (singular) holomorphic
foliations possessing leaves covered by $\C$, they might be less so when the main object in study is the actual solution
of a differential equation.

The statement of Theorems~A and~A' indicate that the structure of the singularities of $\cald$ lying in $\Delta_{\infty}$ must
bear significant information on the global dynamics of corresponding vector fields. This is an idea that can be thought of as a principle
similar to Painlev\'e's test for differential equations having ``meromorphic solutions'' which actually constitutes a context
where our methods can equally well be applied. In a sense, this might provide an element of explanation for the remarkable
effectiveness of Painlev\'e's test. In any case, letting this principle guide us, it is natural to wonder that complete vector fields whose associated foliations $\cald$ have only
``simple singularities'' in $\Delta_{\infty}$ must be amenable to a detailed global analysis.
Throughout this paper, by ``simple'' singularity, it is meant
the following types of singular points $q \in \Delta_{\infty}$ for $\cald$:
\begin{enumerate}
\item Non-degenerate singularities: this means that $\cald$ can locally be represented by a vector field having
non-degenerate linear part at $q$ (i.e. the Jacobian matrix of $X$ at $q$ is invertible, equivalently, it possesses $n$
eigenvalues different from zero). Besides, since resonances may arise, we assume that $q$ is {\it not of Poincar\'e-Dulac type},
i.e. if all the eigenvalues of $\cald$ at $q$ belong to $\R_+^{\ast}$ then $\cald$ must be locally linearizable about $q$.

\item Codimension~$1$ saddle-nodes: these are singularities of $\cald$ lying in $\Delta_{\infty}$
whose eigenvalue associated to the direction transverse to $\Delta_{\infty}$ is equal to {\it zero} whereas it has $n-1$
eigenvalues different from zero and corresponding to directions contained in $\Delta_{\infty}$. Again we require that
the $(n-1)$-dimensional singularity induced on the plane $\Delta_{\infty}$ should not be a singularity of Poincar\'e-Dulac type.
\end{enumerate}

Note that singular points of $\cald$ as in item~(1) above are necessarily isolated though this is no longer the case for
Codimension~$1$ saddle-nodes since these singularities may be contained in
a curve of singularities of $X$ transverse to $\Delta_{\infty}$.

Before stating Theorem~B, we need to fix some terminology. Recall that in algebraic geometry a {\it rational fibration}\,
on a manifold $M$ consists of a non-constant holomorphic map $f : M \rightarrow N$, where $N$ has dimension one less than the
dimension of $M$, which satisfies the following two conditions:
\begin{itemize}

\item The map $f$ must be a submersion away from a codimension~$1$ subset of $M$.

\item The fiber over a ``generic'' point $p \in N$ must be a rational curve.
\end{itemize}
This definition still makes sense when $f$ is defined only on the complement of an analytic subset $S$ with codimension
at least~$2$ in $M$, modulo a possible ``compactification'' of the fibers over generic points in $M$.
To avoid confusion with established literature, this slightly more general structure is going to be called a {\it completely
integrable rational foliation}\, (the name non-linear pencil might also be appropriate, if a bit confusing).
Note that given a completely integrable rational foliation on $M$, it is possible to birationally modify $M$ so as to turn
the foliation in question into an actual rational fibration naturally induced by the initial map $f$. Next, we have:

\vspace{0.1cm}

\noindent {\bf Theorem~B}. {\sl Let $X$ be a complete polynomial vector field on $\C^n$ whose singular set has codimension at
least~$2$. Suppose that all singularities of $\cald$ lying in $\Delta_{\infty}$ are as in items~(1) or~(2)
above. Then the foliation $\cald$ associated to $X$ is a completely integrable rational
foliation on $\C P(n)$, i.e. $X$ is completely integrable.}

\vspace{0.1cm}

Note that the assumption that $X$ is {\it complete} as vector field is indispensable for the preceding statement
and cannot be replaced by other weaker standard notions such as the slightly weaker condition of having
``meromorphic solutions defined on $\C$''. In fact, consider the pair of commuting vector fields given by
\begin{eqnarray*}
Z_0  & = & (-3x^2 +y^2 +2xz) \frac{\partial}{\partial x} + 2y(-3x +2z) \frac{\partial}{\partial y} + 2z (3x-z) \frac{\partial}{\partial z} \\
Z_{\infty} & = & 2y (-x+z) \frac{\partial}{\partial x} + (3x^2 -y^2) \frac{\partial}{\partial y} + 2yz \frac{\partial}{\partial z} \, .
\end{eqnarray*}
Consider also the linear span of $Z_0, \, Z_{\infty}$, i.e. all vector fields that are obtained as a linear
combination of $Z_0, \, Z_{\infty}$. It is shown in~\cite{guillotCRAS} that the solutions of every element in this family of
vector fields are meromorphic functions defined on all of $\C$. In other words, these vector fields are ``very close''
to complete vector fields. In addition, every two members of this family have essentially the same
simple singularities on $\Delta_{\infty}$ and these are simple in the above indicated sense. Nonetheless, this family
contains an infinite set of vector fields whose underlying foliations are not completely integrable in the sense
that their leaves cannot be compactified into Riemann surfaces (and in particular they cannot be compactified into
rational curves). This example sheds light on the importance of the assumption on ``completeness'' made in
the statement of Theorem~B.

Theorem~B will be proved in Section~6.2. The statement of this theorem may be compared to results of
\cite{bruno} for complete polynomial vector fields on $\C^2$. It is to be noted that the results of \cite{bruno}
chronologically preceded the classification obtained in \cite{marco3}. Also, the more recent paper
\cite{guillotreb} contains a general classification theorem for meromorphic vector fields admitting ``maximal solutions''
on algebraic surfaces and these include complete vector fields as in \cite{marco3}. All these
questions are however wide open for $n \geq 3$ and Theorem~B appears as a contribution to
them.

To have a better appreciation of the difficulties involved in these problems, following \cite{guillotreb},
let us consider the case of {\it semi-complete vector fields}, i.e. vector fields admitting maximal solutions.
Recall that a vector field is said to be semi-complete on a domain $U$ if its solution $\phi$ verifying $\phi (0) =p \in U$
is defined on a {\it maximal domain}\, of $\C$ for all $p \in U$. Here a domain $V \subseteq \C$ where the solution $\phi$
is defined is said to be {\it maximal}\, if for every point $\hat{T}$ in the boundary $\partial V$ of $V$ and every
sequence $\{ T_i \} \subset V$ such that
$T_i \rightarrow \hat{T}$, the sequence $\phi (T_i)$ leaves every compact set in $U$, cf. Section~2.2
for further details. Clearly, a complete vector field is automatically semi-complete since we can take $V = \C$ so as to have
$\partial V =\emptyset$. If polynomial semi-complete vector fields on $\C^n$ are
considered, then even the quadratic homogeneous case
is already hard to understand once $n \geq 3$. In fact, A. Guillot has conducted detailed research
about semi-complete quadratic homogeneous vector fields in
\cite{guillotFourier}, \cite{guillotIHES}. In \cite{guillotFourier}, by building on a certain variant of Painlev\'e test,
he introduced certain lattices (of coefficients) where all these vector fields are to be found whereas in
\cite{guillotIHES} he studied the special case of Halphen's
vector fields and the related problem about actions of ${\rm PSL}\, (2, \C)$ on compact $3$-manifolds. The
beauty and depth of these results motivated us to try to apply our techniques to vector fields satisfying the conditions
stated in \cite{guillotFourier} which will be said to belong to the Painlev\'e-Guillot lattice (the reader interested
in the case $n = 3$ is referred to \cite{guillotThesis} for a specially detailed discussion).

Note that a semi-complete vector field of $\C^3$ belonging to Painlev\'e-Guillot lattice may not be complete and,
moreover, its orbits (thought of as leaves of the associated foliation) may be hyperbolic Riemann surfaces. The
latter situation actually occurs with Halphen vector fields except for a few special cases cf. \cite{guillotIHES},
\cite{guillotThesis} or Section~7.2. In this Introduction, by a Halphen vector field it is always meant a
``hyperbolic'' Halphen vector field, cf. Section~7.2.
These vector fields are semi-complete and the maximal domains of definitions for their solutions are either
a bounded region of $\C$ or a hyperbolic unbounded region (for example the complement of a disc).

In the various classification results obtained by Guillot on quadratic semi-complete vector fields,
see \cite{guillotThesis}, \cite{guillotFourier}, \cite{guillotChazy}, there is a special class of ``exceptional
cases'' whose dynamics is very hard to be understood. In this direction, our methods allow us to say something
of non-trivial about these dynamics by considering the existence of {\it dicritical singularities at infinity}\,
for vector fields belonging to Painlev\'e-Guillot lattices. Given a vector field $X$ in a
Painlev\'e-Guillot lattice, a singular point for the associated foliation lying in the hyperplane at infinity
which has all its eigenvalues contained in $\R_+$ will be called a
{\it dicritical singularity at infinity}\, for $X$. Now we have:

\vspace{0.2cm}

\noindent {\bf Theorem~C}. {\sl Suppose that $X$ is a holomorphic vector field defined on compact manifold $M$.
Consider a singularity $p \in M$ of $X$ and denote by $X_k$ the first non-zero homogeneous
component of the Taylor series of $X$ at $p$. Suppose that one of the following condition holds:

\noindent \hspace{1.0cm}  $\bullet$ $X_k$ is a vector field in Painlev\'e-Guillot lattice having no
dicritical singularity at infinity.

\noindent \hspace{1.0cm}  $\bullet$ $X_k$ is a hyperbolic Halphen vector field (in case $M$ has dimension~$3$).

\noindent Then $M$ does not a carry a K\"ahler structure.}

\vspace{0.1cm}

Note that the second item of Theorem~C is sharp in the sense that \cite{guillotIHES} contains examples of compact
$3$-manifolds equipped with a global holomorphic vector field exhibiting the singularity of a hyperbolic Halphen
vector field. Naturally the corresponding manifolds are not K\"ahler.

As to the first item, we are aware of no explicit example of vector field having no dicritical singularity at infinity
and these do not exist for $n=2$. They are unlikely to exist for $n=3$, though we have no clear idea of what
may happen in higher dimensions. In fact, from the known (low-dimensional) cases, it appears that the quadratic vector fields
in question have a ``tendency'' to exhibit dicritical singularities at infinity. In this sense, as stated,
the first item of Theorem~C may be void. However, there is a number of immediate ways to turn this item into meaningful
statements about the dynamics of the vector field in question when dicritical singularities at infinity are present,
cf. Section~2.2. For example, we have:

\vspace{0.2cm}

\noindent {\bf Theorem~C'}. {\sl Suppose that $X$ is a holomorphic vector field defined on compact
K\"ahler manifold $M$ of dimension~$n \in \N$. Consider a singularity $p \in M$ of $X$ and denote by $X_k$
the first non-zero homogeneous component of the Taylor series of $X$ at $p$. If $X_k$ belongs to a
Painlev\'e-Guillot lattice, then $X_k$ has dicritical singularities at infinity. Furthermore, every regular
leaf of the foliation induced by $X_k$ on $\Delta_{\infty}$ must pass through one of these dicritical
singular points.}

\vspace{0.1cm}

Theorems~C and~C' will be proved in Section~7. The proofs are, indeed, very short though based on the material
developed in the preceding sections. The rest of the section will be taken up by a discussion of the main dynamical
issues associated to Halphen vector fields. The results presented there are definitely not new as they can
all be found in \cite{guillotIHES}
together with a large amount of additional information. Yet, the discussion conducted in Section~7 makes the
article self-contained in the sense that all properties of Halphen vector fields needed by Theorems~C
and~C' are proved here.
Besides, we have two additional motivation to carry out the mentioned analysis. Our first motivation
has to do with the well-known fact that, in the context of differential equations without movable critical points,
there is the phenomenon of ``natural boundaries''. When this phenomenon is regarded from the point of view of semi-complete
vector fields, it simply means
that the maximal domain of definition of the solution is bounded in $\C$. This is precisely what happens in the case of
Halphen vector fields. Whereas the statement of Theorems~A and~A' can straightforwardly be adapted to solutions of differential
equations that are meromorphic functions on $\C$, it is unclear that our method provides information in the case of solutions
having a ``natural boundary''. This leads us to work out the discussion of Halphen vector fields to show how the presence
of an associated fibration can be combined with ideas from Kleinian group theory to yield new insights in these cases
as well. A second motivation is that our discussion leads to a generalization of this picture in terms of representations
of ${\rm SL}\, (2,\C)$ in higher dimensions. Indeed, this paper ends with an Appendix containing
some questions for which we believe the ideas developed in the course
of this work may be useful. These questions include non-free representations of ${\rm SL}\, (2,\C)$.

\vspace{0.1cm}

\noindent {\bf Acknowledgements}: Both authors are very grateful to A. Guillot
for many discussions concerning several objects present in this work and, in
particular, for having explained to us many issues in \cite{guillotIHES}. Thanks are also due to the anonymous
referee for very valuable comments and suggestions. Discussions with X. Gomez-Mont concerning complete
real vector fields and Lorenz equations also improved our understanding of the material. Finally additional
comments by F. Cano and J.-P. Ramis have contributed to enlarge the field of possible future applications of this work.

The second author is partially supported by Funda\c{c}\~ao para a Ci\^encia e
Tecnologia (FCT) through CMUP, through the Post-Doc grant SFRH/BPD/34596/2007 and through
the project PTDC/MAT/103319/2008.


\section{Overview of methods, further results and background material}

This section contains a general description of the structure of the paper as well as some ``qualitative'' explanation
of our techniques. ``Quantitative'' information required by the corresponding proofs will be supplied in the subsequent
sections. Some complements to the theorems stated in the Introduction will also be provided along with
background material on semi-complete vector fields.

\subsection{Methods and results}

First, a general point should be made about the vector fields and/or differential equations considered in this work. This is due to the
fact that they are far from being ``generic''. For example, complete vector fields on $\C^n$ are very ``non-generic''
among polynomial/rational vector fields or among singular
holomorphic foliations on projective spaces. Indeed, the leaves of a foliation on $\C P(n)$ induced by a complete vector field
are quotients of $\C$ as Riemann surfaces and this, by itself, is already very ``non-generic''. Whereas they are ``non-generic'',
their interest can hardly be questioned since, for example, they constitute a natural Lie algebra for
the group of algebraic automorphisms of $\C^n$ \cite{lempert} and remain an object of intensive study as shown by the recent works
of A. Bustinduy, L. Giraldo, Brunella and others (cf. \cite{giraldo1}, \cite{giraldo2}, \cite{marco3}). Actually, when working with differential equations, we often encounter very special
(i.e. ``non-generic'') examples that turn out to play crucial roles in the theory. Apart from complete vector fields, our techniques
also apply to semi-complete ones, i.e. to those vector fields admitting ``maximal solutions'' defined on subsets of $\C$
(cf. below and Section~2.2 for further details). Halphen vector fields as studied in \cite{guillotIHES} satisfy this
condition and they will be revisited in Section~7. The importance of Halphen vector fields is undisputed
since they appear in Mathematical Physics, in the study of Ricci flow on
homogeneous spaces as well as in number theory through the celebrated functions
$P, \, Q , \, R$ of Ramanujan. Yet another class of special equations/vector fields that fits in the pattern of our theory consists of
those equations having meromorphic solutions defined on all of $\C$. In fact, modulo straightforward adaptations,
the statements of Theorems~A and~A' still apply to this class of equations.
Here the reader is also reminded that the class of differential equations
with meromorphic solutions includes the Painlev\'e~1, 2, 4 equations, the ``modified''
Painlev\'e~3 and~5 equations as well as many Chazy equations. As already mentioned, linear integral equations in
Gelfand-Levitan-Marcenko form also lead to equations having ``P-property'' that can similarly be treated,
see \cite{Ablowitz-1}, \cite{Ablowitz-2}.
Even in the case when the solutions possess a ``natural boundary'', and therefore are not defined on all of $\C$,
our methods can sometimes be used. An worked out example of this situation is provided by our discussion of Halphen
vector fields in Section~7. Another direction that is left
for future investigation concerns the connections of our work with the point of view developed by  X. Gomez-Mont and his collaborators,
see \cite{xavier1}, \cite{xavier2} concerning in particular the real Lorenz attractor for which a
``real'' variant of our method seems to yield new information.

Let us now begin to outline the structure of this paper.
Consider a polynomial vector field $X$ on $\C^n$ and denote by $\cald$ the associated foliation induced on $\C P(n)$. Let
$X_d$ stand for the top-degree homogeneous component of $X$ (having degree~$d \geq 2$) and suppose that $X_d$ is not a multiple
of the radial vector field. Under this assumption, the foliation $\cald$ leaves the hyperplane at infinity $\Delta_{\infty}
= \C P(n) \setminus \C^n$ invariant. In addition, and modulo a minor remark discussed in Section~3, this foliation coincides with
the foliation induced on $\Delta_{\infty}$ by $X_d$. Alternatively, and modulo the natural identification
$\Delta_{\infty} \simeq \C P(n-1)$,
the foliation in question is simply given by the direction of $X_d$ projected on $\C P(n-1)$ viewed as the space of radial lines
in $\C^n$ (note that $X_d$ is homogeneous and it is not a multiple of the radial vector field). Yet, a third way to see this
foliation consists of noting that it coincides with the foliation induced on the
exceptional divisor $\Delta_0 \simeq \C P(n-1)$ by the punctual blow-up of $X_d$ at the origin.
The foliation associated to $X_d$ on $\C P (n)$ is going to be denoted by $\fol$ and its restriction to $\Delta_{\infty}$
by $\fol_{\infty}$. If $L_{\infty}$ is a leaf of $\fol_{\infty}$ then the ``cone over $L_{\infty}$'' is invariant by $\fol$.

Fundamentally, our method relies on estimating the ``speed'' of the vector field $X$ near $\Delta_{\infty}$. This is done
in two steps. The first step consists of eliminating the unbounded factor of $X$ over $\Delta_{\infty}$
so as to obtain a ``local regular vector field'' about every regular point $p \in \Delta_{\infty}$ of $\fol_{\infty}$.
However, it turns out that these locally defined
vector fields depend to some extent on the choice of local coordinates so that they do not
patch together in a ``foliated'' global vector field. Nonetheless,
two local representatives obtained through overlapping coordinates differ only by a multiplicative
constant. This means that this collection of local vector fields defines
a global affine structure (induced by $X_d$ or by $X$) on every leaf of
$\fol_{\infty}$. In other words, the foliation $\fol_{\infty}$ can be equipped with a {\it global foliated
affine structure}\, though this affine structure does not give rise to a ``global foliated vector field''. Another
version of this affine structure already appeared in \cite{guillotFourier} as well as in a previous work of
the first author \cite{julioMex}
under the name of ``renormalized time-form''. It also plays an important role in \cite{guillotreb}.
In our context, the interest of the mentioned affine structure arises from the fact that
it lends itself well to provide estimates for the flow of $X$ as long as accurate
estimates for the ``distance'' from the orbit in question to $\Delta_{\infty}$ are available.

Here comes the second ingredient of our construction, namely a quantitative measure of ``the rate of approximation'' of a leaf
of $\fol$ to $\Delta_{\infty}$. Because $\Delta_{\infty} \subset \C P(n)$ and the Fubini-Study metric on $\C P(n)$ has positive
curvature, it is well-known that complex submanifolds always bend themselves towards $\Delta_{\infty}$, cf. for example
\cite{langevin}. In our case, this implies that the distance
(relative to the Fubini-Study metric) of a leaf $L$ of $\fol$ to $\Delta_{\infty}$ can never reach a local minimum unless this
minimum is {\it zero}. Our mentioned second ingredient is reminiscent from this remark. Actually, we shall
use the euclidean metric on suitably chosen affine coordinates, as opposed to the globally defined Fubini-Study metric.
The advantage of our choice lies in the fact that
the euclidean metric is better adapted to work with the above mentioned affine structure.
Besides, by exploiting the fact that the submanifolds in questions are actual leaves of a foliation,
a quantitative version of the rate of approximation of a leaf to $\Delta_{\infty}$ is derived. The phenomenon goes essentially as
follows. At each regular point $p$ of a leaf $L$ of $\fol$ there is the steepest descent
direction of $L$ towards $\Delta_{\infty}$, namely the negative of the gradient of the distance function restricted
to $L$. This yields a singular real
one-dimensional oriented foliation $\calh$ on $L$. Furthermore the conformal structure of $L$ is such that the foliation
$\calh^{\perp}$ orthogonal to $\calh$ is constituted by level curves for the mentioned distance function. Roughly speaking,
an exponential rate of approximation for $L$ to $\Delta_{\infty}$ over the trajectories of $\calh$ can then be obtained.
This estimate combines to
the ``uniform'' estimates related to the foliated affine structure to produce accurate estimates for the time
taken by the flow of $X$ over trajectories of $\calh$. The discussion actually shows that the time taken by $X$ to cover an entire
(infinite) trajectory is {\it finite} provided that the trajectory remains away from the singularities of $\fol$
lying in $\Delta_{\infty}$. This results is then sharpened in Section~5 by allowing the trajectory to accumulate on (simple) singular
points and still obtaining an analogous estimate. In particular, there is only one special type of ``simple'' singularity that
may yield an ``endpoint'' for the trajectories of $\calh$ and, in this case, this will be an intersection point between the leaf
$L$ and the hyperplane $\Delta_{\infty}$: the corresponding trajectory of $\calh$ should then be thought of as being
``finite''. Finally, switching back and forward between estimates involving $X_d$ and estimates involving
$X$ is not hard since $X$ is close to $X_d$ near $\Delta_{\infty}$.

The material mentioned above is covered in Sections~3,~4 and~5. Armed with these results we begin in Section~6 to prove
the theorems stated in the Introduction. Theorems~A and~A' are very natural. Since $X$ is complete
its integral over a trajectory of $\calh$ cannot converge. Besides this trajectory can never ``reach'' $\Delta_{\infty}$ since
$X$ is complete on $\C^n$. This observation tends to clash
with our previous estimate asserting convergence of the integral in question as long as
the corresponding trajectory remains away from the singularities of $\fol$ (or $\cald$) lying in $\Delta_{\infty}$.
The apparent contradiction is then explained by the fact that the flow of $X$ spends all but a finite amount of its existence
in arbitrarily small neighborhoods of the singular set.
The proof of Theorem~A' goes along similar lines. In fact, the estimates
carried over trajectories of $\calh$ remain valid for every oriented foliation $\calh^{\theta}$ forming an angle $\theta \in
(-\pi/2 , \pi/2)$ with $\calh$. Once again the foliations $\calh^{\theta}$ are well-defined since the leaves of
$\fol, \, \cald$ are endowed with a conformal structure. Modulo fixing a base point and using the obvious identifications,
the union of the corresponding trajectories span an unbounded region of $\C$ viewed as the domain
of definition of the solution in question. The area of this region is ``large'' since it is comparable
to the area of ``large discs''.

The above mentioned theorems are clearly sharp since the complement of a compact part of a solution cannot
``globally'' confine at singular points unless the vector field is completely integrable. Indeed, owing to
Remmert-Stein theorem, this type of confinement would mean that the solution is contained in a rational curve
and, in turn, if ``most'' solutions are contained in rational curves then the
underline foliations must have all its leaves contained in rational curves, i.e.
it must be a {\it completely integrable rational foliation}. There are however simple
complete polynomial vector fields, such as
$y \partial /\partial y + xy [ x \partial /\partial x - y \partial /\partial y ]$
on $\C^2$, whose orbits accumulate on all of the ``line at infinity''. These orbits
are therefore non-compact.

In view of Theorems~A and~A', it is natural to imagine that the singular set of $\cald$ contains
significant amount of information about
the global geometry of complete polynomial vector fields. Theorem~B is a contribution to the study of these
vector fields as well as a test for the extent to which their global dynamics can be determined by the structure of their
singularities. From an abstract point of view, this may be seen as a first attempt to understand the remarkable
effectiveness of the so-called Painlev\'e's test in differential equations. In fact, since Theorems~A,~A' can be adapted to the context
of differential equations having meromorphic solutions, we may expect that the local information concentrated in the singular points
is likely to strongly influence the global behavior of the solution itself.
To substantiate this ``principle'', the idea will be
to consider complete vector fields having ``simple singularities'' and to check what can then be said about the
vector field in question. From this point of view, Theorem~B is totally satisfactory since the dynamics of the corresponding vector
field is fully determined.

The proof of Theorem~B is arguably the most elaborate application of our techniques.
Let us briefly describe its main ingredients. The central difficulty is to guarantee the existence of a
``dicritical singularity'' for $\cald$ in $\Delta_{\infty}$, i.e. a linearizable singularity all of whose
eigenvalues belong to $\R_+$. The existence of this type of singularity implies, in particular, that the generic
orbit of the vector field $X$ is of type $\C^{\ast}$ in the sense of \cite{suzuki}
and several additional properties follow at once. To ensure the existence of this singularity is, however, a subtle question that
can be approached as follows. First, let $X$ be replaced by its top-degree homogeneous component $X_d$
along with its associated foliation denoted by $\fol$. The property of having a dicritical singularity at
$\Delta_{\infty}$ is common to $\cald$ and $\fol$ so that
it is more convenient to work with homogeneous vector fields. Nonetheless, when replacing $X$ by $X_d$, we need
to cope with the fact that $X_d$ is no longer complete but only {\it semi-complete}. In other words, every solution
$\phi : V \subset \C \rightarrow U$ of $X_d$ on $\C^n$ is such that whenever a sequence $\{ T_i \}
\subset V$ converges to a point $\hat{T}$ in the boundary of $V$ the corresponding sequence $\phi (T_i)$ leaves every compact
set in $\C^n$. Being only semi-complete $X_d$ ``may reach the infinity in finite time'' and this gives rise to further difficulties.

Another difficulty arising from the difference between complete and semi-complete vector fields is the fact that the leaves
of the foliation associated to a semi-complete vector field may be hyperbolic Riemann surfaces, as it happens in the case
of Halphen vector fields, cf. Section~7. However, in the case of a foliation associated to the top-degree
homogeneous component of a complete vector field, it can be proved that the
corresponding leaves are still quotients of $\C$. This is done by resorting to a result due to Brunella
concerning the pluri-subharmonic variation of the foliated Poincar\'e metric,
cf. \cite{marco2}. The solutions of $X_d$ will therefore be meromorphic functions defined on $\C$ or in
$\C$ minus one or two points. Next, we bring in our
results involving the time taken by $X_d$ to cover trajectories of $\calh$ in the singular context (here it is used the main
result of Section~5, namely Theorem~\ref{maintheo}). Theorem~\ref{maintheo} immediately implies that the solutions
cannot be meromorphic on all of $\C$ and, by exploiting additional properties of the foliations $\calh, \, \calh^{\perp}$,
a contradiction ensuring the existence of the desired dicritical singularity is finally reached.

Let us now make some comments about the assumption that the singularities of $\cald$ lying in $\Delta_{\infty}$
are simple in the sense described in the Introduction.
This assumption does not immediately simplify the problem since there may exist
codimension~1 saddle-node whose local analysis is already fairly complicated. Also,
the statement of Theorem~B may be extended to encompass
more general singularities belonging to the class of ``absolutely isolated singularities'', cf. \cite{canoetc}. While
we shall not seek to accurately establish any of these extensions, at the very end of Section~5 the reader will find
some information on the structure of more degenerate singularities for which our methods may still work.
It is also interesting to observe that our techniques apply equally well to rational vector fields and not only to polynomial
ones. In practice, passing from polynomial to rational vector fields amounts to changing the
the divisor of poles of the vector field in question. The divisor of poles of a rational vector field
may or may not include $\Delta_{\infty}$ and its analysis
leads to numerous additional possibilities whose understanding may partially be facilitated by our ideas.
In particular, several Painlev\'e equations fall in this class of problems.

As mentioned A. Bustinduy, L. Giraldo and their collaborators have been investigating properties of the solutions
of complete vector fields through various methods such as the theory of Nevanlinna and Andersen-Lempert theories,
see \cite{giraldo1}, \cite{giraldo2}, \cite{lempert} and so on.
Similarly, if taken into account a classical result due to Forstneric \cite{forst}, our method is likely
to find some applications in the theory of holomorphic differential equations blowing-up in finite real time.
This should lead to some progress in questions similar to those treated by Fornaess and Grellier in \cite{fornaess}
which itself connects with previous works by a number of authors including possible
applications in the spirit of \cite{fornaessbuzzard}.

Finally, and as already mentioned, the beautiful results obtained by A. Guillot in \cite{guillotFourier}, \cite{guillotIHES},
provide a natural motivation to try to apply our techniques
to quadratic semi-complete vector fields as those considered by him. These vector fields are referred to as
belonging to the Painlev\'e-Guillot lattices. The tools developed in the course of this work will enable us to show that a
vector field in the Painlev\'e-Guillot lattice failing to have a dicritical singularity at infinity must have
leaves that are hyperbolic Riemann surfaces. This fact, in turn, will quickly lead us
to Theorem~C by resorting again to Brunella's result on the variation of the Poincar\'e metric, see \cite{marco2}.
A point to be made here has to do with the lack of explicit examples of vector fields in Painlev\'e-Guillot lattice
having no dicritical singularity at infinity. We believe this example does not exist in dimension~$3$ and it is unclear
to us whether or not it does in higher dimensions. Yet, the argument used in the proof of Theorem~C allows us to conclude
that the foliation induced on $\C^3$ associated to the vector field in question not only has dicritical singularity
at infinity, but also satisfies the following conclusions:
\begin{itemize}
  \item The foliation induced on the hyperplane at infinity is such that all its leaves have to go through
  a dicritical singularity lying in the hyperplane in question (Theorem~C').

  \item The restriction of the vector field $X$ to a leaf $L$ of its associated foliation is either complete
  or conjugate to the vector field $x^2 \partial /\partial x$ on all of $\C$. In the second case, the blow-up of
  $X$ at the origin has a dicritical singularity on the exceptional divisor and, moreover, the saturated to
  these dicritical singularities define an open set where the vector field $X$ admits non-constant first integrals.
\end{itemize}
Concerning the proofs of the above claims, the reader is referred to Remark~\ref{completingtheoremC}.

Another minor point that can be mentioned in a similar direction is that the statement of Theorem~C also holds for
vector fields slightly more general than those belonging to Painlev\'e-Guillot lattices. Namely, in our case, the condition
used by Guillot may be relaxed to allow the eigenvalue
associated to the direction transverse to the exceptional divisor to vanish.
Another relatively minor point has to do with a slight relaxation of the assumption made in Theorems~C and~C' or,
more precisely, with the assumption that $X^k$ is ``quadratic''. Recall that in the statement of Theorems~C and~C'
the vector field $X$ has the form $X=X_k + \cdots$ where $X_k$ is the first non-zero homogeneous component of the
Taylor series of $X$ at the origin. Now, note that $X^k$ may have a codimension~$1$ zero-set, in which case we may
set $X_k = P. Y^{\rm cd2}$ where $P$ is a homogeneous polynomial and $Y^{\rm dc2}$ a homogeneous vector field
whose zero-set has codimension at least~$2$. Because the foliations induced on the corresponding projective space
by $X^k$ and by $Y^{\rm dc2}$ coincide, the statements of the mentioned theorems remain valid for vector fields
$X$ whose first non-zero homogeneous component is a {\it multiple}\, of a semi-complete vector field
lying in Painlev\'e-Guillot lattice. In fact, in the Painlev\'e-Guillot lattice there are (semi-complete)
vector fields admitting non-constant holomorphic first integrals. If $Y$ stands for one of these vector fields
and $P$ stands for a holomorphic first integral of $Y$, then the statement of Theorem~C (resp.
Theorem~C') also applies to vector fields
$X$ whose first non-zero homogeneous component has the form $PY$, for example. The reader will also notice that
a similar extension concerning hyperbolic Halphen vector fields is void in the sense that the vector fields
in question have only constant holomorphic first integrals.

Still concerning vector fields in Painlev\'e-Guillot lattices, it follows that
Halphen vector fields are again special in the sense that they
{\it do have}\, dicritical singularities at infinity and still the leaves of their associated foliation may be hyperbolic Riemann
surfaces. Whereas these results, and many others, are due to A. Guillot and appear in \cite{guillotIHES}, we found it was worth
re-obtaining them by following our general point of view. This discussion takes up most of the last section of this paper.
In particular, it involves some considerations about convergence of Poincar\'e series that differ from their classical theory.

Another motivation for us to revisit Guillot's work on Halphen vector fields is to pave the way for other possible
applications of our techniques, some of them indicated in the Appendix. Namely, it consists of classifying
the first homogeneous components at a singular point of a globally defined holomorphic vector field on a compact K\"ahler threefold.
As it will be explained later, this classification must be identical to the classification of
the top degree homogeneous components of complete polynomial vector fields on $\C^3$, cf. Appendix.

Finally, let us also point out a curious remark involving Theorem~C and, more generally, semi-complete
homogeneous vector fields. In fact, singularities of homogeneous vector field on $\C^3$ (or $\C^n$)
possess a natural meromorphic ``dual'' represented by a neighborhood of the (hyper-) plane at
infinity (even though this neighborhood cannot be collapsed to a singular point).
More precisely, as detailed in Section~3, the blow-up at the origin of a homogeneous vector field
leads to an exceptional divisor sharing a natural ``duality'' with the divisor obtained at infinity of the
corresponding projective space. We shall refer to a neighborhood of the hyperplane at infinity as the
{\it dual singularity}\, (assuming that the singularity of a homogeneous polynomial vector field is
implicitly fixed). By virtue of Theorem~C and of the global realization of Halphen vector fields constructed in
\cite{guillotIHES}, it is natural to ask whether the dual singularity of a
hyperbolic Halphen vector field can be realized by a complete meromorphic vector field on a complex
$3$-manifold not necessarily compact (where by complete meromorphic vector field it is meant a meromorphic vector field that is
complete in the complement of its pole locus). The answer to this question turns out to be negative as it follows from the discussion
in Section~7.

\subsection{A brief review of semi-complete vector fields and additional background material}

Most of the discussion below concerns basic properties of semi-complete
vector fields that will often be encountered in the course of this paper. Some general subtle notions involving
singular foliations and their corresponding leaves, as needed by Brunella's theorem \cite{marco2},
will also quickly be reviewed.

First consider a $1$-dimensional singular holomorphic foliation $\cald$ defined on a compact manifold $M$ and denote by
${\rm Sing}\, (\cald)$ its singular set. Thus ${\rm Sing}\, (\cald)$ is an analytic set of $M$ having
codimension at least~$2$. Since Brunella's theorem \cite{marco2} will be used in Section~6, we shall adopt
in this work the definition of ``leaf'' for $\cald$ that is required for his theorem to hold. The subtle
point in this notion of ``leaf'' lies in
the fact that ``leaves'' are sometimes allowed to contain points from ${\rm Sing}\, (\cald)$. Since the definition
of ``leaf'' {\it for the restriction of $\cald$ to $M \setminus {\rm Sing}\, (\cald)$}\, is clear,
we can work on a local setting
and consider the $n$-dimensional polydisc ${\bf D}^n$ about the origin. This polydisc comes equipped with the trivial
fibration ${\bf D}^n = {\bf D}^{n-1} \times {\bf D} \rightarrow {\bf
D}^{n-1}$. A meromorphic map $f : {\bf D}^n \rightarrow M$ is said to
be a {\it foliated meromorphic immersion}\, if the indeterminacy set $I (f)$ of $f$ intersects each
vertical fiber of ${\bf D}^n$ over a discrete set and if $f$ satisfies the following additional
conditions:
\begin{itemize}

\item $f$ is an immersion on the complement of $I (f)$.

\item In the complement of $I (f)$, $f$ takes vertical fibers to leaves of $\cald$ (more
generally to the leaves the foliation under consideration).
\end{itemize}

Consider now a regular point $p$ in $M \setminus {\rm Sing}\, (\cald)$
and let $L_p'$ denote the leaf through $p$ of the (regular) foliation obtained by
restriction of $\cald$ to $M \setminus {\rm Sing}\, (\cald)$. A closed subset $K \subset L_p'$ is called a {\it vanishing
end}\, of $L_p'$ if all the conditions below are satisfied:

\noindent $\bullet$ $K$ is isomorphic to the punctured disc and the
holonomy of the restriction of $\cald$ to $M \setminus {\rm Sing}\,
(\cald)$ corresponding to the cycle $\partial K$ has finite order
$k$.

\noindent $\bullet$ There is a foliated meromorphic immersion $f : {\bf D}^n \rightarrow M$ such that

\hspace{0.4cm} $\imath \, )$ $I (f) \cap (\{ 0\} \times {\bf D}) =  0 \in {\bf D} \subset \C$, where
$\{ 0 \}$ stands for the origin of ${\bf D}^{n-1} \subset \C^{n-1}$.

\hspace{0.4cm} $\imath \imath \, )$ The image of $f$ restricted to $(\{ 0\} \times {\bf D})$
is the interior of $K$.  Furthermore $f: (\{ 0\} \times {\bf D})
\rightarrow {\rm Int} \, (K)$ is a regular covering of degree $k$,
where ${\rm Int} \, (K)$ stands for the interior of $K$.

\vspace{0.1cm}

The general definition of ``leaf'' for a foliation $\cald$ on $M$ as above goes as follows.
Consider a regular point $p \in M \setminus {\rm Sing}\, (\cald)$ along with the leaf
$L_p'$ through $p$ of the (regular) foliation obtained by restricting $\cald$ to
$M \setminus {\rm Sing}\, (\cald)$. If $L_p'$ possesses no vanishing ends, then the leaf $L_p$ of $\cald$
containing $p$ is exactly $L_p'$. Otherwise this leaf $L_p$ will consist of
$L_p'$ with the ends of the vanishing ends added to it where the operation of adding an
end to $L_p'$ should be understood in the sense of orbifolds: the
multiplicity of the added point will precisely be the order $k$ of
the holonomy relative to $\partial K$. These orbifolds can then be turned into Riemann surfaces by standard normalization.
An immediate consequence of the preceding construction is as follows.

\begin{coro}
\label{addingpointstoleaves}
Let $\cald, \, M$ and ${\rm Sing}\, (\cald)$ be as above. Fixed $p \in M \setminus {\rm Sing}\, (\cald)$, let
$L_p$ (resp. $L_p'$) denote the leaf of $\cald$ through $p$ (resp. the leaf of the restriction of $\cald$ to
$M \setminus {\rm Sing}\, (\cald)$ through $p$). Then $L_p' \subset L_p$ and $L_p \setminus L_p'$ is a discrete
set.\qed
\end{coro}

With the above definition of leaf in place, the main result of Brunella \cite{marco2} which will
find applications later in this work can be stated as follows: if $\cald$ is a singular holomorphic foliation
defined on a compact K\"ahler manifold $M$, then the Poinca\'e metric along the leaves of $\cald$ has
a pluri-subharmonic variation. In particular, unless no leaf of the foliation in question is hyperbolic,
the set of non-hyperbolic leaves is ``small'' in the sense that it is a pluri-polar set.

After this general considerations about holomorphic foliations, we recall that a meromorphic vector field
$X$ defined on an open set $U$ of a (possibly open) manifold $M$ naturally defines a singular holomorphic
foliation on $U$. In particular, if $X$ is a meromorphic vector field defined on a compact manifold
$M$, then it induces a singular holomorphic foliation $\cald$ as above on all of $M$.

Our next step is to remind the reader the accurate definition of semi-complete vector fields.

\begin{defnc}
A holomorphic vector field~$X$ on a complex manifold~$M$ is called semi-complete
if for every~$p\in M$ there exists a connected domain~$U_p\subset\C$
with~$0\in U_p$ and a map~$\phi_p:U_p\to M$ such that:
\begin{itemize}
\item $\phi_p(0)=p$ and $d \phi_p(t)/d t|_{t=t_0}=X(\phi_p(t_0))$.
\item For every sequence~$\{t_i\}\subset U_p$ such
that~$\lim_{i\rightarrow\infty}t_i\in\partial U_p$ the sequence~$\{\phi_p(t_i)\}$
escapes from every compact subset of~$M$.
\end{itemize}
A meromorphic vector field~$X$ on a complex manifold~$M$ is
semi-complete if its restriction to the open set where~$X$ is holomorphic is semi-complete in the
above mentioned sense.
\end{defnc}
The reader will note that the standard theorem about existence of local solutions for ordinary differential
equations ensures that a map $\phi : U_p \rightarrow M$ satisfying the first condition in the preceding definition
always exists. It is therefore the second condition that makes the definition non-trivial. This second condition is a natural
generalization of the analogous phenomenon that always happens for real-time ordinary differential equations
when the time approaches one of the endpoints of its maximal domain of definition. In this sense, semi-complete
vector fields are those whose solutions admit a maximal domain of definition in $\C$. It follows at once from
this definition that vector fields whose solutions are meromorphic functions defined on $\C$ are automatically
semi-complete. Besides, the solutions of semi-complete vector fields may actually be defined on bounded
domains of $\C$ in an {\it essential way}, \cite{guillotIHES} or Section~7. This {\it essential boundary}
is thus a continuum of singularities for the solution of the associated differential equation. Besides, this
boundary may move with the initial condition (or rather with the leaf of the underlying foliation). Thus
this class of vector fields/equations is more general than those possessing Painlev\'e property, see for example
\cite{PPainleve}.

The following simple lemma already conveys some useful information concerning semi-complete vector fields.

\begin{lema}\label{Adim1}
A semicomplete meromorphic vector field on a curve is necessarily holomorphic.
\end{lema}
\begin{proof}
Let~$X$ be a meromorphic vector field on the curve~$\Sigma$ and suppose that~$X$ has a pole at~$p\in\Sigma$.
The vector field is locally given, in a neighborhood of~$p$, by~$z^{-q}f(z)\partial/\partial z$ for some~$q>0$
and a holomorphic and non-vanishing function~$f$. There is  a coordinate~$w$ where the vector field has the
form~$w^{-p}\partial/\partial w$. The solution with initial condition~$w_0\neq 0$ is multivalued and given by~$\sqrt[p+1]{(1+p)t+w_0^{p+1}}$. Hence, there is no neighborhood of~$p$ where the vector field is semicomplete.
\end{proof}

This implies that a semicomplete meromorphic vector field on a compact curve is globally holomorphic and thus,
unless it is identically zero, the curve must be either elliptic or rational.

More generally, consider a meromorphic vector field $X$ along with its associated singular holomorphic foliation $\cald$.
The singular set of $\cald$ (resp. $X$) will be denoted by ${\rm Sing}\, (\cald)$ (resp. ${\rm Sing}\, (X)$).
Unlike ${\rm Sing}\, (\cald)$, note that ${\rm Sing}\, (X)$ contains the divisor of zero and poles of $X$ so that
it may have codimension~$1$ components. Thus, given a point $p \in M$ that is regular for $X$, consider the leaf
$L_p$ of $\cald$ through $p$. Inside $L_p$, there are two open sets that may naturally be considered, namely:
\begin{itemize}
  \item The set $V_{\rm reg} \subset L_p$ identified to the leaf $L_p'$ containing $p$ of the restriction of $\cald$
  to the complement of ${\rm Sing}\, (\cald)$.
  \item The set $W_X$ consisting of those points in $L_p$ at which the vector field $X$ is holomorphic and
  different from zero.
\end{itemize}
Clearly $W_X \subseteq V_{\rm reg} \subseteq L_p$. On $W_X$, consider the {\it time-form}\,
namely the $1$-form $dT$ defined by letting $d_qT . X(q) =1$. The time-form is holomorphic and non-zero $W_X$.
It has a meromorphic extension to $V_{\rm reg}$ and, {\it a priori}, may have essential singularities at the discrete set
$L_p \setminus V_{\rm reg}$.

At this point two additional remarks can be made concerning the definition of semi-complete vector fields.
The first one is that $X$ is semi-complete if and only if for every point $p$ regular for $X$, the
natural map $\phi_p : U_p \rightarrow W_X$ is proper and hence is a covering since it is
clear a local diffeomorphism. Also, by using exploiting this condition, it is easy to see that
for every embedded (one-to-one) path $c : [0,1] \rightarrow W_X \subset L_p$, the integral
$$
\int_c dT
$$
is different from {\it zero}\, provided that $X$ is semi-complete,
cf. \cite{julio}.

At this point, Lemma~\ref{Adim1} can be improved as follows.

\begin{lema}\label{Adim2}
Consider a meromorphic vector field along with it associated singular foliations $\cald$.
Fix a leaf $L$ of $\cald$ not contained in the divisor of zeros or poles of $X$
and suppose that $X$ is semi-complete. Then the restriction $X_{\vert_L}$ of $X$
to $L$ is holomorphic on all of $L$. Besides, if $p \in L$ is a singular point of $X_{\vert_L}$, then
the second jet of $X_{\vert_L}$ at $p$ is different from zero.
\end{lema}

\begin{proof}
Let us first show that $X_{\vert_L}$ is holomorphic. In view of Lemma~\ref{Adim1}, it suffices to show
that $X_{\vert_L}$ cannot have an essential singularity at a point $p \in L \setminus V_{\rm reg}$. Assuming
for a contradiction the existence of a point $p \in L \setminus V_{\rm reg}$ at which $X$ has an essential
singularity, note that $p$ is also an essential singularity for the time-form $dT$ induced on $L$ by $X$.
Now fix a local disc $B \subset L$ about $p$ and consider the map
${\rm Dev}\, : \widetilde{B \setminus \{p\}} \longrightarrow \C$ defined by
$$
{\rm Dev} \, (x) = \int_{x_0}^x dT \; ,
$$
where $\widetilde{B \setminus \{p\}}$ stands for the universal covering of the punctured disc
$B \setminus \{p\}$ and where $x_0$ is a fixed base point. The semi-complete nature of $X$ implies that
the map ${\rm Dev}$ must be one-to-one. This is however impossible as it follows from a simple application
of Picard theorem, cf. \cite{julioMex}.

It remains only to check that the second jet of $X_{\vert_L}$ cannot vanish at a (necessarily isolated)
singular point. The proof is a simple variant of the argument given in the proof
of Lemma~\ref{Adim1}, details are left to the reader.
\end{proof}

Lemma~\ref{Adim2} has the following useful corollary.

\begin{coro}
\label{overaleafC}
Suppose that $X$ is a semi-complete vector field defined on the complement
of a discrete set $\aleph \subset \C$. Then $X$ is holomorphic on all of $\C$ and, in fact,
extends to a holomorphic vector field globally defined on $\C P(1)$.\qed
\end{coro}

Semi-complete vector fields have additional useful global properties. For example, the semi-complete
nature of a vector field is invariant by birational transformations, an invariance property
that is not verified for {\it complete vector fields}. In this sense, from a point of view of birational
geometry, the notion of semi-complete vector fields is more natural than the notion of complete ones,
cf. \cite{guillotreb}.

Another less immediate, though still elementary, global property originally established in \cite{ghreb}
asserts that the space of semi-complete holomorphic vector fields is closed under the topology of
uniform convergence. More precisely, suppose that $\{ X_n \}$ is a sequence of holomorphic vector fields
defined on some (possibly open) manifold $M$ converging to a (holomorphic) vector field $X$ on $M$ for
the topology of uniform convergence on compact subsets. Under this assumption, the limit vector field
$X$ must be semi-complete provided that $X_n$ is a semi-complete vector field for every $n \in \N$.

From the preceding results, the following useful fact can be derived.

\begin{lema}
\label{Adim3}
Suppose that $X$ is a semi-complete polynomial vector field on $\C^n$ having degree~$d$. If $X_d$ denotes
the homogeneous component of degree~$d$ of $X$, then $X_d$ is itself semi-complete on all of $\C^n$.
In particular, if $X_d$ is a non-constant multiple $fR$ of the radial vector field
$R = x \partial /\partial x + y \partial /\partial y + z \partial /\partial z$, then the degree of
the (homogeneous) form $f$ must equal~$1$.
\end{lema}

\begin{proof}
To show that $X_d$ is itself a homogeneous semi-complete vector field, consider homotheties $\Lambda^k$ of
$\C^n$ having the form $\Lambda^k (x_1, \ldots ,x_n) = (k x_1 ,\ldots , k x_n)$, for $k \in \N^{\ast}$.
Clearly for every $k \in \N^{\ast}$, the vector field $(\Lambda^k)^{\ast} X$ is semi-complete
on $\C^n$. Since a constant multiple of a semi-complete vector field still is semi-complete, it follows
that the vector fields $Y_k = k^{1-d} (\Lambda^k)^{\ast} X$ are semi-complete on $\C^n$
for every $k \in \N$. By sending $k \rightarrow \infty$, it becomes clear that the sequence of vector fields
$Y_k$ converge uniformly to $X_d$ on compact parts of $\C^n$. It then follows that $X_d$ is semi-complete
as desired.

For the second part of the statement, note that every radial line through the origin is left invariant
by the radial vector field $R$ and hence by $X_d$. By restricting $X_d$ to a ``generic'' line as before,
we obtain a $1$-dimensional semi-complete vector field having an isolated singular point at the origin.
However, owing to Lemma~\ref{Adim2}, the order of this singular point cannot exceed~$2$. Hence
the degree of the non-constant homogeneous polynomial $f$ cannot exceed~$1$ which completes the proof
of the lemma.
\end{proof}

\section{Homogeneous vector fields and their foliations}\label{sechom}

Unless otherwise stated, throughout this paper all homogeneous vector fields have degree
$d \geq 2$ and are supposed not to be a multiple of the radial vector field. In this section
we shall work in dimension~$3$ rather than in $\C^n$ just to abridge notations since all arguments
presented in the sequel can be carried over word-by-word to
higher dimensions.

Consider a homogeneous polynomial vector field $X$ of degree
$d \geq 2$ defined on $\C^3$. Since $X$ is homogeneous, its associated foliation $\fol$
is invariant by homotheties of the form $(x, y, z) \mapsto (\lambda x,\lambda y ,\lambda z)$,
$\lambda \in \C^{\ast}$ and, therefore, also induces a foliation on $\CP (2)$. An alternative
way to look at this situation consists of punctually blowing-up $X$ at the origin of $\C^3$.
We denote by $\wdc^3$ the corresponding blow-up of $\C^3$ and by $\Delta_0 =\pi^{-1} (0)$
the resulting exceptional divisor, where $\pi: \wdc^3 \mapsto \C^3$ represents the corresponding
blow-up projection. The transform (blow-up) $\tXX$ (resp. $\tilf$) of $X$ (resp.
$\fol$) vanishes identically over $\Delta_0$ (resp. leaves $\Delta_0$ invariant), as it
follows from the fact that the degree of $X$ is strictly greater than~$1$ (resp.
that $X$ is not a multiple of the radial vector field).

Recalling also that $\wdc^3$ can be viewed as a line bundle over $\Delta_0 =\pi^{-1}
(0)$, let $\calp_0$ denote the bundle projection $\calp_0: \wdc^3
\rightarrow \Delta_0$. This line bundle can be compactified into a
projective line bundle by adding the ``section at infinity''
$\Delta_{\infty}$. Denoting by $M$ the total space of the resulting projective
line bundle, it follows that $M$ is equipped with two bundle projections
$\calp_0, \, \calp_{\infty}$ realizing it as a projective bundle
respectively over $\Delta_0, \, \Delta_{\infty}$. The manifold $M$
is also isomorphic to the blow-up of $\CP (3)$ at the origin.
The vector field $\tXX$ can meromorphically be extended to $M$ so
that it induces a holomorphic foliation, still
denoted by $\tilf$, on all of $M$. Besides $\tilf$ leaves both
$\Delta_0, \, \Delta_{\infty}$ invariant since $X$ is homogeneous and it is not
a multiple of the radial vector field. The foliation induced on
$\Delta_0$ (resp. $\Delta_{\infty}$) by restriction of $\tilf$ is
going to be denoted by $\tilf_0$ (resp. $\tilf_{\infty}$). Because $\tilf$ comes from a
homogeneous vector field, these foliations coincide with the restrictions of $\tilf$
to $\Delta_0, \, \Delta_{\infty}$. As to the vector field $\tXX$,
its pole divisor coincides with $\Delta_{\infty}$ and it has order $d-1>0$. The zero divisor of
$\tXX$ is the union of $\Delta_0$ (a component of order $d-1 >0$) with the transform of the
zero divisor of $X$.

Naturally the singular set of $\tilf$
has codimension at least~$2$. Besides this singular set is saturated (i.e. invariant)
by the fibers of $\calp_0$ (resp. $\calp_{\infty}$) due to the
invariance of $\fol$ by homotheties of the form $(x, y, z) \mapsto (\lambda x,
\lambda y, \lambda z)$, $\lambda \in \C^{\ast}$. In particular, the foliations
$\tilf_0, \, \tilf_{\infty}$ automatically have singular sets of
codimension at least~$2$ {\it inside} $\Delta_0, \, \Delta_{\infty}$ (in other words the
intersection of the singular set of $\tilf$ with $\tilf_0, \, \tilf_{\infty}$ yields a set
of codimension at least~$2$ inside $\Delta_0, \, \Delta_{\infty}$).

Consider a non-algebraic leaf $L$ of $\tilf$ not contained in
$\Delta_0 \cup \Delta_{\infty}$. The projection of $L$ onto
$\Delta_0$ (resp. $\Delta_{\infty}$), $\calp_0 (L) =L_0$ (resp.
$\calp_{\infty} (L) =L_{\infty}$), is clearly a leaf of $\tilf_0$
(resp. $\tilf_{\infty}$) since the initial vector field $X$ is homogeneous. Furthermore one
immediately checks that the restriction of $\calp_0$ (resp.
$\calp_{\infty}$) to $L$ realizes $L$ as an Abelian covering of
$L_0$ (resp. $L_{\infty}$). It then follows that the
non-compact leaves $L, \, L_0, \, L_{\infty}$ have all the
same nature: either they are all covered by $\C$ or they are all
covered by the unit disc~$D$. Furthermore $L_0, \, L_{\infty}$ are isomorphic as Riemann surfaces while
$L$ is an Abelian covering of $L_0, \, L_{\infty}$.

In this way, we may focus on the behavior of $\tXX$ near its pole divisor
$\Delta_{\infty}$ or near $\Delta_0$ according to our convenience.
Next, consider a leaf $L_{\infty}$ of $\tilf_{\infty}$. By {\it the cone
over $L_{\infty}$}\, it is meant the $2$-dimensional immersed
singular surface $\calp_{\infty}^{-1} (L_{\infty})$ which is
invariant by $\tilf$. In other words, if $\psi(T) = (x(T), y(T),0)$,
$T \in \Omega \subseteq \C$, is a local parametrization of
$L_{\infty}$, then the cone is parameterized by $\Phi(T,z) = (x(T),
y(T),z)$, $z \in \C$. The singular points of $\calp_{\infty}^{-1}
(L_{\infty})$ belong to fibers sitting over the singular set of
$\tilf_{\infty}$ which, we recall, may intersect $L_{\infty}$
non-trivially due to the definition of ``regular leaf'' adopted in Section~2.2.
Away from its singularities, $\calp_{\infty}^{-1} (L_{\infty})$ can
be viewed as a complex surface equipped with a singular holomorphic
foliation. Let us then denote by $S$ this surface and by $\tilf_S$
the foliation on $S$ obtained by restriction of $\tilf$ to $S$. Note
that $S$ is invariant under the automorphism $(x,y,z)\mapsto
(x,y,\dl z)$, $\dl \in \C^*$ and so is the foliation $\tilf_S$.

Since $S$ is a 2-dimensional variety, $\tilf_S$ is a codimension~$1$
singular foliation on it and, hence, it has a transversely conformal structure. This
allows us to keep good control of the directions over which the leaves
of $\tilf_S$ ``become closer one to the others'' modulo choosing an
auxiliary Hermitian metric. This idea is well-known
and can be found, for instance, in \cite{Ghys-bourb}. In our case,
however, we shall use an explicit parametrization.
For this, let $M$ be equipped with affine coordinates $(x , y, z)$ such that
\begin{enumerate}

\item[(i)] $\{ z=0\} \subset \Delta_{\infty}$, $(x , y) \in
\C^2$, $z \in \C$.

\item[(ii)] the vector field $\tXX$ is given by
\begin{equation}
\tXX = \frac{1}{z^{d-1}} \left[ F (x , y) \frac{\partial}{\partial
x} + G (x , y) \frac{\partial}{\partial y} + zH (x , y)
\frac{\partial}{\partial z} \right]   \label{tXX}
\end{equation}
where $F ,G$ are polynomials of degree~either $d$ or $d+1$ and $H$ is a polynomial
of degree~$d$ (the independence of $F ,G$ and $H$ on the variable $z$ is a
consequence of the homogeneous character of $X$).

\item[(iii)] The projection $\calp_{\infty} : M \rightarrow \Delta_{\infty}$
in the above coordinates becomes $(x , y, z) \mapsto (x , y)$.

\end{enumerate}

Affine coordinates with the above indicated properties can be obtained as
follows. Recall that the blow-up $\widetilde{\C}^3$ of $\C^3$ at the origin
possesses affine coordinates $(x,y,w), \, (u,a,b), \, (r,v,s)$ arising from the realization of $\widetilde{\C}^3$
as the gluing of three copies of $\C^3$ by means of the identification
\begin{align*}
b &= \frac{1}{x} \, , \, \, u = xw \quad (b \ne 0, \, x \ne 0)\\
a &= \frac{1}{r} \, , \, \, u = rv \quad (a \ne 0, \, r \ne 0)\\
s &= \frac{1}{y} \, , \, \, v = yw \quad (s \ne 0, \, y \ne 0) \, .
\end{align*}
In the affine coordinates $(x,y,w)$, the vector field $\tXX$ takes the form
\[
\tXX = w^{d-1} \left[ F(x,y) \frac{\partial}{\partial x} + G(x,y) \frac{\partial}{\partial y} + wH(x,y) \frac{\partial}{\partial w} \right]
\]
for some polynomials $F, \, G, \, H$ depending solely on the variables $x,y$ (since $X$ is homogeneous). Now, to
obtain the mentioned coordinates, it suffices to take $w = 1/z$.

Note that $\Delta_{\infty}$ is itself isomorphic to $\CP (2)$. Thus the affine coordinates $(x,y) \simeq (x,y,0)$
on $\Delta_{\infty}$ defines an affine copy of $\C^2$ inside $\Delta_{\infty}$. Associated to the mentioned affine
$\C^2 \subset \Delta_{\infty}$, there is a notion of ``line at infinity'' for $\Delta_{\infty}$
itself. We shall denote this ``line'' by
$\Delta_{\infty}^{(x,y)}$. In particular, it follows that the domain of
definition of the coordinates $(x , y, z)$ coincides with the open set
$M \setminus (\Delta_0 \cup \calp_{\infty}^{-1} (\Delta_{\infty}^{(x,y)}))$.
Naturally the choice of the affine coordinates $(x,y)$ and of the line $\Delta_{\infty}^{(x,y)}$ are not canonical.
For a generic choice of these coordinates, $\Delta_{\infty}^{(x,y)}$ does not contain singular points of the corresponding
foliation on $\Delta_{\infty}$ and, besides,
$\Delta_{\infty}^{(x,y)}$ is not invariant by this foliation. Now, we have:

\begin{lema}
\label{lineatinfinitygeneric}
Suppose that the affine coordinates $(x,y)$ are chosen so that the resulting ``line at infinity'' $\Delta_{\infty}^{(x,y)}$ is not invariant by
the corresponding foliation on $\Delta_{\infty}$. Then the top-degree component of the vector field $\tXX$
has the form
\begin{equation}\label{topdegree}
z^{1-d} f(x , y) [x \partial /\partial x + y \partial /\partial y + z \partz]
\end{equation}
for a certain homogeneous polynomial $f$ having degree equal to~$d$.
\end{lema}

\begin{proof}
Suppose that the initial homogeneous vector field $X$ is given in standard coordinates $(z_1,z_2,z_3)$ for
$\C^3$ by $X = A(z_1,_2,z_3) \partial/\partial z_1
+ B(z_1,z_2,z_3) \partial/\partial z_2 + C(z_1,z_2,z_3)\partial/\partial z_3$. Then,
with the change of coordinates
$$
(x,y,z) \mapsto \left( \frac{x}{z}, \frac{y}{z}, \frac{1}{z} \right) = (z_1,z_2,z_3)
$$
the vector field $\tXX$ is given in a neighborhood of the hyperplane at infinity by
\[
\tXX = z^{1-d} [F(x,y) \frac{\partial}{\partial x} + G(x,y) \frac{\partial}{\partial y} + zH(x,y) \frac{\partial}{\partial z}]
\]
for $F(x,y) = A(x,y,1) - x C(x,y,1)$, $G(x,y) = B(x,y,1) - y C(x,y,1)$ and $H(x,y) = - C(x,y,1)$.
Now, the initial Euclidean coordinates $(z_1,z_2,z_3)$ for $\C^3$ can be chosen so that
none of the functions $A, \, B, \, C$ is divisible by $z_3$. This assumption, combined with the non-invariance of the ``line
at infinity'' $\Delta_{\infty}^{(x,y)}$ by the foliation in question,
implies that $F, \, G$ (resp. $H$) have degree $d + 1$ (resp. $d$). Since
$A(x,y,1)$ and $B(x,y,1)$ have degree at most~$d$, it follows that the top-degree
homogeneous component of $F$ (resp. $G, \, H$) is given by $x$ (resp.
$y, \, z$) times the top-degree homogeneous component of $C$. In other words,
the top-degree homogeneous component of the vector field $\tXX$ has the form~(\ref{topdegree})
as desired.
\end{proof}

A further comment concerning the difference between the foliation $\tilf_{\infty}$ induced by $\tXX$ on $\Delta_{\infty}$ and
the corresponding foliation $\tilf$ in the $3$-dimensional space is also needed. To be more precise, consider the vector field
$\tXX$ given by Formula~(\ref{tXX}) in the coordinates $(x,y,z)$. If $F,G$ have only trivial common factors, then the foliation
induced by $\tXX$ on $\Delta_{\infty}$ is given in $(x,y, \{z=0 \})$ coordinates by $F(x,y) \partial /\partial x + G(x,y)
\partial /\partial y$. Suppose now that $F$ and $G$ possess nontrivial common factors. Set $\textsc{P} = {\rm g.c.d.}\, (F,G)$
so that $F = \textsc{P}. a(x,y)$ and $G = \textsc{P} . b(x,y)$ with $a,b$ having only trivial common factors. In this case, the
foliation $\tilf_{\infty}$ is actually represented by the vector field $a(x,y) \partial /\partial x + b(x,y) \partial /\partial y$.
With this observation in place, we need to go one step further and consider also
the common divisors between $\textsc{P}$ and $H$. When $\textsc{P}$ and $H$ have non-trivial common
factors, then these common factor can be (factored out and)
eliminated without changing any of the foliations $\tilf, \, \tilf_{\infty}$.
Hence, as far as the foliations $\tilf$, $\tilf_{\infty}$ are concerned, we can suppose without loss of generality that
${\rm g.c.d.}\, (\textsc{P}, H)$ is invertible. Once this normalization has been made, two distinguished cases
may occur, namely:

\begin{itemize}
\item Suppose that $\textsc{P}$ is invertible (after reducing to the case where ${\rm g.c.d.}\, (\textsc{P}, H)$ is invertible).
Then the restriction of $\tilf$ to $\Delta_{\infty}$ coincides with
$\tilf_{\infty}$. Besides, in this case, the singular set of $\tilf$ intersects $\Delta_{\infty}$ in finitely many points.

\item Suppose that $\textsc{P}$ is not  invertible (after reducing to the case where ${\rm g.c.d.}\, (\textsc{P}, H)$ is invertible).
In this case, the foliation $\tilf_{\infty}$ does not coincide with the restriction of $\tilf$ to $\Delta_{\infty}$ since the latter
contains a curve of singularities which is induced in the above coordinates by $\textsc{P}$. In particular, the singular set
of $\tilf$ intersects $\Delta_{\infty}$ in a curve plus, occasionally, finitely many isolated points.
\end{itemize}

Summarizing the preceding discussion, the foliation $\tilf$ associated to $\tXX$ can be supposed to be
given by a polynomial vector field of the form
\begin{equation}
Y = \textsc{P} \left[ a (x,y) \frac{\partial}{\partial
x} + b (x , y) \frac{\partial}{\partial y} \right] + z H (x , y)
\frac{\partial}{\partial z} \, ,
\end{equation}
where ${\rm g.c.d.}\, (\textsc{P}, H)$ is constant. Furthermore the previously
defined vector field $\tXX$ is given in the same coordinates by
$$
\tXX = z^{1-d} Q(x,y) Y
$$
where $Q (x,y)$ is a polynomial. From this, it also follows that the projective curve $\{ \textsc{P}
=0 \} \subset \Delta_{\infty}$ (if not empty) is constituted by singularities of
$\tilf$ whereas its ``generic'' point is regular for
$\tilf_{\infty}$. Besides there are two different possibilities that need to be considered:
\begin{description}

\item[(a)] $\{ \textsc{P} = 0 \} \subset \Delta_{\infty}$ is invariant by
$\tilf_{\infty}$.

\item[(b)] $\{ \textsc{P} =0 \} \subset \Delta_{\infty}$ is not invariant
by $\tilf_{\infty}$.

\end{description}

\begin{obs}
\label{willitbeuseful?}
{\rm It will be seen later (Propositions~\ref{aproposition} and~\ref{aproposition2}) that $\{P = 0\} (\subseteq
\Delta_{\infty})$ is not invariant by $\tilf_{\infty}$ provided that the homogeneous polynomial
vector field of degree $d \geq 2$ is supposed also to be semi-complete.
For this reason the possibility of having $\{P=0\}$ invariant by $\tilf_{\infty}$ will be excluded from our discussion.}
\end{obs}

Our purpose is now to equip the leaves of $\tilf$ in $\Delta_{\infty}$
with an abelian form $\omega_1$ naturally related to the holonomy of the leaf
in question. This will be done in the affine copy of $\C^3$ in $M$ corresponding
to the domain of definition of the coordinates $(x,y,z)$.
With the preceding notations, let us fix a regular leaf $L_{\infty} \subset \Delta_{\infty}$
and a point $p \in L_{\infty}$ regular for $\tilf$. Under this
assumption, the leaf $L_{\infty}$ can locally be parametrized in the form $(x,y(x))$, or
$(x(y),y)$, and $z=0$. It suffices to consider a local parametrization of the form
$(x,y(x))$ since the other possibility is analogous. The vector field $\tXX$ then yields
$$
dz/dx = z H(x,y(x))/F(x,y(x)) \; .
$$
Therefore
\begin{equation}
z = z_0 \exp \left[ \int_{x_0}^x \frac{H(x,y(x))}{F(x,y(x))} dx \right] \; . \label{omega1}
\end{equation}
Thus we define an abelian form $\omega_1$ on $L_{\infty}$ by declaring that the coefficient
of $\omega_1$ at $(x,y(x))$ is nothing but $-H(x,y(x)) / F(x,y(x))$ (the minus sign is only a matter
of convention). In particular we note that possible non-trivial common factors between $F,H$ are automatically
canceled out in the definition of $\omega_1$.
If the leaf were parameterized in the form $(x(y),y)$, the analogous result would yield for coefficient
$-H(x,y(x)) / G(x,y(x))$. The form $\omega_1$ is the ``logarithmic derivative of the holonomy''
for the foliation $\tilf_S$ induced on the cone $S$ over $L_{\infty}$. This means the following:
let $L$ be a leaf of $\tilf_S$ and consider a path $c:[0,1] \mapsto
L$, on $L$. Denoting by ${\rm Hol}(c)$ the holonomy associated to
$c$, we have
\[
({\rm Hol} (c))^{\prime}(c(0)) = e^{-\int_c \omega_1} \, ,
\]
where ${\rm Hol} (c)$ is identified with a map between open sets of $\C$ equipped with the
coordinate~$z$.

Fixed a regular leaf $L_{\infty} \subseteq \Delta_{\infty}$ of $\tilf$
there are real trajectories, or paths, contained in $L_{\infty}$
and possessing a contractive holonomy. To construct these trajectories we proceed
as follows. The Abelian
form $\omega_1$ induces on $L_{\infty}$ a pair of real 1-dimensional
oriented singular foliations: the foliations given by $\{ {\rm
Im}(\omega_1) = 0\}$ and by $\{ {\rm Re}(\omega_1) = 0\}$. Denote by
$\mathcal{H}$ the oriented foliation defined by $\{ {\rm
Im}(\omega_1) = 0\}$, being the orientation determined by the
positivity of ${\rm Re}(\omega_1)$, i.e. if $\phi(t)$ is a
parametrization of a leaf of $\mathcal{H}$ then ${\rm Re}(\omega_1.\phi^{\prime}(t))
= \omega_1.\phi^{\prime}(t) > 0$.
Each oriented trajectory of the foliation $\mathcal{H}$ will be called a real trajectory.

To make use of the foliation $\mathcal{H}$, it is clearly important
to have information about its singular set. Since $\mathcal{H}$ depends only on the foliation associated
to $\tXX$ (rather than on $\tXX$ itself), we identify four ``critical regions''
that may give rise to singularities for $\mathcal{H}$, namely:
\begin{enumerate}
\item Singular points of $\tilf_{\infty}$.

\item Points in the curve $\{ H =0 \}$ (assuming as before that ${\rm g.c.d.}\, (\textsc{P}, H)$ is a constant).

\item Points in the curve $\{ \textsc{P} =0 \}$.

\item The line at infinity $\Delta_{\infty}^{(x,y)} \subset \Delta_{\infty}$
(defined by means of the affine coordinates $(x,y)$).

\end{enumerate}
In the sequel we shall determine the structure of the
foliation $\calh$ in cases (2),~(3) and (4) above. The discussion of
singular points of $\tilf_{\infty}$ will mostly be carried out in Sections~4 and~5.

\begin{obs}
\label{forintersectionpoints}
- {\bf A comment about the local behavior of $\calh$ about certain ``degenerate'' intersection points} -
{\rm The purpose of this remark is to explain why certain ``more degenerate'' points belonging
to the intersection of different curves as above need not be singled out in our discussion. Naturally,
this discussion concerns only those points that are also regular for $\tilf_{\infty}$ since the singular
points of this foliation need to carefully be discussed later.

Let us then consider, for example, regular points for $\tilf_{\infty}$
belonging to the intersection of the curves $\{ H =0 \}$ and $\{ \textsc{P} =0 \}$.
Clearly these points are in finite number. We claim that the exact nature of the singularity of $\calh$ at them need not be
worked out. The reason for this is as follows. Consider local coordinates $(u,v,w)$ identifying the point
in question to the origin of $\C^3$ and such that the foliation is locally represented by the vector
field $\partial /\partial u$. In particular, the intersection points of $\{ H =0 \}$ and $\{ \textsc{P} =0 \}$
are reduced to the origin and it is contained in the local leaf $L_0 = \{ v=w=0\}$. In this sense,
the trajectories of $\calh$ on the remaining leaves are ``well-defined'', i.e. their local behavior
is supposed to have been determined. Next, consider a poly-disc $B(\epsilon)$ of radius $\epsilon >0$ about the origin.
The behavior of $\calh$ away from $B(\epsilon)$ is hence determined, including for those trajectories
contained in the leaf $L_0$. The trajectories on $\calh$ lying in $L_0 \cap B(\epsilon)$ can then {\it be defined}\,
through the corresponding trajectories lying in leaves different from $L_0$. For example, consider a point $(u_0,0,0)$
lying in the boundary of $B(\epsilon)$. To {\it define}\, the trajectory of $\calh$ through $(u_0,0,0)$, we
consider a sequence of points $(u_0, \delta_1,\delta_2)$ converging to $(u_0,0,0)$ and the corresponding
$\calh$-trajectories $l_{\delta_1 \delta_2}$ through these points. On the complement of $B(\epsilon)$,
these trajectories converge to (disconnected) segments of $\calh$-trajectories contained in $L_0$.
We can then use as trajectory through $(u_0,0,0)$ {\it inside $L_0 \cap B(\epsilon)$}
any segment joining two connected components of the above mentioned segments of $\calh$-trajectories contained in $L_0$.
In this way, the behavior of the $\calh$-trajectories in $L_0 \cap B(\epsilon)$ is fully determined by
the behavior of $\calh$ at the boundary of $B(\epsilon)$. For example, suppose
that for every point in the boundary of $B(\epsilon)$, the corresponding $\calh$-trajectory is oriented {\it inward} the poly-disc
$B(\epsilon)$. Then the trajectories of $\calh$ inside $L_0 \cap B(\epsilon)$ should be regarded as ``exhibiting
a sink singularity at the origin''. In other words, we should consider that these oriented trajectories
have an ``endpoint'' where they meet the boundary of $B(\epsilon)$, cf. Definition~\ref{justadddefinition1}.

Concerning the previous discussion, it is also convenient to point out that the polynomial $\textsc{P}$ will
be constant in most of our applications.

Let us close this discussion, with a few analogous comments concerning points belonging to the line at infinity
$\Delta_{\infty}^{(x,y)}$. First recall that our ``generic'' choice of affine coordinates $(x,y)$
is such that $\Delta_{\infty}^{(x,y)}$ neither contains singularities of $\tilf_{\infty}$ nor is invariant by
this foliation. Among points in $\Delta_{\infty}^{(x,y)}$, there are three class of ``non-generic points'' that
may be regarded as ``more degenerate than generic points''. Namely, we have intersection points
with $\{ H =0 \}$, with $\{ \textsc{P} =0 \}$ and those points where $\tilf_{\infty}$ fails to be transverse
to $\Delta_{\infty}^{(x,y)}$ (i.e. ``tangency points''). Again the collection of all these points form a finite
set that can be treated with the same point of view discussed above. Alternatively, the reader may also
consider the argument provided in the proof of Lemma~\ref{lineatinfinity} pointing out a sort of ``dominant''
behavior of $\Delta_{\infty}^{(x,y)}$ over the other ``critical regions''.}
\end{obs}

In view of Remark~\ref{forintersectionpoints}, let us begin to work out the local behavior of $\calh$
at points in the above listed ``critical regions'' without paying special attention to points that belong simultaneously
to more than one of these regions. First, consider
the curve $\{ H =0 \}$ corresponding to zeros of $\omega_1$. In fact, for the time
being, we shall restrict ourselves to points in the curve $\{H=0\}$ that happens to be regular for the foliation
$\tilf$.

\begin{lema}\label{tiposing}
Let $p \in \Delta_{\infty}$ be a regular point of $\tilf$. Assume that $p$ lies in the curve $\{ H =0 \} \cap \Delta_{\infty}$
(but not in $\{ \textsc{P} =0 \}$ since $p$ is regular for $\tilf$). Then $p$ is a singular point for $\mathcal{H}$. Besides
the local structure of $\mathcal{H}$ restricted to the leaf of $\tilf$ through $p$ is a saddle with $2m$ (real) separatrices
(for a certain $m \geq 1$).
\end{lema}

\begin{proof}
Since $p \in \Delta_{\infty}$ is a regular point for $\tilf$, it follows that $P$ does not vanish at $p$. Furthermore,
at least one between $F$ and $G$ does not vanish at $p$ as well. Assume, without loss of generality, that $P(p) \ne 0$.
We then conclude that the restriction of $\omega_1$ to $L_p$ is holomorphic about $p$ with a zero at $p$.
The structure of the real foliation induced near a zero of a holomorphic $1$-form on a Riemann surface is always a
saddle as in the statement. Here the number ``$m$'' of separatrices corresponds precisely
to the order of $p$ as zero of $\omega_1$.
\end{proof}

Let us now work out the behavior of $\calh$ at points of $\{ \textsc{P} =0 \}$ (again, only regular points for $\tilf$ are
considered here). Clearly it is sufficient to consider the domain of definition of the coordinates $(x,y,z)$. Similarly, if
$\textsc{P} = \textsc{P}_1^{k_1} \cdots \textsc{P}_l^{k_l}$ is the decomposition of
$\textsc{P}$ into irreducible components, then it suffices to consider the curve $\{ \textsc{P}_1^{k_1} =0 \}$.

\begin{lema}\label{curvaP=0}
Suppose that $\{ \textsc{P}_1 =0 \} \cap \Delta_{\infty}$ is not invariant by $\tilf_{\infty}$. If $k_1 \geq 2$
then $\omega_1$ has a pole of order $k_1 \geq 2$ at a generic point $p$ of this curve so that $\calh$ has a
saddle-singularity at $p$. On the other hand, if $k_1 = 1$, then $\omega_1$ has a simple pole at a generic point
$p$ of this curve, whose residue equals $H(p)/ F^{\ast} (p)$ where $F^{\ast} = F/  \textsc{P}_1$.
\end{lema}

\begin{proof}
Suppose that the curve $\{P_1 = 0\} \cap \Delta_{\infty}$ is not invariant by $\tilf_{\infty}$. Then, at generic
points in $\{P_1 = 0\}$, the curve in question is transverse to $\tilf_{\infty}$.

 for a
generic point of this curve we have that $\{P_1 = 0\}$ is transverse to $\tilf_{\infty}$. The point $p$ can
also be chosen sufficiently generic so that $\{P_1 = 0\}$ is smooth at $p$ and no other irreducible component
of $P$ vanishes at $p$. Under these generic assumptions, it follows that $F^{\ast}(0) \ne 0$. Moreover, since
we are assuming ${\rm g.c.d.}\, (\textsc{P}, H)$ to be invertible, we can also assume that $H(p) \ne 0$. The
$1$-form $\omega_1$ has therefore a pole of order $k_1$ whose coefficient is equal to $H(p)/F^{\ast}(p)$.
\end{proof}

Let us now consider points belonging to the line at infinity $\Delta_{\infty}^{(x,y)} \subset \Delta_{\infty}$.
Here the reader is reminded that our choice of coordinates was made so that $\Delta_{\infty}^{(x,y)}$
contains no singular point of the corresponding foliations.

\begin{lema}
\label{lineatinfinity}
Points belonging to $\Delta_{\infty}^{(x,y)}$ yield source singularities for $\calh$ provided that the coordinates $(x,y)$ are
generically chosen.
\end{lema}

\begin{proof}
As mentioned, our choice of coordinate is such that $\Delta_{\infty}^{(x,y)}$ neither contains singular points of $\tilf_{\infty}$
nor is invariant by $\tilf_{\infty}$. It then follows that each
point $p$ in $\Delta_{\infty}^{(x,y)}$ locally belongs to a
unique leaf $L_p$ of $\tilf_{\infty}$. Thus
the map that assigns to $p$ the residue at $p$ of the $1$-form $\omega_1$ is globally defined
on $\Delta_{\infty}^{(x,y)}$. To establish the statement, it suffices to check this residue equals $1$
at a generic point of $\Delta_{\infty}^{(x,y)}$. Indeed, by a continuity argument,
this will imply that the residue must be real strictly positive
at every point $p$ in $\Delta_{\infty}^{(x,y)}$ so that all these points constitute source singularities for $\calh$.
Alternatively, the reader may use the point of view discussed in Remark~\ref{forintersectionpoints}.

Let us then consider those points where $\Delta_{\infty}^{(x,y)}$ is transverse to $\tilf_{\infty}$.
Let $(u,v,w)$ be new local affine coordinates for $M$ where
$w$ is the coordinate transverse to $\Delta_{\infty}$ and such that the line at infinity
$\Delta_{\infty}^{(x,y)}$ is given by $\{u = 0\}$. The standard change of coordinates
associated is then given by $(u,v,w) \longmapsto (1/u, v/u, w) = (x,y,z)$. In these new
coordinates, the vector field $\tXX$ becomes (up to multiplication by $w^{1-d}$)
$$
-u^2 F(1/u, v/u) \frac{\partial}{\partial u} + u(-vF(1/u, v/u) + G(1/u, v/u)) \frac{\partial}{\partial v}
+ w H(1/u, v/u) \frac{\partial}{\partial w} \; .
$$
Recall that the polynomial vector field $F(x,y) \partial /\partial x + G(x,y) \partial /\partial y$
has degree $d+1$. Furthermore its component of degree $d+1$ has the form $f(x , y) [x \partial
/\partial x + y \partial /\partial y]$ where $f$ is homogeneous of degree~$d$ (cf. Lemma~\ref{lineatinfinitygeneric}).
In particular, $u^2 F(1/u, v/u)$ has a pole of order $d-1$ over $\{ u=0\}$. Similarly the top-degree homogeneous
component of $-vF(1/u, v/u) + G(1/u, v/u)$ vanishes identically so that $\Delta_{\infty}^{(x,y)}$
represents a polar component of degree~$d-1$ for the component of $\tXX$ in the $v$-direction as
well. Finally, the order of poles of $H(1/u,v/u)$ over $\Delta_{\infty}^{(x,y)}$ equals~$d$.
Formula~(\ref{omega1}) then shows that $\omega_1$ has poles of order~$1$ over $\Delta_{\infty}^{(x,y)}$.
Indeed the principal part of $\omega_1$ is simply $1/u$, since the top-degree homogeneous component
of $\tXX$ is given
by Equation~(\ref{topdegree}). The statement follows at once.
\end{proof}

Before proceeding further, let us summarize the information so far obtained about the singular set of $\calh$ in the
``critical regions''~(2), (3) and (4).
\begin{itemize}
\item[(a)] The regular points of $\tilf_{\infty}$ contained in $\{H = 0\} \cap \Delta_{\infty}$ always
provide singular points for $\calh$. Such singular points correspond to saddles with $2m$ (real) separatrices,
for $m \geq 1$.

\item[(b)] The generic points in $\{P = 0\} \cap \Delta_{\infty}$ also provide singular points for $\calh$.
These points can provide either  poles of order~$\geq 2$ for $\omega_1$ or poles of order~$1$ for $\omega_1$.
In the first case, the corresponding singular behavior of $\calh$ corresponds to a saddle.

\item[(c)] The points belonging to the ``line at infinity'' always yields singular points for $\calh$. More
precisely, they provide simple poles with residue equal to~$1$ and therefore yields source singularities for
$\calh$.
\end{itemize}
Concerning the second item above, in the case where $\omega_1$ has a simple pole at a generic point, relevant
information on the residue cannot be obtained without further information on the vector field. In fact, the
residue (which is given by $H(p)/F^{\ast}(p)$) takes its values in $\C^{\ast}$. Nonetheless, if the vector field
$\tXX$ is supposed to be semi-complete, then the mentioned residue must belong to
$\R^{\ast}$. Therefore, the singular point
corresponds to a sink (resp. source) provided that the residue belongs to $\R_-$ (resp. $\R_+$). This is the
contents of the next lemma.

\begin{lema}\label{curvaP=0_semi-complete}
Assume that $\{P_1 = 0\} \cap \Delta_{\infty}$ is not invariant by $\tilf_{\infty}$ and that $k_1 = 1$. Assume
in addition that $X$ is semi-complete. Then $\omega_1$ has a simple pole at a generic point $p$ of the curve
$\{P_1 = 0\} \cap \Delta_{\infty}$ and the residue of $\omega_1$ at $p$ belongs to $\R^{\ast}$.
\end{lema}

\begin{proof}
Suppose that the curve $\{P_1 = 0\} \cap \Delta_{\infty}$ is not invariant by $\tilf_{\infty}$. Then, at
generic points of $\{P_1 = 0\}$ this curve intersects $\tilf_{\infty}$ transversely.
Consider a generic point $p \simeq (0,0,0) \in \{P_1 = 0\} \cap \Delta_{\infty}$ and local coordinates
$(u,v,z)$ about $p$ where the foliation $\tilf_{\infty}$ is locally represented
by $\partial /\partial u$ and where $P_1(u,v) = u$
(note that the point $p$ can be chosen so that $\{P_1 = 0\}$ is smooth
at $p$). Modulo choosing $p$ sufficiently generic to ensure that neither $H$ nor any other irreducible component
of $P$ vanishes at $p$, the vector field $\tXX$ takes the local form
\[
\tXX = z^{1-d} Q(u,v) \left[ u f(u,v) \frac{\partial}{\partial u} + z h(u,v) \frac{\partial}{\partial z} \right]
\]
where $f(0,0) \ne 0$ and $h(0,0) \ne 0$. Now the semi-complete character of $X$ ensures that the quotient of the
eigenvalues of the linear part of the vector field $u f(u,v) \partial /\partial u + zh(x,y)
\partial /\partial z$ is a rational number.
In other words, the quotient $h(0,0)/f(0,0)$ belongs to $\Q^{\ast} \subseteq \R^{\ast}$. This quotient, however,
also represents the residue of $\omega_1$ at $p \simeq (0,0,0)$. The lemma is proved.
\end{proof}

To close this section, let us introduce the global notion of trajectory for the foliation $\mathcal{H}$ under the
condition that the trajectory in question remains away from the singular set of $\tilf$. For this it is also convenient to
consider the standard Euclidean metric on the affine copy of $\C^3$ in which
the affine coordinates $(x,y,z)$ are defined since we shall also want to define the {\it length}\, of a global trajectory.
Here we shall need the notion of {\it endpoint of a trajectory}, cf. Definition~\ref{justadddefinition1}
below. We emphasize that all definitions provided here are only valid
for (segments of) trajectories remaining away from the singular set of $\tilf_{\infty}$. They will then be
completed in Section~5 once the local behavior of $\calh$ about these singular points will have been
worked out.

To motivate the definition, let us first provide some geometric interpretation for
the foliated $1$-form $\omega_1$. Fix a regular leaf $L_{\infty} \subseteq \Delta_{\infty}$
and a point $p_0 \in L_{\infty}$ that is regular for $\tilf$. Let $l$ be the real trajectory of $\mathcal{H}$ through $p$
and $S$ the cone over $L_{\infty}$. Consider a parametrization $c:[0,1] \rightarrow l$ of the segment of this trajectory
joining $p_0 = c(0)$ to $p_1 = c(1)$. Then the holonomy map ${\rm Hol}(c): \Sigma_0 \rightarrow \Sigma_1$, where
$\Sigma_0, \, \Sigma_1$ are vertical complex lines equipped with the coordinates~$z$, satisfies
\begin{equation}\label{contr}
|({\rm Hol}(c))^{\prime}| = e^{- Re(\int_c \omega_1)} < 1 \, .
\end{equation}
Clearly this formula means that the holonomy map in question is contractive. The role played by these trajectories in our
discussion can be summarized as follows. Nearby a sink singularity $p$ of $\calh$, all $\calh$-trajectories
converge to $p$. Estimate~(\ref{contr}) guarantees that
the distance of the leaves of $\tilf_S$ to $L_{\infty}$ has a local
minimum at $p$ (which may well be zero). On the other hand, nearby a
source $p$, all real trajectories go away from $p$. This means that the
distance of the leaves of $\tilf_S$ to $L_{\infty}$ reaches a local maximum at $p$.

We can now go on two provide a global definition for the trajectories of $\calh$. The reader is again reminded that,
until Section~5, these definitions are only valid provided that the trajectory in question remains away from
the singular set of $\tilf_{\infty}$. On the other hand, recall that at regular points of $\tilf_{\infty}$,
the foliation $\calh$ can have only three types of singularities, namely: sources, sinks and saddles.

\begin{defnc}
\label{justadddefinition1}
A point $p$ regular for $\tilf_{\infty}$ is a future endpoint (resp. past endpoint)\, for a trajectory
of $\calh$ if and only if $\calh$ has a sink singularity (resp. source singularity) at $p$.
\end{defnc}

Thus, by definition, only sinks or sources
singularities of $\mathcal{H}$ (corresponding to maxima or minima for the
distance function between a leaf) can provide {\it endpoints}\, for a trajectory of $\mathcal{H}$. Hence,
to have a global definition of $\calh$-trajectories it only remains to say how they are defined
on a neighborhood of saddle singularity of $\calh$. For this, recall that a saddle singularity has
an even number ($2m$) of separatrices, $m$ of them converging to the singular point and $m$ of them
being ``emanated'' from the singular point. We now have:

\begin{defnc}
\label{justadddefinition2}
Suppose that $l$ is a trajectory of $\calh$ converging, as a separatrix, to a saddle
singularity of $\mathcal{H}$. The trajectory $l$ is then continued from this singular point
by following any of the local separatrices that are emanated from the mentioned singular point.
\end{defnc}

The reader will not fail to observe that a trajectory of $\calh$ passing through saddle singularities
of $\calh$ keeps giving rise to contractive holonomy maps what, in turn, justify the previous definitions.

The length of a $\calh$-trajectory is then defined by adding up its lengths (for the auxiliary metric
fixed above) over foliated coordinates and this procedure is conducted simply by locally following
the orientation. More precisely, consider a point $p$ and denote by $l^+_p$ the (oriented) semi-trajectory
of $\calh$ initiated at $p$. The length of $l^+_p$ is obtained by adding lengths of its (local) segments provided
that this trajectory can locally be continued (and regardless of whether or not we pass several times over the same
points of $M$). The definition of length for an entire trajectory $l$ of $\calh$ (as opposed to a semi-trajectory)
follows naturally.

It follows from the preceding that the length of $l^+_p$ is finite if and only if $l^+_p$ has a future
endpoint (i.e. $l^+_p$ meets a sink singularity). In particular, if $l^+_p$ becomes periodic, then its length
is automatically infinite.

We can now extend the previous definitions to trajectories of $\calh$ defined on all of $M$ and not
only on $\Delta_{\infty}$. In fact, the
real oriented foliation $\mathcal{H}$ or, equivalently, the $1$-form $\omega_1$ has been introduced for
leaves contained in $\Delta_{\infty}$. As soon as a leaf $L_{\infty} \subseteq \Delta_{\infty}$ is fixed,
their definition can naturally be adapted to every leaf on $S$. Going back to our specific case in which
$\omega_1$ is characterized by Formula~(\ref{omega1}), it follows that the local trajectories of $\calh$
on $L \subseteq S$ are determined as the lift in $T_{(x,y(x))}L_{\infty}$ of the vector $v$ where $v$ is
such that $v.H(x,y(x)) / F(x,y(x))$ belongs to $\R_-$. Also, it is to be noted that the corresponding abelian
form $\omega_1$ is independent of the leaf in the same cone $S$. In fact, Equation~(\ref{omega1}) shows that it
depends solely on $L_{\infty}$. These remarks can be summarized as follows.
\begin{enumerate}
\item The trajectory of $\calh$ through a point $(p_1, p_2,p_3)$ projects on the trajectory of $\calh$ through the point
$(p_1,p_2,0)$ which, in addition, is globally contained in the plane $\{z=0\}$.

\item Since the absolute value of the coordinate ``$z$'' is always decreasing over a trajectory of $\calh$, it follows that
the trajectory of $\calh$ through $(p_1, p_2,p_3)$ has infinite length if and only if the the trajectory of $\calh$ through $(p_1, p_2,0)$
has infinite length.
\end{enumerate}

The following simple lemma will also be important later on. For this lemma we should take into account that, whereas $\Delta_{\infty}^{(x,y)}$
can be chosen ``generic'', it always possesses points of tangency with the foliation $\tilf_{\infty}$.

\begin{lema}
\label{remainingcompact}
For a generic choice of the affine coordinates $(x,y)$ an oriented trajectory of $\calh$ never intersects
$\Delta_{\infty}^{(x,y)}$. Besides there is a compact part $K \subset \C^3$ and a constant $C_K$ so that
the following holds: every segment of trajectory $l$ of $\calh$ whose total length is greater than $C_K$
verifies the condition that the
part of $l$ lying in $\C^3 \setminus K$ is less than, say, $1/10$ of the total length of the segment in question.
\end{lema}

\begin{proof}
Let $q_1, \ldots , q_r$ be the points where $\Delta_{\infty}^{(x,y)}$ is tangent to $\tilf_{\infty}$ and fix
a small neighborhood $W_i$ of $q_i$, $i=1, \ldots ,r$. Then there is a ``tubular neighborhood'' $V$ of
$\Delta_{\infty}^{(x,y)} \setminus \bigcup_{i=1}^r W_i$ so that the following holds: for every point
$p \in \partial V \setminus \bigcup_{i=1}^r W_i$ the trajectory of $\calh$ through $p$ is transverse to
$\partial V$ and oriented outwards $V$. In other words, no trajectory of $\calh$ may enter $V$ without
entering first some $W_i$.

On the other hand the structure of $\calh$ trajectories on $W_i$ is easy to describe. If
$\Delta_{\infty}^{(x,y)}$ is ``sufficiently generic'', then the tangency of $\Delta_{\infty}^{(x,y)}$
at $\tilf_{\infty}$ at $q_i$ is quadratic (for all $i \in \{1, \ldots, r\}$). In particular if $L_i$ is the
local leaf of $\tilf_{\infty}$ through $q_i$, the point $q_i$ is itself a source-type singularity for the
trajectories of $\calh$. Thus no trajectory of $\calh$ actually intersects $\Delta_{\infty}^{(x,y)}$. Finally,
if $V$ and the neighborhoods $W_i$ are sufficiently small, then the length of a segment of trajectory lying in
$V\cup \bigcup_{i=1}^r W_i$ is less than, say, $1/30$ the length of the segment of same trajectory in
$K=\C^3 \setminus V\cup \bigcup_{i=1}^r W_i$ which is determined by two ``successive'' passages of the trajectory
in question through $V\cup \bigcup_{i=1}^r W_i$. The statement then follows.
\end{proof}

\begin{obs}\label{lastsection3}
{\rm In certain cases it may be useful to make a ``non-generic'' choice of the affine coordinates $(x,y)$ so as to have a
line at infinity $\Delta_{\infty}^{(x,y)}$ passing through certain singular points of $\tilf_{\infty}$. We shall briefly
mention one situation of this type later on, cf. Remark~(\ref{tobeincluded}).}
\end{obs}

\section{Renormalization in the exceptional divisor}

Our fundamental tool to derive Theorem~A is a procedure of renormalization for
the complex time near the divisor of poles of $\tXX$, i.e. near $\Delta_{\infty}$. This construction will play
a major role in the rest of the paper. Let us then begin by describing this procedure. We shall keep the notations of Section~3
emphasizing the $3$-dimensional case though all the results presented below are valid in arbitrary dimensions.

As before, let $X$ stand for a homogeneous polynomial vector field of degree $d \geq 2$ and assume that $X$ is not a multiple
of the radial vector field. Denote by $\fol$ the foliation associated to $X$. It is convenient to remind the reader
that ``leaves'' for the foliation $\fol$ are defined as in Section~2.2. In particular, the restriction of $X$ to a leaf
$L$ of $\fol$ may contains zeros, poles and essential singularities of $X$. All these non-regular points of $X$ form, however,
a discrete subset of $L$ with respect to the intrinsic topology of the leaf $L$ as Riemann surface (unless the leaf
in question is fully contained in the divisor of zeros and poles of $X$).
The presence of the vector field $X$ allow us to endow every (regular) leaf $L$ as before
with the foliated time-form $dT$ defined, as in Section~2.2, by letting $dT.X =1$. The time-form will also be denoted
by $dT_L$ when we want to emphasize the leaf $L$ under consideration. As observed,
the time-form is well-defined provided that $L$ is not contained in the divisor of zeros and poles of $X$.
If the vector field $X$ is supposed to be semi-complete, then its restriction to $L$ is everywhere holomorphic and the
orders of its zeros cannot exceed~$2$, cf. Lemma~\ref{Adim2}. It follows at once that $dT$ is meromorphic on all of $L$
and it has no zeros. Besides the poles of $dT$ have order bounded by~$2$.
Finally, recall also that given a curve $c:[0,1] \rightarrow L$ joining two points
$c(0)$, $c(1)$ in $L$ satisfying $X(c(0)) \neq 0$, $X(c(1)) \neq 0$, the integral
$\int_c dT$ measures the time needed to cover $c$ from $c(0)$ to $c(1)$ following
the flow of $X$ so long $X$ is semi-complete. In fact, when a vector field is semi-complete the notion of ``time''
arising from its semi-global flow is well-defined.

Consider now $\tXX$, the vector field induced by $X$ on $M$. Throughout this section, generic affine coordinates $(x,y,z)$
as in Section~3 are supposed to be fixed. In particular, $\Delta_{\infty}^{(x,y)}$ neither contains singular points
of $\tilf_{\infty}$ nor is invariant by this foliation. Since $\tXX$ has poles over $\Delta_{\infty}$,
the time-form is not defined for the regular leaves of $\tilf_{\infty}$. It is, however, possible to define a
``renormalized time-form'' on a neighborhood of each regular point $p$ of a leaf $L_{\infty} \subseteq
\Delta_{\infty}$. This goes as follows. Let $L_{\infty} \subseteq \Delta_{\infty}$ be a regular
leaf of $\tilf$ and let $p \in L_{\infty}$ be a regular point of $L_{\infty}$ which is not singular
for $\tilf$. Choose local coordinates $(u,v,w)$, $\{ w=0\} \subset \Delta_{\infty}$ around $p$ where the foliation becomes locally
given by the vector field $\partial / \partial u$. In these coordinates, the vector field $\tXX$ is given by ${w}^{1-d} f(u,v,w)
\partial/\partial u$. The ``renormalized time-form'' over $L_{\infty}$ is then defined as
$du/f(u,0,0)$. In other words, it is obtained from $\tXX$ by eliminating its pole
component. Naturally there is no canonical choice for the coordinate $w$ and this prevents us
from having a global definition for the ``renormalized time-form''. In accurate terms, the local form
$du/f(u,0,0)$ is not globally defined on $L_{\infty}$ because, when a change of coordinates is
performed, two local definitions of this ``renormalized time-form'' will agree only up to a
multiplicative constant. Therefore, whereas the previous construction
does not define an Abelian form on $L_{\infty}$, it endows $L_{\infty}$ with
an affine structure (for further details we refer to \cite{guillotreb}). The purpose
of this section is to exploit this affine structure to estimate the domain of definition of the
solutions of $\tXX$. As it will be seen, precise estimates can be obtained in this way as long as
the ``evolution'' of the coordinate ``$z$'' is well-controlled (where ``$z$'' refers to the affine coordinates
$(x,y,z)$).

Although we have defined the ``renormalized time-form'' only at regular points of $\tilf$,
this form admits a natural asymptotic extension to the singularities of $\tilf$ lying in $\Delta_{\infty}$.
Details on these extensions will be given as they become necessary.

Now let us return to homogeneous polynomial vector fields on $\C^3$.
Fix a point $p_0$ contained in the singular set of $\tilf_{\infty}$.
Suppose that the restriction of $\tilf$ to a neighborhood of $p_0$
is given by Equation~(\ref{tXX}) so that
\begin{equation}\label{campo_forma_usual}
\tXX = \frac{1}{z^{d-1}} \left[ F(x,y) \frac{\partial }{\partial x}
+ G(x,y) \frac{\partial }{\partial y} + zH(x,y) \frac{\partial
}{\partial z} \right].
\end{equation}
With the notations of Section~\ref{sechom}, let $\textsc{P} = {\rm
g.c.d.}\, (F,G)$ so that $F = \textsc{P}. a (x,y)$ and $G =
\textsc{P} . b (x,y)$. Denoting by $\overline{\textsc{P}}$
the greatest common divisor between $\textsc{P}$ and $H$, we can set
$\textsc{P} = \overline{\textsc{P}} \textsc{P}^{\ast}$ and
$H = \overline{\textsc{P}} H^{\ast}$. It follows that
\begin{equation}
\label{campo}
\tXX = \frac{\overline{\textsc{P}}}{z^{d-1}} \left[ \textsc{P}^{\ast}(x,y) \left( a (x,y)
\frac{\partial }{\partial x} + b (x,y) \frac{\partial }{\partial y}
\right) + zH^{\ast} (x , y) \frac{\partial}{\partial z} \right]
\end{equation}
where $p_0 \simeq (0,0,0)$. If $\textsc{P}^{\ast}$ is not constant,
the curve in $\Delta_{\infty}$ induced by $\{ \textsc{P}^{\ast} =0\}$ is
singular for $\tilf$, though its generic points are regular for $\tilf_{\infty}$. From this
point of view $\{ \textsc{P}^{\ast} =0\} \cap \Delta_{\infty}$ may or may not be invariant by $\tilf_{\infty}$.
Nonetheless, when dealing with semi-complete vector fields, the following holds.

\begin{prop}
\label{aproposition}
Assume that $X$ is a homogeneous semi-complete vector field with degree greater than or equal to~$3$. Suppose that
$\tXX$ is as in Equation~(\ref{campo}). Then no irreducible component of $\{\textsc{P}^{\ast} = 0\} \cap \Delta_{\infty}$
is invariant by $\tilf_{\infty}$. In other words, a regular leaf of $\tilf_{\infty}$ can intersect the singular set
of $\tilf$ only over a discrete set (for the intrinsic topology on the leaf in question).
\end{prop}

\begin{proof}
As already observed, codimension~$1$ components of the singular set of $\tilf$ coincide with the set
$\{ \textsc{P}^{\ast} =0\} \cap \Delta_{\infty}$. Without loss of generality, denote by
$\textsc{P}_1$ an irreducible component of $\textsc{P}^{\ast}$ giving rise to an (irreducible) curve $\mathcal{C}=
\{ \textsc{P}_1 =0\} \cap \Delta_{\infty}$ that happens to be invariant by $\tilf_{\infty}$. We are going
to conclude from this condition that $X$ cannot be semi-complete.

To do this, denote by $m \geq 1$ be the multiplicity of $\textsc{P}_1$ as component of $\textsc{P}^{\ast}$.
At a generic point of $\mathcal{C}$, local coordinates $(u,v,w)$
can be found so that $\{ w=0\} \subset \Delta_{\infty}$ and $\mathcal{C}$ is identified with $\{v=0, w=0\}$. In these
coordinates, we naturally have $\textsc{P}_1 (u,v) = v$. Besides, if the chosen point is sufficiently generic, then we also
have $H(0,0) \neq 0$ and $a(0,0) \ne 0$. On the other hand, $b$ must be divisible by $v$ since $\mathcal{C}$ is invariant
by $\tilf_{\infty}$.

By means of the above defined local coordinates, $\tXX$ can be identified to a vector field defined about the
origin of $\C^3$. Since $m \geq 1$, the first non-zero homogeneous component
of $\tXX$ at the origin is given by
\[
\tXX^H = w^{1-d} v^k \left[ \da v \frac{\partial }{\partial u} + \dl w \frac{\partial }{\partial w} \right]
\]
for some constants $\dl = H(0,0) \in \C^{\ast}$, $k \geq 0$ and $\da \in \C$. Note that $\da \ne 0$ if and only
if $m =1$. Furthermore, $k$ is the greatest (non-negative) integer such that $v^k$ divides $\overline{\textsc{P}}$.
The hyperplanes $\{v = {\rm cte}\}$ are invariant by the foliation associated to $\tXX^H$. For each non-zero
constant (${\rm cte}$) sufficiently small, the differential equation associated to $\tXX^H$ is such that
$\dot{w} = {\rm cte}^k \dl w^{2-d}$. Since $d \geq 3$, the vector field associated to this differential equation
has a pole at $w=0$. In turn, the existence of this pole ensures that the corresponding solution is multivalued contradicting
the assumption that $X$ is semi-complete. The proof of the lemma is over.
\end{proof}

Concerning the case of {\it quadratic homogeneous}\, vector fields, the preceding lemma can nicely be complemented
under the additional assumption that the singular set of $X$ has codimension greater than or equal to~$2$. Namely, we have:

\begin{prop}\label{aproposition2}
Assume that $X$ is a quadratic homogeneous semi-complete vector field whose singular set has
codimension greater than or equal to~$2$. Suppose that $\tXX$ is as in Equation~(\ref{campo}). Then no irreducible
component of $\{\textsc{P}^{\ast} = 0\} \cap \Delta_{\infty}$ is invariant by $\tilf_{\infty}$.
\end{prop}

\begin{proof}
Let $\textsc{P}_1$ denote a non-trivial irreducible component of $\textsc{P}^{\ast}$ and assume for a contradiction that $\{\textsc{P}_1
= 0\} \cap \Delta_{\infty}$ is invariant by $\tilf_{\infty}$. Denote by $L$ the intersection $\{\textsc{P}_1 = 0\}
\cap \Delta_{\infty}$ and note that $L$ is contained in a leaf of $\tilf_{\infty}$. Fix $p \in L$ such that $\tilf_{\infty}$
is regular at $p$. The point $p$ can be chosen so that neither
$H$ nor any other irreducible component of $\textsc{P}^{\ast}$ vanishes at $p$. Next, consider local coordinates
$(u,v,w)$ about $p$, $\{ w=0\} \subset \Delta_{\infty}$, so that
\begin{itemize}
\item[(a)] $p \simeq (0,0,0)$
\item[(b)] $\textsc{P}(u,v) = v$
\item[(c)] the foliation $\tilf_{\infty}$ is ``horizontal'', i.e. is represented by the vector field $\partial /\partial u$.
\end{itemize}
Denote by $\da$ the order of $\textsc{P}_1$ as component of $\textsc{P}^{\ast}$ and note that $\alpha$ cannot
exceed~$3$ since $X$ is supposed to be of degree~$2$.
In the above coordinates the vector field $\widetilde{X}$ becomes
\[
\widetilde{X} = \frac{1}{w} \left[ v^{\da} f(u,v) \frac{\partial }{\partial u} + w g(u,v) \frac{\partial}{\partial w} \right] \, .
\]
Besides, modulo taking $p$ sufficiently generic, it can also be assumed that both $f(0,0)$ and $g(0,0)$ are
different from zero.

Let $\tilf$ denote the foliation associated to $\widetilde{X}$.
To complete the proof of the proposition, we are going to show
the local holonomy of $\tilf$ with respect to the invariant axis
$\{u = v = 0\}$ does not coincide with the identity map. This contradicts the assumption that $\widetilde{X}$ is semi-complete
since the restriction of $\widetilde{X}$ to the mentioned axis is the regular (constant) vector field $g(0,0) \partial /\partial x$
(the reader is reminded that $g(0,0) \neq 0$).
In fact, being regular, the integral of the time-form over a loop encircling the origin is zero. If the local holonomy of the
mentioned invariant axis is not trivial, then this loop lifts into an open path in a nearby leaf $L$. Since the intrinsic distance
in $L$ between the endpoints of this open path is bounded below by a positive constant, it follows easily the existence of an open
path $c \subset L$ over which the integral of the corresponding time-form $dT_L$ vanishes what is impossible for a semi-complete
vector field, cf. Section~2.2.

To compute the local holonomy map associated to $\{u = v = 0\}$ with respect to the foliation $\tilf$, let
us consider the loop given by $w(t) = e^{2\pi i t}$. Set $h(x,y) = f(x,y)/g(x,y)$ and note that $h$ is holomorphic
about $(0,0)$ with $h(0,0) \neq 0$
since both $f(0,0)$ and $g(0,0)$ are supposed to be different from zero. Now, it
follows that $(u(t) ,v(t))$ satisfies the differential equation
\[
\begin{cases}
\frac{du}{dt} = \frac{du}{dw} \frac{dw}{dt} = 2\pi i v^{\da} h(u,v) \, , \\
\frac{dv}{dt} = \frac{dv}{dw} \frac{dw}{dt} = 0
\end{cases} \, .
\]
Next, by setting
\[
u(t) = \sum_{j,k} a_{jk}(t) u_0^j v_0^k \quad \text{and} \quad v(t) = \sum_{j,k} b_{jk}(t) u_0^j v_0^k \, ,
\]
the equation $dv /dt = 0$ implies that $b_{01}(t)$ is constant equal to~$1$ and that $b_{jk}(t)$ vanishes identically
for every pair $(j,k) \ne (0,1)$. This is equivalent to saying that $v(t) = v_0$ for all $t$.

Consider now the second equation. Let $h(u,v) = \sum_{n,m} h_{nm} u^n v^m$. From the equation
$du/dt = 2\pi i v^{\da} h(u,v)$, we conclude that
$$
\sum_{j,k} a_{jk}^{\prime}(t) u_0^j v_0^k = 2\pi i v_0^{\da} \sum_{n,m} h_{nm} \left( \sum a_{jk}(t) u_0^j v_0^k \right)^n y_0^m
$$
which, in turn, is equivalent to
$$
\sum_{j,k} a_{jk}^{\prime}(t) u_0^j v_0^k = \sum_{n,m} 2\pi i h_{nm} \left( \sum_{j,k} a_{jk}(t) u_0^j v_0^k \right)^n v_0^{m+\da} \, .
$$
Comparing the coefficient of the monomial $v_0^{\da}$ in both right and left hand sides of the preceding equation, it follows that
\[
a_{0\da}^{\prime}(T) = 2\pi i h_{00} \, ,
\]
where $h_{00} =h(0,0) \neq 0$. Since $a_{0\da}(0)$ equals zero, we obtain that $a_{0\da}(t) = 2\pi i h_{00} t$ and, therefore,
\[
a_{0\da}(1) = 2\pi i h_{00} \ne 0 \, .
\]
Hence the holonomy map $(u_0,v_0) \mapsto (u(1) ,v(1))$ does not coincide with the identity since $u(1)$ is not independent
of $v_0$. This yields the desired contradiction and ends the proof of Proposition~\ref{aproposition2}.
\end{proof}

Unlike the time-form, the {\it renormalized time-form} is defined for every regular leaf of the foliation whether or not the leaf
is contained in the zero/pole divisor of $X$. Propositions~\ref{aproposition} and~\ref{aproposition2} then imply that
the ``renormalized time-form'' can be defined over every leaf $L_{\infty} \subseteq \Delta_{\infty}$ provided that $X$
satisfies the conditions in the preceding statements. In view of this, and unless otherwise stated, throughout the rest of
this paper we shall assume the following.

\noindent {\bf General assumption}: No irreducible component of the curve $\{ \textsc{P}^{\ast} =0\} \cap \Delta_{\infty}$ is
invariant by $\tilf_{\infty}$.

As pointed out above, under the semi-complete assumption, $\{ \textsc{P}^{\ast} =0\} \cap \Delta_{\infty}$ is not
invariant for the foliation induced by the vector field $X$. Nonetheless, for semi-complete vector fields,
a lot more can be said about this (non-$\tilf_{\infty}$-invariant) curve.
In particular, when $X$ is a homogeneous semi-complete
vector field with degree at least~$3$, the proof of Proposition~\ref{aproposition} also yields:

\begin{prop}\label{order1P}
Assume that $X$ is a homogeneous polynomial semi-complete vector field with degree $d \geq 3$. Let $\tXX$ be
as in Equation~(\ref{campo}) and assume also that $\textsc{P}^{\ast}$ is not invertible. Then every non-trivial irreducible
component of $\textsc{P}^{\ast}$ has order~$1$. Furthermore, every non-trivial irreducible
component of $\textsc{P}^{\ast}$ must also appear as an irreducible component of $ \overline{\textsc{P}}$.
\end{prop}

\begin{proof}
Assume that $\textsc{P}^{\ast}$ is not invertible. Let $\textsc{P}_1$ be a non-trivial irreducible component of
$\textsc{P}^{\ast}$ and denote by $m \geq 1$ the order of $\textsc{P}_1$ w.r.t $\textsc{P}^{\ast}$. Since $\{\textsc{P}_1
=0\} \cap \Delta_{\infty}$ is not invariant by $\tilf_{\infty}$, at a generic point of $\{\textsc{P}_1
=0\} \cap \Delta_{\infty}$ this curve is transverse to $\tilf_{\infty}$. A generic point $p$ can be chosen so that,
in addition, neither $H$ nor any other irreducible component of $\textsc{P}^{\ast}$ vanishes at $p$. Finally, we can also
suppose that $\{\textsc{P}_1 = 0\}$ is smooth at $p$. Next, consider local coordinates $(u,v,w)$ about $p$, $\{ w=0\} \subset
\Delta_{\infty}$, satisfying the following conditions:
\begin{itemize}
\item $p \simeq (0,0,0)$.
\item $\textsc{P}_1(u,v) = v$.
\item The foliation $\tilf_{\infty}$ is locally represented by the vector field $\partial /\partial v$.
\end{itemize}
In other words, we have chosen coordinates $(u,v,w)$ where $\tXX$ takes the form
\[
\tXX = w^{1-d} \overline{\textsc{P}}(u,v) \left[ v^m g(u,v) \frac{\partial}{\partial v} + w h(u,v) \frac{\partial}{\partial w} \right]
\]
where both $g(0,0) , \, h(0,0)$ are different from {\it zero}.

Suppose that $m > 1$. Then, the first non-zero homogeneous component of $\tXX$ at the origin (identified to $p$) is given by
\[
\tXX^H = \dl v^k w^{2 - d} \frac{\partial}{\partial w}
\]
where $\dl = H(0,0)$ and $k$ is the order of $\textsc{P}_1(u,v) = v$ w.r.t. $\overline{\textsc{P}}$. Since $d \geq 3$,
the restriction of $\tXX^H$ to the invariant planes $\{v = {\rm cte}\}$
is not semi-complete provided that ${\rm cte}$ is different from
zero. This contradicts the semi-complete assumption over $X$, cf. Section~2.2. It then follows that $m=1$.

From now on we have $m = 1$. It remains to prove that $k$ must be strictly positive. Thus, let us assume for a
contradiction that $k = 0$. The first
non-zero homogeneous component of $\tXX$ at the origin (identified to $p$) is hence given by
\[
\tXX^H = w^{1-d} \left[ \da v \frac{\partial }{\partial v} + \dl w \frac{\partial }{\partial w} \right]
\]
for some constant $\da \in \C^{\ast}$ and where $\dl$ equals $H(0,0)$. Now, the restriction of the foliation associated to
$\tXX^H$ to the invariant hyperplane $\{v = 0\}$ is given by $w^{2-d} \partz$ which is not semi-complete provided that
$d \geq 3$, cf. Section~2.2. The proposition is proved.
\end{proof}

As an immediate consequence, we have the following:

\begin{coro}\label{trivializingP}
Assume that $X$ is a homogeneous polynomial semi-complete vector field with degree greater than or equal to~$3$. Suppose
also that the singular set of $X$ has codimension least~$2$. Then $\textsc{P}^{\ast}$ is invertible, i.e. $\textsc{P}^{\ast}$ is
a constant.
\end{coro}

\begin{proof}
Assume that $\textsc{P}^{\ast}$ is not invertible and consider a non-trivial irreducible component $\textsc{P}_1$ of
$\textsc{P}^{\ast}$. According to Proposition~\ref{order1P}, $\textsc{P}_1$ must also be an irreducible component of
$\overline{\textsc{P}}$. This immediately implies that the singular set of $X$ has codimension~$2$ and the statement follows.
\end{proof}

Concerning the case of homogeneous polynomial vector fields of degree $d=2$, the following holds:

\begin{prop}
Suppose now that $X$ is a homogeneous quadratic semi-complete vector field. Suppose that $\textsc{P}^{\ast}$
is not invertible. Then every non-trivial irreducible component of $\textsc{P}^{\ast}$ has order~$1$.
\end{prop}

\begin{proof}
Let $\textsc{P}_1$ be a non-trivial irreducible component of $\textsc{P}^{\ast}$ and denote by $m$ the order of
$\textsc{P}_1$ w.r.t $\textsc{P}^{\ast}$. Owing to the general assumption (in turn justified by Propositions~\ref{aproposition}
and~\ref{aproposition2}), the algebraic curve $\{\textsc{P}_1 = 0\} \cap \Delta_{\infty}$
is not invariant by $\tilf_{\infty}$ (cf. Proposition~\ref{aproposition2}). This means that this
algebraic curve is transverse to $\tilf_{\infty}$ at a generic point of it. Moreover, about a sufficiently generic
point of $\{\textsc{P}_1 = 0\} \cap \Delta_{\infty}$, there are local coordinates $(u,v,w)$ where $\tXX$ becomes
\[
\tXX = w^{1-d} \overline{\textsc{P}} \left[ v^m g(u,v) \frac{\partial}{\partial v} + w h(u,v) \frac{\partial}{\partial w} \right]
\]
with both $g(0,0) , \,  h(0,0)$ different from {\it zero}. In particular, for fixed $u$, the
``$2$-dimensional vector field'' $v^m g(u,v) \partial /\partial v +
w h(u,v) \partial /\partial w$ has exactly one eigenvalue different from zero at the origin. By means of an elementary and well-known
calculation, this fact ensures that the local holonomy arising from the axis $\{ v=0\}$ cannot coincide with the identity. Therefore
the corresponding vector field cannot be semi-complete since its restriction to the mentioned axis is regular at $\{ v=w=0\}$
(details are as in the proof of Proposition~\ref{aproposition2}).
\end{proof}

Unlikely the case of homogeneous polynomial semi-complete vector fields with degree $\geq 3$, the set of homogeneous
polynomial semi-complete vector fields of degree~$2$ admitting an irreducible component (not invariant by $\fol$ and)
contained in its singular set is not empty, even when the singular set of $X$ is supposed to have codimension at least~$2$.
Indeed, the homogenous vector field
\[
X = xz \frac{\partial}{\partial x} + (2yz + y^2) \frac{\partial}{\partial y} + z^2 \frac{\partial}{\partial z}
\]
is a semi-complete vector field and induces a foliation on the
hyperplane at infinity admitting a non-invariant curve contained
in the singular set of the foliation associated to $X$. In fact, in the standard affine coordinates $(x,y,z)$
of Section~$3$, $\tXX$ is given by
\[
\tXX = \frac{1}{z} \left[ y(1+y) \frac{\partial}{\partial y} - z \frac{\partial}{\partial z} \right] \, .
\]

Although the set of homogeneous polynomial semi-complete vector fields of degree $d=2$ admitting a non-trivial irreducible
component (not invariant by $\fol$ and) contained in the singular set of $\fol$ is not empty, the corresponding vector fields
will be ruled out from our discussion. In fact, as far as the main results presented in the Introduction are concerned,
whenever the singular set of a foliation
plays a specific role, this singular set is supposed to consist only of
simple singularities (in the sense described in the Introduction). However, simple singularities in
this sense are not compatible with the presence of curves of singular points contained in the hyperplane of infinity for the
foliation in question. Alternatively, it should be noted that the foliations associated to the above non-empty set of homogeneous
polynomial vector fields can be easily described. In fact, we have:

\begin{lema}
Let $X$ be a homogeneous quadratic vector field and let $\textsc{P}^{\ast}$ be as above. If $\textsc{P}^{\ast}$
is not invertible, then the foliation $\tilf_{\infty}$ is induced by a vector field of degree~$0$ or~$1$.\qed
\end{lema}

After the preceding lemma, the case where $d=2$ and $\textsc{P}^{\ast}$ is not constant can directly be treated and details
will be left to the reader.

After this long detour, let us now go back to the real $1$-dimensional foliation $\mathcal{H}$. The rest of this section
is devoted to establishing Theorem~\ref{introducelabel4} which, in fact, makes no assumption on whether or not the
corresponding vector field $X$ is semi-complete. This theorem will further be extended in the next section and the final
result, when combined to the preceding material about semi-complete vector fields, will provide us with the required quantitative
information to prove the main results stated in the Introduction.

Consider then a homogeneous polynomial vector field $X$ of degree $d \geq 2$ that is not a multiple of the radial vector field.
Whether or not $X$ is semi-complete, we can consider the vector field $\tXX$ on $M$ along with its associated foliation
$\tilf$ and the induced foliation $\tilf_{\infty}$ on $\Delta_{\infty}$.
Next, fix a regular leaf $L_{\infty} \subseteq \Delta_{\infty}$
of $\tilf$ and let $S = \mathcal{P}^{-1}_{\infty}(L_{\infty})$ be the cone over $L_{\infty}$. Denote by $\mathcal{H}$ the
oriented $1$-dimensional real foliation induced by the Abelian form $\omega_1$ (cf.
Section~\ref{sechom}). It is also useful to consider other foliations similar to
$\mathcal{H}$. For this let us consider an angle $\theta \in (-\pi/2, \pi/2)$. Denote by
$\mathcal{H}^{\theta}$ the oriented foliation whose (oriented) trajectories make an
angle of $\theta$ with the (oriented) trajectories of $\mathcal{H}$. It is clear that
these foliations are well-defined under the same conditions that $\mathcal{H}$. It is
also clear that the holonomy maps of $\tilf_S$ obtained over the trajectories of
$\mathcal{H}^{\theta}$ are still contractions as in Formula~(\ref{contr}) (up to a multiplicative
constant). In the sequel we denote by $l^{\theta}$ an oriented trajectory of
$\mathcal{H}^{\theta}$.

Given (a segment of) a trajectory $l_p$ of $\mathcal{H}$ (resp. $l_p^{\theta}$ of
$\mathcal{H}^{\theta}$), we are interested in the value of the integral $\int_{l_p} dT$
(resp. $\int_{l_p^{\theta}} dT$).
To investigate the behavior of this integral, it is clear that the singularities of
$\tilf$ on $\Delta_{\infty}$ will pose further difficulties. Thus it is natural to begin with (segments
of) trajectories of $\mathcal{H}$ (resp. $\mathcal{H}^{\theta}$) that remain away from
the corresponding singular set. In order to do it, let $W$ be a sufficiently small open neighborhood
of the singular set of $\tilf$ on $\Delta_{\infty}$. Let $l_p$ (resp. $l_p^{\theta}$) be
(a segment of) a trajectory of $\calh$ (resp. $\calh^{\theta}$). We can now state one of our main results.
Despite our $3$-dimensional setting, the reader will immediately
check that this result holds in arbitrary dimensions (as it is always the case in the present section).

\begin{teo}
\label{introducelabel4}
Suppose that $l_p$ (resp. $l_p^{\theta}$) is contained in $\Delta_{\infty} \setminus W$. Then
$\int_{l_q} dT$ (resp. $\int_{l_q^{\theta}} dT$) converges for all $q = (p,z_0) \in \mathcal{P}^{-1}_{\infty}(p,0)$,
where $l_q$ (resp. $l_q^{\theta}$) denotes the lift of $l_p$ (resp. $l_p^{\theta}$) to the leaf of
$\tilf$ through $q$ and where $dT$ stands for the time form associated to $\tXX$. More precisely, assuming $W$ fixed,
and assuming that a trajectory $l_q^{\theta}$ of $\mathcal{H}^{\theta}$ does not intersect $W$,
there exists a constant $C$ (varying continuously with $\theta$) such that for every path $c :[0,1] \rightarrow L$,
$c(0) =q$, with image contained in $l_q^{\theta}$, we have
$$
\left|\int_c dT\right| \leq \int_c |dT| \leq \int_{l_q^{\theta}} |dT| < C \vert z_0 \vert^{d-1} \, ,
$$
where $d \geq 2$ stands for the degree of the initial homogeneous vector field $X$.
\end{teo}

\begin{proof}
It suffices to prove the statement for the case of a trajectory $l_p$ of $\calh$ since the adaptations
needed for trajectories of $\calh^{\theta}$ are clear. Also, we can suppose without loss of generality
that the length of $l_p$ is infinite. Finally, we recall that the affine coordinates $(x,y)$ are as in Section~3, namely
$\Delta_{\infty}^{(x,y)}$ is not invariant by $\tilf_{\infty}$
and $\Delta_{\infty}^{(x,y)}$ contains no singularities of $\tilf_{\infty}$.

Let $W$ be the previously chosen open neighborhood of the intersection of $\Delta_{\infty} \setminus W$
with the singular set of $\tilf$. Assume that $l_p$ is connected and totally
contained in $\Delta_{\infty} \setminus W$. Since the intersection of $\Delta_{\infty} \setminus W$ with the
singular set of $\tilf$ is empty and the length of $l_p$ is infinite, the only singularities of $\mathcal{H}$
that may be met by the trajectory $l_p$ are saddle singularities of $\calh$ (occurring at regular points of $\tilf$).
However, according to Definition~\ref{justadddefinition2}, the corresponding trajectories of
$\mathcal{H}$ are continued by following separatrices that are ``emanated'' from
the saddle in question. Moreover, with this global definition of $\calh$-trajectory, the uniform contractive character
of the corresponding holonomy maps is still verified.

Naturally an oriented trajectory of $\calh$ cannot intersect the polar divisor of $\omega_1$ since the latter
is constituted by ``source-like'' singularities for $\mathcal{H}$. Away from $W$, this polar divisor coincides with
the ``line at infinity $\Delta_{\infty}^{(x,y)} \subset \Delta_{\infty}$. Though a trajectory of $\calh$ may come ``close''
to $\Delta_{\infty}^{(x,y)}$, owing to
Lemma~\ref{remainingcompact} we know that every sufficiently long segment of $l_p$ has ``most of its length''
contained in a fixed compact part of the affine $\C^2$ associated to the coordinates $(x,y)$. Let then a compact part $K$ of
the mentioned affine copy of $\C^2$ be fixed. Since $F$ is clearly bounded
on $K$, the estimates of Lemma~\ref{remainingcompact}
allow us to conclude the following: every sufficiently long segment $c_p$ of $l_p$
can be split into a concatenation $c_1 \ast c_2 \ast \cdots \ast c_k$ such that:
\begin{enumerate}
\item The image of $c_i$, for $i$ odd, is contained in the compact set $K$. Besides at points belonging to these segments
the absolute value of $\omega_1$ is bounded from below, i.e. $|\omega_1| \geq \da > 0$.

\item If $i_0$ is odd, then the sum of the lengths of all even $i$'s, $i <i_0$, is less than, say,
$2/3$ the sum of the lengths of $c_1, \ldots ,c_{i_0}$.

\item The absolute value of the coordinate ``$z$'' decreases monotonically over the segment $c_p$.

\end{enumerate}

Fix $q \in \mathcal{P}^{-1}_{\infty}(p)$ and let $L$ be the leaf
through $q$. Consider the lift of $l_p$ to $L$ and denote it by
$l_q$. Note that $l_q$ is an oriented trajectory of $\mathcal{H}$
over $L$. We want to express $l_q$ in the corresponding affine coordinates $(x,y,z)$. More precisely, our goal will be
to compute the value of its last coordinate $z$. For this, consider a
connected oriented path $c$ contained in $l_p$ and joining $p$ to
another point of $l_p$. Consider also a lift of $c$ contained in
$l_q$. The $z$-coordinate of the mentioned lift is given by
$z = z_0 \exp[-\int_c \omega_1]$
where $z_0$ is the $z$-coordinate of $q$. In other words, $z_0$ is
the ``height" of $q$ relatively to $L_{\infty}$. In particular
\begin{eqnarray*}
|z| &=& \left|z_0 e^{-\int_c \omega_1}\right| = |z_0| e^{-{\rm Re}
\int_c \omega_1} = |z_0| e^{-\int_0^1 {\rm Re} (\omega_1(c(t)).
c^{\prime}(t)) dt}\\
&=& |z_0| e^{-\int_0^1 |(\omega_1(c(t)). c^{\prime}(t)| dt} \leq
|z_0| e^{-\int_0^1 \da |c^{\prime}(t)| dt /3} = |z_0| e^{-\da {\rm
length}(c) /3}
\end{eqnarray*}
This estimate shows that, whenever a segment of $l_p$
having length equal to $3\ln(2)/ 2\da$ is lifted in a regular leaf of $\tilf$
projecting over $L_{\infty}$, the height of the final point of the lift in
question is at most $1/2$ of the height of its initial point.

Now the integral $\int_{l_q} dT$ can be estimated as follows. The
time-form on $L$ is given, in local coordinates, by
$dT = z^{d-1}dx/F(x,y)$.
Since $l_p$, the projection of $l_q$ by $\mathcal{P}_{\infty}$, is
contained on a compact set not intersecting the singular set of
$\tilf_{\infty}$, the absolute value of $F(x,y)$ is bounded from
below, i.e. $|F(x,y)| \geq \db > 0, \qquad {\rm for \;\; \; all } \; \; \; (x, y)
\in \Delta_{\infty} \setminus W$.
Otherwise we replace $F$ by $G$ (recall that we are dealing only
with regular points of $\tilf$ on $\Delta_{\infty}$). Hence, considering $l_q$
as the concatenation of segments having length equal to $3\ln(2)
/2 \da$, $l_q = \sum_{i=0}^{\infty}l_{i,q}$, it follows that
\begin{eqnarray*}
\left| \int_{l_q} dT \right| &=& \left| \sum_{i=0}^{\infty} \int_{l_{i,q}} \frac{z^{d-1}}{F(x, y)} dx \right|
\leq  \sum_{i=0}^{\infty} \left| \int_0^1 \frac{z_{i,q}^{d-1}(t)}{F(x_{i,q}(t),y_{i,q}(t))} x^{\prime}_{i,q}(t) dt \right|\\
&\leq&  \sum_{i=0}^{\infty} \int_0^1 \frac{|z_{i,q}(t)|^{d-1}}{|F(x_{i,q}(t),y_{i,q}(t))|} |x^{\prime}_{i,q}(t)| dt
\leq  \sum_{i=0}^{\infty} \int_0^1 \frac{|z_0|^{d-1}(\frac{1}{2})^{i(d-1)}}{\db} |l^{\prime}_{i,p}(t)| dt\\
&\leq&  \frac{|z_0|^{d-1}}{\db} {\rm length }(l_{i,p}) \sum_{i=0}^{\infty}\left( \frac{1}{2^{d-1}} \right)^{i}
=  \frac{3|z_0|^{d-1} {\rm ln}(2)}{2\da \db} \frac{1}{1 - (\frac{1}{2})^{d-1}} < \infty
\end{eqnarray*}
where $l_{i,q}(t) = (x_{i,q}(t), y_{i,q}(t), z_{i,q}(t))$, $t \in
[0,1]$, is such that $l_q = \sum_{i=0}^{\infty} l_{i,q}$ and
$\mathcal{P}_{\infty}(l_{i,q}) = l_{i,p}$. The theorem follows.
\end{proof}

What precedes shows that the above mentioned integral is, indeed, bounded on $\Delta_{\infty} \setminus W$.
Our next goal will be to get rid of the condition on $W$, i.e. we want to allow the trajectory $l_p$ (resp.
$l_p^{\theta}$) to accumulate on the singular set of $\tilf$ in $\Delta_{\infty}$. This will lead us to study
the behavior of this integral over segments of trajectories of $\calh$ (resp. $\calh^{\theta}$) that are close
to the singularities of $\tilf$. This local analysis will be the object of the Section~5. Nonetheless, to finish
the current section, let us give an elementary general result concerning trajectories of $\calh, \, \calh^{\theta}$
that are contained in a local separatrix for a singularity of $\tilf, \, \tilf_{\infty}$ in the particular case where
{\it $\tilf$ is associated to a semi-complete vector field $\tXX$}\, verifying also the preceding conditions. This goes as follows.

Consider again a vector field $\tXX$ as in Equation~(\ref{campo_forma_usual}). Let $p \in \Delta_{\infty}$
be a singular point of $\tilf$ and consider a (germ of) analytic curve $Sep \subset \Delta_{\infty}$ passing
through $p$, invariant by $\tilf_{\infty}$ and not entirely contained in the singular set of $\tilf$. Besides, let $\gamma
(t)$ denote a local, irreducible, Puiseaux parametrization for $Sep$ defined on a neighborhood of $0 \in \C$. Denote by
$f(t) \frac{\partial}{\partial t}$ the pull-back of the restriction of the vector field $F(x,y) \partial /\partial x
+ G (x,y) \partial /\partial y$ to $Sep$ by $\gamma$. Denote also by $h = h(t)$ the function $t \mapsto H
\circ \gamma (t)$. Then, the pull-back of the restriction of $\tXX$ to the cone over $Sep$ is given by
\[
\tXX_S = z^{1-d} \left[ f(t) \frac{\partial}{\partial t} + zh(t) \frac{\partial}{\partial z} \right] \, .
\]
Denote by $k$ (resp. $l$) the order of $f$ (resp. $h$) at $0 \in \C$. Now we have:

\begin{lema}
\label{separatrixndimensions}
Assume that $\tXX$, as in Equation~(\ref{campo_forma_usual}), is semi-complete and consider
the vector field $\tXX_S$ along with integers $k,l$ as above. Then $l \geq k-1$ and the nature of $\omega_1$ (restricted
to $Sep$) at $p$ is determined by the relation between $k$ and $l$. More precisely, the following holds.
\begin{itemize}
\item If $l > k$ then $\omega_1$ is holomorphic and the restriction of $\calh$ to $Sep$ has a saddle singularity at $p$
with $2m$ separatrices (for a certain $m \geq 1$).

\item If $l = k$ then $\omega_1$ is regular at $p$ (and, in particular, holomorphic).

\item If $l=k-1$ then $\omega_1$ has a simple pole at $p$. The residue of this pole is equal to $\da = -(h/f^{\prime})(0)$.
Then the restriction of $\calh$ to $Sep$ has a sink (resp. source) at $p$ provided that $\da \in \R_+$ (resp. $\da \in \R_-$).

\end{itemize}
\end{lema}

As explained in \cite{guillotreb}, cf. also the beginning of the present section, the vector field $\tXX_S$ induces an affine
structure on $\{ z=0\}$. According to \cite{guillotreb}, Section~3, this affine structure can be compared to the standard
euclidean structure to yield a $1$-form $\beta$ called the {\it affine defect}\, of the mentioned affine structure. In the
present case, the $1$-form $\beta$ is simply
\begin{equation}
\beta =  \left( \frac{-f'}{f} + (d-1) \frac{h}{f} \right) \, dt \, . \label{formbeta}
\end{equation}

\begin{proof}[Proof of Lemma~\ref{separatrixndimensions}]
The argument given here relies on part of the theory developed in \cite{guillotreb}. Keep the preceding notations
and suppose that $X$ is semi-complete. Since $X$ is semi-complete, the affine
structure induced by $X$ on $\{ z=0\}$ is {\it uniformizable}, cf. \cite{guillotreb}. In turn, according to Proposition~6
of \cite{guillotreb}, the uniformizable character of the affine structure in question ensures that $\beta$ has at
most a simple pole at $t=0$. Furthermore the residue associated to this simple pole has the form $-1+1/n$ where
$n \in \Z^{\ast} \cup \{ \infty \}$. We also point out here that, when $n \in \Z^{\ast}$, then the ``Fundamental Lemma''
proven in \cite{guillotreb} ensures that the local holonomy map associated to the invariant axis $\{ z=0\}$ has finite
order dividing~$n$.

Now assume that $\beta$ has only simple poles and note that the poles of $f'/f$ are necessarily simple as well.
In this case, we conclude that also $h/f$ can have at worst simple poles. In other words, we have proved that
$l \geq k-1$ provided that $X$ is semi-complete.

Assume now that $l=k-1$ and denote by $\alpha \neq 0$ the residue of the (simple) pole of $h/f$ at $t=0$. It was
seen that the residue of $\beta$ at $t=0$, if not zero, has the form $-1+ 1/n$ and from this it follows that
$\alpha =1/n$. In particular $\alpha$ must be real provided that $X$ is semi-complete. On the other hand,
$\alpha$ is also the residue of $\omega_1$ at $t=0$. In particular, if $\alpha  >0$ (resp. $\alpha <0$) then
$\calh$ has a sink (resp. source) at $t=0$.

To complete the proof of the lemma, let us suppose that $l\geq k$. It follows at once from the definition of
$\omega_1$ that this form is holomorphic and non-zero at $t=0$ provided that $l=k$. Similarly, if $l >k$, then
$\omega_1$ is still holomorphic at $t=0$ however, in this case, $t=0$ constitutes a {\it zero}\, of $\omega_1$.
The corresponding consequences for the local behavior of $\calh$ having already been known, the proof of the
lemma is completed.
\end{proof}

\begin{obs}
{\rm In the preceding argument, it should be emphasized that the quotient $h/f^{\prime}$ at $0$ only must
belong to $\R$ (in fact, to $\Z^{\ast}$) because $X$ is supposed to be semi-complete. Indeed, the residue of a simple
pole for the $1$-form $\beta$ need not be $-1 +1/n$, with $n \in \Z^{\ast} \cup \{ \infty\}$ unless the affine
structure giving rise to $\beta$ is uniformizable.

If the assumption of having a semi-complete vector field $X$ is dropped, then the foliations $\calh, \, \calh^{\theta}$
may also admit singular points that behave as ``centers'' at the points corresponding to $t=z=0$ in the previous local coordinates
$(t,z)$. This would add to the list of sink, source and saddle singularities.}
\end{obs}

In closing let us point out again that the preceding statements hold in arbitrary dimensions despite the fact that we have
chosen to emphasize the $3$-dimensional case. Details are left to the reader.

\section{Structure of $\mathcal{H}$ near singular points of $\tilf_{\infty}$}

This section is devoted to establishing an extension of Theorem~\ref{introducelabel4} allowing the trajectories
of $\calh, \, \calh^{\theta}$ to accumulate on singularities of $\tilf, \, \tilf_{\infty}$. These singularities, however, will be
supposed to have a sort of ``simple behavior''. Here, it should be noted that the assumptions made on the structure
of the singularities in question are not superfluous since certain ``saddle-node'' singularities of nature
different from those considered in the Introduction
give rise to new complications preventing us from generalizing Theorem~\ref{introducelabel4} without further information.
On the other hand, the main result of this section, namely Theorem~\ref{maintheo} below, remains valid under a large
class of singular points, cf. the comments at the end of the section.
As in Sections~3 and~4, we still keep notations emphasizing the $3$-dimensional case. The extensions
of the arguments to higher dimensions, however, does not pose further difficulties.

Throughout this section we shall deal with a homogeneous semi-complete vector field $X$ on $\C^3$ which, in addition, is supposed to have
a singular set of codimension at least~$2$. Besides the degree~$d$ of $X$ is supposed to verify $d \geq 2$.

Let us consider again the foliation $\tilf$ associated to a homogeneous (polynomial) semi-complete vector field $X$ on
$\C^3$ and assume that the singularities of $\tilf$ lying in $\Delta_{\infty}$ are simple in the sense stated in the
Introduction. Recall that the foliation $\tilf$ is tangent to the vector field $\tXX$ obtained from the initial vector field~$X$
and given in the affine coordinates $(x,y,z)$ of Section~3 by Equation~(\ref{campo_forma_usual}). Namely, we have
\begin{equation}
\tXX = \frac{1}{z^{d-1}} \left[ F(x,y) \frac{\partial}{\partial x} + G(x,y) \frac{\partial}{\partial y} +
z H(x,y) \frac{\partial}{\partial z} \right] \, . \label{xyzglobalcoordinates1}
\end{equation}
Suppose that $p \in \Delta_{\infty}$ is a singular point of $\tilf$ and consider local coordinates $(u,v,w)$ about $p$, with
$w$ locally equal to~$z$. In these coordinates, a local representative for the foliation $\tilf$ is provided by
the vector field $Y$ having the form
\begin{equation}
Y = \overline{F} (u,v) \frac{\partial}{\partial u} + \overline{G} (u,v) \frac{\partial}{\partial v} +
w \overline{H} (u,v) \frac{\partial}{\partial w} \label{uvwlocalcoordinates1}
\end{equation}
for certain holomorphic functions $\overline{F}, \overline{G}$ and $\overline{H}$ without non-trivial common factors,
where $p$ is identified to the origin of $\C^3$. The singular point $p$ is then
said to be {\it simple}\, (in the sense described in the Introduction) if the linear part
of $Z = \overline{F} \partial /\partial u + \overline{G} \partial /\partial v$ at $(0,0) \in \C^2$ possesses two
eigenvalues different from zero. Moreover, in case the quotient of these eigenvalues happens to be a positive integer,
then the induced foliation on $\Delta_{\infty}$
is linearizable (in other words, it is not conjugate to its Poincar\'e-Dulac normal form, cf. for example \cite{arnold}).
With these assumptions, Theorem~\ref{introducelabel4} admits the following extension:

\begin{teo}\label{maintheo}
Let $X$ be a homogenous polynomial semi-complete vector field whose singular set has codimension~$\geq 2$. Assume
that all the singularities of $\tilf$ are simple (in the sense indicated in the Introduction). Suppose that there
is $\theta \in (-\pi/2, \pi/2)$ and a point $\textsc{P}\in \Delta_{\infty}$ such that the trajectory $l_{\textsc{P}}^{\theta}$
of $\mathcal{H}^{\theta}$ through $\textsc{P}$ has infinite length. Then $\int_{l_q} dT$ converges for all $q \in
\mathcal{P}^{-1}_{\infty}(\textsc{P})$, where $l_q$ denotes the lift of $l_{\textsc{P}}$ to the leaf through $q$ and $dT$ is
the time form associated to $\tXX$.
\end{teo}

This section is entirely devoted to the proof of Theorem~\ref{maintheo}. Applications of Theorem~\ref{maintheo},
along with Theorem~\ref{introducelabel4}, will be worked out in Sections~6 and~7.
Before proceeding further, let us first revisit the statement of Theorem~\ref{maintheo} to make its assumptions clear.

First, the singular set of $X$ has codimension $\geq 2$. This implies that $\overline{P}$ as in
Equation~(\ref{campo}) is invertible. It can therefore be assumed constant equal to~$1$. Thus, in the case
where $X$ has degree $d \geq 3$, $P^{\ast}$ is invertible as well (cf. Corollary~\ref{trivializingP}). The
same does not necessarily occur for homogeneous (polynomial) vector fields of degree~$d=2$. Nonetheless,
since we are assuming the singular points of $\tilf$ on $\Delta_{\infty}$ to be simple, $P^{\ast}$ can
also be supposed invertible (and therefore constant) even for $d=2$. So, from now on, the greatest common
divisor between $F$ and $G$ in Equation~(\ref{campo}) is assumed to be~$1$.

To begin with, fix a point $p \in \Delta_{\infty}$ contained in the singular set of $\tilf$. Recall that the
two eigenvalues of $\tilf_{\infty}$ at $p$ are supposed to be different from zero (and, in case they are of
the form $1,N$ with $N \in \mathbb{Z}_+$, $\tilf_{\infty}$ is supposed not to be conjugate to its Poincar\'e-Dulac
normal form). To be more precise, let $Y$ be the vector field in~(\ref{uvwlocalcoordinates1}) with singular set
of codimension at least~$2$ and tangent to $\tilf$. The conditions on the singularities of $\tilf, \, \tilf_{\infty}$
imply that ${\rm g.c.d.}\, (\overline{F},\overline{G}) = 1$. They also imply that the vector field
$Z = \overline{F} \partial /\partial u + \overline{G} \partial /\partial v$ has eigenvalues
$\dl_1, \, \dl_2$ at $(0,0) \simeq p$ with $\dl_1 \dl_2 \neq 0$. Furthermore if $\dl_1, \, \dl_2$ is of the form $1,N$ with
$N \in \mathbb{Z}_+$, then $Z$ is linearizable (recalling that a non-linearizable vector field verifying the
preceding conditions must be conjugate to $(Nx + y^N) \partial /\partial x + y \partial /\partial y$). This
summarizes the assumption of Theorem~\ref{maintheo}.

Comparing the expressions for the vector fields $\tXX$ and $Y$, respectively given in Formulas~(\ref{xyzglobalcoordinates1})
and~(\ref{uvwlocalcoordinates1}), it immediately follows that $H (p) = \overline{H} (0,0)$.

Now we state the following:

\begin{lema}
\label{thislabelwasmissing}
Fix a separatrix $Sep$ for $\tilf_{\infty}$ at a (simple) singular point $p_0 \in \Delta_{\infty}$.
Assume that $H (p) = \overline{H} (0,0) = 0$. Then the Abelian form $\omega_1$ on the cone over $Sep$ is holomorphic.
\end{lema}

\begin{proof}
Note that the above mentioned vector field $Z$ representing $\tilf_{\infty}$ about $p$ has, by assumption,
a linear part with two eigenvalues $\dl_1, \,  \dl_2$ different from {\it zero}. Suppose that $Sep$ is a (possibly
singular) irreducible local separatrix for $\tilf_{\infty}$ at $p$ and denote by $\gamma$ an irreducible
Puiseaux parametrization for $Sep$. Since $\dl_1 \dl_2 \neq 0$, it is immediate to check that the order $k$
at $0 \in \C$ of the one-dimensional vector field obtained by pulling-back the restriction of the
vector field $Z$ to $Sep$ by $\gamma$ equals~$1$. The statement
then results from Lemma~\ref{separatrixndimensions}.
\end{proof}

It follows from Lemma~\ref{thislabelwasmissing} that $p$ is either a regular point
or a saddle singularity for $\calh$ provided that $H(p) =0$. In view of the discussion at the end of
Section~3 (about the global definition of $\calh$-trajectories), singular points of saddle-type for $\calh$
do not yield endpoints for any trajectory
of $\calh$. Indeed, every trajectory of $\calh$ entering a small neighborhood of the singular point in question
will eventually leave this same neighborhood. Besides, as already shown, from a global point of view every trajectory
of $\calh$ gives rise to a contracting holonomy map in the appropriate sense.

Whereas Lemma~\ref{thislabelwasmissing} does not require the vector field $X$ to be semi-complete, this assumption
definitely plays a role in our next lemma concerning the case $H (p) = \overline{H} (0,0) \neq 0$. The reader will
also note that this lemma already appears in \cite{guillotFourier}.

\begin{lema}\label{hreal}
Suppose that the initial degree~$d$ homogeneous polynomial vector field $X$ is semi-complete. Suppose also that
$H (p) = \overline{H} (0,0) \neq 0$. Then $d=2$. Furthermore the ratios $\dl_1 /\overline{H} (0,0)$ and $\dl_2 /\overline{H} (0,0)$
are both integers numbers (and hence reals).
\end{lema}

\begin{proof}
Since the singular points of $\tilf$ in $\Delta_{\infty}$ are supposed to be simple, the foliation
$\tilf_{\infty}$ possesses at least two smooth separatrices through $(0,0) \simeq p$. Without loss of
generality, these separatrices may be supposed to coincide with the axes $u, \, v$. To prove that
$\dl_1 /\overline{H} (0,0) \in \Z^{\ast}$, consider the restriction of $\tilf$ to the $2$-plane sitting over
the separatrix $Sep = \{ v=w=0\}$ of $\tilf_{\infty}$. Clearly this $2$-plane is invariant by $\tilf$ and
locally parameterized by the coordinates $u, \, w$. The restriction $\tXX_{\vert}$ of $\tXX$ to the $2$-plane in
question expressed in $(u,  w)$-coordinates is given simply by $\tXX_{\vert} = w^{1-d} [\overline{f} (u)
\partial /\partial u + w \overline{h} (u) \partial /\partial w ] = w^{-1}  [\overline{f} (u)
\partial /\partial u + w \overline{h} (u) \partial /\partial w ]$, since $d=2$.

On the other hand, the vector field $\tXX_{\vert}$ is semi-complete on a neighborhood of the origin. Furthermore
$\overline{h} (0) = \overline{H} (0,0) \neq 0$ while $\overline{f} (u) = \dl_1 u + \cdots$. In particular
$\overline{f} (0) =0$. From this it follows that the axis $\{ u=0\}$ is invariant by $\tXX_{\vert}$ and that
the restriction of $\tXX_{\vert}$ to the mentioned axis is a regular one-dimensional vector field. This restriction
being regular and $\tXX_{\vert}$ being semi-complete, it follows that the local holonomy map associated to the axis in question
must coincide with the identity (cf. the discussion in the proof of Proposition~\ref{aproposition2}). Since
$\overline{f} (u) = \dl_1 u + \cdots$ an elementary calculation shows that the above mentioned holonomy
map cannot coincide with the identity unless $\dl_1 /\overline{H} (0,0)$ is an integer. The case of
$\dl_2 /\overline{H} (0,0)$ being analogous, the proof of the lemma is over.
\end{proof}

Summarizing, when $X$ is semi-complete, both quotients $\dl_1 /\overline{H}(0,0), \, \dl_2 /\overline{H} (0,0)$ are
non-zero integers. In particular, the quotient of the eigenvalues of $\tilf_{\infty}$ at the singular point $p_0$ is
a rational number since it is given by $\dl_1/\dl_2$. Because we are treating the case $\overline{H} (0,0) \neq 0$ where
the $1$-form $\omega_1$ has a simple pole at the origin ($\simeq p$), there follows the existence of two different
cases  according to whether $\dl_1/\dl_2 \in \Q_+$ or $\dl_1/\dl_2 \in \Q_-$. The first possibility cane easily
be treated.

\begin{lema}\label{quotientinR+}
Let $X$ be as in the statement of Lemma~\ref{hreal}.
With the preceding notations suppose that $\dl_1/\dl_2 \in \Q^+$.
Then $p$ is a sink $($resp. source$)$ singularity for $\calh$ provided
that $H(0,0)/\dl_1 > 0$ $($resp. $H(0,0)/\dl_1 < 0)$. In both cases,
$p$ yields an endpoint for the trajectories of $\calh$.
\end{lema}

\begin{proof} It suffices to consider the case $\dl_1 /\overline{H} (0,0) > 0$.
Clearly the structure of $\calh$ over the two (smooth) separatrices
of $\tilf_{\infty}$ at $p$ corresponds to sinks. As to the remaining
leaves, recall that they all accumulate on the origin. Furthermore the
structure of $\calh$ over regular points of these leaves has to be of
the same nature as the corresponding structure over the smooth
separatrices. That is to say that all these trajectories point inward
the singularity $p \simeq (0,0)$. The lemma is proved.
\end{proof}

The next step is to consider the case in which $\dl_1/\dl_2 \in \Q^-$. The restriction of
$\tilf_{\infty}$ to a neighborhood of $p$ admits exactly $2$~separatrices. These
separatrices are the unique leaves (of the restriction of $\tilf_{\infty}$
to a neighborhood of $p$) ``radially'' accumulating on the singular point
$p$. In vague terms, the separatrices are the only leaves of $\tilf_{\infty}$
accumulating on $p$ if we ignore the effect of the local holonomy of this
foliation. Denote by $Sep$ one of the separatrices. The restriction of $\mathcal{H}$ to
$Sep$ may have a singular point at $p \in Sep$. The nature of this singular
point depends also on the sign of the quotient $\dl_1 / \overline{H} (0,0)$. If $\dl_1 / \overline{H} (0,0) > 0$
then $p$ corresponds to a sink of $\mathcal{H}$ (or of $\omega_1$ by a small
abuse of notation) over $Sep$. Conversely, in the case where $\dl_1 /\overline{H} (0,0) < 0$,
the singular point corresponds to a source. We note, however, that $\dl_1 /\overline{H} (0,0)$
and $\dl_2 / \overline{H} (0,0)$ have opposite signs. This implies that if $p$ is a sink of
$\omega_1$ for one of the separatrices then $p$ is a source for the other.

The above indicated issue about source and sinks singularities appearing on the two separatrices
of a singularity $p$ as before deserves further comments. First, if we consider real trajectories of $\mathcal{H}$ in the
separatrix admitting $p$ as a sink, then these trajectories will reach a future endpoint at $p$. Somehow compensating
the existence of this future endpoint, in the other separatrix new $\calh$-trajectories are issued. These phenomena can
however occur for only finitely many leaves of our foliation since each separatrix of a singularity as above can give
rise to only one global leaf of $\tilf, \, \tilf_{\infty}$. In particular it will play no significant role in the proof
of any of the theorems stated in the Introduction. In this concern, a far more important observation concern those
$\calh$-trajectories whose projection on $\Delta_{\infty}$ enters a small neighborhood of $p$ but are not contained
in the corresponding separatrix of $p$. In fact, these trajectories can naturally be continued through the ``saddle''
associated to the singularity of $\tilf_{\infty}$ so as to eventually leave a fixed neighborhood of $p$. Indeed, the
foliation $\mathcal{H}$ is regular over all leaves of $\tilf_{\infty}$ different from the two separatrices
on a neighborhood of $p$. Besides, as we are going to see
next, the ``continued'' trajectory keeps the contractive character of its holonomy.

Suppose then that the eigenvalues $\dl_1, \dl_2$ at $p_0$ satisfy $\dl_1/\dl_2 \in \Q^-$.
Let us still assume that $\overline{H}(0,0) \ne 0$ so that it can be normalized to
be~$1$. Let $U_{\de} = \{(x,y,z): |x|, |y| \leq \de\}$ be a small neighborhood
of the origin $\simeq p_0$, not containing other singular points of $\tilf_{\infty}$.
Fix a regular leaf $L_{\infty} \subseteq \Delta_{\infty}$ (distinct from the separatrices)
intersecting $U_{\de}$ and consider a real trajectory $l \subseteq L_{\infty}$ for $\calh$.
For these singularities we have:

\begin{prop}
\label{siegeldomain}
Let $X$ be as in the statement of Lemma~\ref{hreal} and assume that $\dl_1/\dl_2 \in \R_-$. Let $U_{\de}$
be a small neighborhood of the (simple) singular point $p \simeq 0$ as above. Then the integral
$\int_{l_q \cap U_{\de}} dT$ is uniformly bounded for every $\textsc{P} \in l$ and $q \in \calp^{-1}_{\infty}(\textsc{P})$.
\end{prop}

\begin{obs}
\label{explainingstatement}
{\rm It should be emphasized that the trajectory $l_q$ in the statement is viewed as a
global trajectory of $\mathcal{H}$. In other words, the intersection $l_q \cap U$ possesses, in general, infinitely
many connected components. The proposition, indeed, claims that the sum of the integrals of $dT$
over all these connected components is uniformly bounded.}
\end{obs}

\noindent {\it Proof of Proposition~\ref{siegeldomain}}.
Let $X$ be as in the statement of Lemma~\ref{hreal}. Since $\overline{H} (0,0) \neq 0$, we can actually assume that
$\overline{H} (0,0) =1$. Also, it follows that the degree~$d$ of $X$ is exactly~$2$ (Lemma~\ref{hreal}). Nonetheless
to help the reader with the discussion conducted immediately after the end of the proof of Proposition~\ref{siegeldomain}
(cf. Appendix to Section~5), we shall denote this degree by~$d$ and only make the substitution $d=2$ at the very end
of the proof.

According to Lemma~\ref{hreal}, both eigenvalues $\dl_1$ and $\dl_2$ must be integers so that $\dl_1/\dl_2$ belongs
to $\Q$. By assumption, this quotient must, in fact, belong to $\Q_-$ i.e. $\dl_1$ and $\dl_2$ have opposite signs.
Next consider the foliation $\tilf_{\infty}$ induced on $\Delta_{\infty}$. Since $\dl_1/\dl_2 \in \Q_-$, the
corresponding singular point $p$ admits two (smooth) separatrices. In local coordinates
$(u,v,w)$ centered about $p$ as before, these separatrices can be identified with the axes
$\{ u=0\}$ and $\{ v=0\}$. Thus, if $(u,v,w)$ are suitably chosen, the vector field $\tXX$ takes the (local) form
\[
\tXX = w^{1-d} \left[ \overline{F} (u,v) \frac{\partial}{\partial u} + \overline{G} (u,v)
\frac{\partial}{\partial v} + w \overline{H} (u,v) \frac{\partial}{\partial w} \right]
\]
where $\overline{F} (u,v) = u(\dl_1 + {\rm h.o.t.})$, $\overline{G} (u,v) = v(\dl_2 + {\rm h.o.t.})$, $\dl_1/\dl_2
\in \Q_-$ and $\overline{H}(0,0) = 1$.

Assume, without loss of generality that $\dl_1 \in \R_+$ (resp. $\dl_2 \in \R_-$) and consider the restriction of
$\omega_1$ to the $u$-axis (resp. $v$-axis). The residue of $\omega_1$ at $0 \simeq p$ with
respect to the mentioned axis is equal to $-\overline{H}(0,0)/\dl_1$ (resp. $-\overline{H}(0,0)/\dl_2$). Therefore
the restriction of $\calh$ to the $u$-axis (resp. $v$-axis) possesses a sink
singularity (resp. source singularity) at $p \simeq 0$. Hence, the real trajectories
contained in the $u$-axis approaches $p$. Similarly, those trajectories contained in
the $v$-axis move away from $p$. It is easy to describe the behavior of $\calh$ on
the regular leaves of $U$ not accumulating at $p$: over a real trajectory of $\calh|_U$
the absolute value of $u$ decreases while the absolute value of $v$ increases. In other
words, a real trajectory moves away from the invariant plane $\{v=0\}$ while approaches
the plane $\{u=0\}$. In particular, whenever a (global) real trajectory $l$ enters the open
set $U_{\de}$ it necessarily leaves $U_{\de}$ as well.

The preceding discussion shows that the only possibility for a $\calh$-trajectory (not contained in the global leaves
arising from the axes $\{ v=w=0\}$ and $\{u=w=0\}$) to accumulate on the singular point $p$
happens when this trajectory enters infinitely many times the open set $U_{\de}$. The sequence
of points defined by the moment in which the mentioned trajectory enters $U_{\de}$ must also
contain a subsequence that converges for a point in the $u$-axis. Also, in this case, it is
immediate to check that the {\it length of each connected component of $l \cap U_{\de}$}\,
is bounded above by some uniform constant.

For each leaf of $\tilf \cap U_{\de}$ not contained in the invariant plane $\{u=0\}$, the time-form
is given by
\begin{equation}\label{timeform}
dT = \frac{w^{d-1}}{\overline{F}(u,v)}du \, .
\end{equation}
The leaf can locally be parameterized by the $u$-variable under the form $(u, v(u), w(u))$
where $w$ is given by Equation (\ref{omega1}). The expressions of $\overline{F}$ and $\overline{G}$ allow us to see
that $v(u) = v_0 (u/u_0)^{\dl_2/\dl_1}g(u)$ for some bounded holomorphic function $g$ on
$\C \setminus \R_-^{\ast}$ verifying $\lim_{u \rightarrow u_0}g(u) = 1$. In turn, the coordinate
$w$ is given by $w = w_0 e^{ -\int_{u_0}^u  \omega_1 }$, where $\omega_1$ coincides with
$-\overline{H}(u,v(u))/\overline{F}(u,v(u))du$. Therefore, substituting $v$ and $w$ in Equation~(\ref{timeform}),
we obtain
\begin{equation}
dT  = w_0^{d-1} \frac{1}{\overline{F}(u,v(u))} e^{ -(d-1) \int_{u_0}^u \omega_1
} \, du \, . \label{introducelabel1}
\end{equation}

Since we need to estimate the integral of the time-form over oriented real trajectories of
$\calh$, let us start by controlling the exponential term. Since $\overline{H}(0,0) = 1$, it follows that
$$
-\frac{\overline{H}(u,v)}{\overline{F}(u,v)} = -\frac{1}{\dl_1 u} \left( 1 + R(u,v)
\right)
$$
for some holomorphic function $R(u,V)$ on a neighborhood of the origin verifying $R(0,0) = 0$.
In particular, if $\varepsilon$ is sufficiently small, the absolute value of $R(u,v)$ is bounded
above by a small constant $0 < \delta << 1$ on $U_{\de}$. If $l$ is a trajectory of $\calh$ then
$\int_l \omega_1$ is a positive real number. Therefore
\[
\left| e^{-(d-1)\int_l \omega_1} \right| = e^{-(d-1) {\rm Re}\int_l
\omega_1} = e^{-(d-1)\int_l \omega_1}.
\]

Consider a (connected) segment of the real trajectory $l$ joining $u_0$ to $u$ where both points
are contained in the neighborhood in question. Denote by $\phi : [0, 1] \rightarrow L$ a
parametrization of this segment satisfying  $\phi(0) = u_0$ and $\phi(1) = u$. Up to a
change of coordinates, ``close" to a rotation, we can assume that the (connected) segment
$\phi ([0,1])$ is totally contained in the real axis. In fact, we can assume that it is
contained on its positive component. In particular, we can take $\phi(t) = u_0 + t(u - u_0)$.
It then follows
\begin{align*}
\int_l \omega_1 &= \int_0^1 \omega_1. \phi = \int_0^1 -\frac{\phi^{\prime}(t)}{\dl_1 \phi(t)} (1 + R(\phi(t),
v(\phi(t))) dt\\
&= \frac{1}{\dl_1} \int_0^1 -\frac{u - u_0}{u_0 + t(u - u_0)} (1 +
R(\phi(t), v(\phi(t))) dt\\
&\geq \frac{1 - \delta}{\dl_1} \int_0^1 -\frac{u - u_0}{u_0 + t(u -
u_0)}dt = \frac{1 - \delta}{\dl_1} \ln \left( \frac{u_0}{u} \right) \, .
\end{align*}
Therefore we obtain
\begin{equation}
\left| e^{ -(d-1) \int_l \omega_1 } \right| \leq C u^{\frac{(d-1)(1
- \delta)}{\dl_1}} \label{introducelabel2}
\end{equation}
where $C$ is a constant depending on $d$, $\dl_1$, $\delta$ and on $u_0 (= \de)$. In more
accurate terms, $C = \de^{\frac{(1-d)(1-\delta)}{\dl_1}}$. In fact,
this estimation should be multiplied by a constant
representing the supremum of the absolute value of the determinant
of the change of coordinates. However we can, basically, include
this quantity in $C$ since the absolute value of the determinant is
bounded above on $U_{\de}$. Actually, the value of the determinant in question is very close to~$1$
since the change of coordinates is ``close" to a rotation. In this
sense, the constant $C$ does not depend on the  segment of the real trajectory.

Now recall that $F(u,v) = \dl_1 u(1 + r(u,v))$, for some holomorphic
function $r$ on $U_{\de}$ satisfying $r(0,0)=0$. Modulo reducing
$\de$, we can assume that $\vert r(u,v) \vert$ is bounded above by a small
constant $0 < \tau << 1$. Therefore, the coefficient of the
time-form satisfies
\[
\vert dT \vert \leq \vert w_0 \vert^{d-1} \frac{C}{\dl_1 (1 - \tau)} u^{\frac{(d-1)(1 -
\delta)}{\dl_1} - 1} \, .
\]
Since the exponent of $u$ is greater than $-1$, the primitive of the
coefficient of the time-form, up to the term $u_0^{d-1}$, is a
bounded holomorphic function. It follows that the integral of the
time-form, up the same mentioned term $u_0^{d-1}$, over each connected component $l_i$ of $l
\cap U_{\de}$ is bounded above. In fact, there is a positive constant
$K$, not depending on the trajectory of $\calh$, such that
\[
\left| \int_{l_i} w_0^{1-d} dT \right|< K.
\]

Finally the integral of the time form along $l_i$ is now bounded by $K$
times the absolute value of a positive power of the variable $w$ in the moment that the
trajectory $l$ enters the open set $U_{\de}$ or, equivalently, on the
starting point of $l_i$. We denote by $w_i$ the value $w$ at the
starting point of $l_i$. As already mentioned, the holonomy of
$\tilf$ is contractive. Therefore, since the length of the real
trajectory between two consecutive starting points of $l \cap U_{\de}$ is
bounded from below, the sequence $w_i$ is such that $\vert w_{i+1} \vert /\vert w_i \vert
\leq k$, for some constant $0 < k < 1$, since the trajectories of $\mathcal{H}$ have contractive
holonomy. Thus
\[
\left| \int_{l \cap U_{\de}} dT \right| \leq \sum \left| \int_{l_i} dT
\right| \leq \sum K \vert w_i \vert^{d-1} \leq \sum K \vert w_0 \vert^{d-1}k^{i(d-1)} =
\frac{K \vert w_0 \vert^{d-1}}{1-k^{d-1}}
\]
Since $d =2$, the last estimate ensures that the corresponding integral is uniformly bounded as desired.\qed

Let us now provide the proof of Theorem~\ref{maintheo}.

\begin{proof}[Proof of Theorem~\ref{maintheo}]
The proof follows immediately from the combination of Theorem~\ref{introducelabel4} with
Proposition~\ref{siegeldomain}.
\end{proof}

\noindent {\bf $\bullet$ Appendix to Section~5: a natural relaxation of the condition imposed on the singularities
of $\tilf$}.

To close this section, we would like to indicate a much weaker assumption on the singularities of $\tilf$ that would
still be enough to yield Proposition~\ref{siegeldomain}, and hence Theorem~\ref{maintheo}. In fact, Proposition~\ref{siegeldomain}
can be extended to ``almost all'' of the class of singularities named ``absolutely isolated'', cf. \cite{canoetc}.

To explain how this generalization can be worked out, let us consider the vector field $\tXX$ given in local coordinates
$(u,v,w)$ about a singular point of $\tilf$ in $\Delta_{\infty}$ by
\begin{equation}\label{tXXxy}
\tXX u^n v^m w^{1-d} \left[ \overline{F}(u,v) \frac{\partial}{\partial u} + \overline{G}(u,v)
\frac{\partial }{\partial v} + w \overline{H}(u,v) \frac{\partial }{\partial w}
\right]
\end{equation}
where $d \geq 2$ and $n,\, m \in \Z$. Assume also that the singularity of the vector field
$\overline{F} (u,v) \partial /\partial u + \overline{G} (u,v) \partial /\partial v$ at $(0,0) \in \C^2$
is simple and that $\overline{H} (0,0) \neq 0$. As stated, Proposition~\ref{siegeldomain} no longer holds
for $\tXX$ as above. However, this proposition still holds for $\tXX$ as above under the additional
assumption that $\max \{ m,n\} \leq 0$, as it can straightforwardly be checked from the above given proof of the
proposition in question (this explains why we decided to make the substitution $d=2$ only at the end of
the mentioned proof).

The preceding observation is the key to adapt Theorem~\ref{maintheo} (and, by means of it, also Theorem~B) to a much
larger class of vector fields possessing ``absolutely isolated singularities'', as opposed to simple singularities,
in the hyperplane at infinity $\Delta_{\infty}$. This goes as follows. Denote again by $X$ a polynomial vector field
whose associated foliation $\tilf$ has only ``absolutely isolated singularities'' in $\Delta_{\infty}$. Suppose also
that the singular set of $X$ has codimension at least~$2$. According to the main result of \cite{canoetc}, these
singularities can be reduced by applying successive punctual blow-up maps to them. Furthermore, if the very generic
assumption that the absence of singularities of type ``saddle-node'' in the reduction procedure is added, then the final
(reduced) singularities will be of the type appearing in Formula~(\ref{tXXxy}). Besides, in the vast majority
of cases, the corresponding integers $m,n$ are non-positive. Thus, for the corresponding vector fields, Theorem~\ref{maintheo}
will still hold.

Summarizing what precedes, it should be said:
\begin{enumerate}
  \item Proposition~\ref{siegeldomain}, and hence Theorem~\ref{maintheo}, cannot be extended to {\it arbitrary singular
  points}\, without additional information on the global dynamics of the foliation $\tilf_{\infty}$ on $\Delta_{\infty}$.

  \item This proposition, however, can be extended to a vast class of singular points that is ``generic'' among singular points
  for any {\it a priori}\, fixed order. In particular, in the class of singular points for which
  Proposition~\ref{siegeldomain} (and Theorem~\ref{maintheo}) still holds, it can be found singularities ``concealing'' some
  very complicated dynamical behavior.
\end{enumerate}

\section{Applications to complete vector fields}

\subsection{Ends of solutions of complete polynomial vector fields on $\C^n$}

This first application concerns Theorem~\ref{introducelabel4}. Consider a complete polynomial vector field $X$ defined on
$\C^n$. Set $X=\sum_{i=0}^d X_i$ where $X_i$ stands for the homogeneous component
of degree~$i$ of $X$. To keep as much as possible the notations used in the
previous sections, the foliation associated to $X$ will be denoted by $\cald$
whereas $\fol$ will stand for the foliation associated to the top-degree homogeneous
component $X_d$. Throughout what follows, the degree~$d$ is supposed to be at least~$2$.

Recall that both foliations $\cald$ and $\fol$ admit holomorphic extensions to
$\C P(n)$ and these extensions are also denoted by $\cald$ and $\fol$.

\begin{lema}
\label{suzuki1.1}
The homogeneous vector field $X_d$ is not a multiple of the radial vector
field
$$
R = x_1 \partial /\partial x_1 + \cdots + x_n \partial /\partial x_n \; .
$$
\end{lema}

\begin{proof}
First note that the vector field $X_d$ is semi-complete on all of $\C^n$ since
it is the top-degree homogeneous component of a complete vector field. More precisely, let
$\Lambda_n$ denote the map $(x_1, \ldots ,x_n ) \mapsto (2^{n} x_1 , \ldots , 2^{n} x_n)$.
Next let $Y_n$ be the vector field defined by $2^{(1-d)n} . \Lambda_n^{\ast} X$. For every~$n$ fixed, $Y_n$
is clearly a complete vector field on $\C^n$ so that its restriction to every compact open set of
$\C^n$ is semi-complete. Besides the sequence $\{ Y_n \}$ converge uniformly on compact sets
to the vector field $X_d$. Since the set of semi-complete vector fields
is closed for the topology of uniform convergence on compact sets, it follows that $X_d$ is semi-complete
on all of $\C^n$.

Suppose now that $X_d$ is a multiple $fR$ of $R$. In view of Lemma~\ref{Adim3}, it follows that $f$
is a linear form, i.e. a homogeneous polynomial with degree~$1$. To obtain a contradiction with this last
possibility, we proceed as follows.
First, note that the generic leaf $L$ of $\cald$ intersects the hyperplane at infinity
of $\C P(n)$ transversely at a regular point $p$ for $\cald$. Besides the point $p$ is regular
for the restriction of $X$ to $L$. In other words, the flow of $X$ reaches the hyperplane
at infinity in finite time. This is impossible since $X$ is complete on $\C^n$. The proof of
the lemma is over.
\end{proof}

Again let $\Delta_{\infty}$ denote the hyperplane at infinity in $\C P(n)$. It follows from the preceding
that $\Delta_{\infty}$ is invariant by both $\cald$ and $\fol$. Besides, the foliations induced
on $\Delta_{\infty}$ by $\cald$ and $\fol$ turn out to coincide. The foliation induced by $\fol$ on $\Delta_{\infty}$
will be denoted by $\fol_{\infty}$. Also $\Delta_{\infty}$ corresponds to the divisor
of poles for both $X, \, X_d$. Since the methods developed in the previous sections apply to foliations associated
to homogeneous vector fields, they are in principle not applicable to $\cald$ but only to $\fol$.
However, near $\Delta_{\infty}$, the foliation $\cald$ becomes very close to $\fol$. In the sequel we are going
to combine these two issues in order to establish Theorem~A.

Let us begin by choosing affine coordinates $(x_1, \ldots, x_{n-1}, z)$ analogous to those
used in Sections~3, 4. Namely the hyperplane $\{ z=0\}$ is contained in $\Delta_{\infty}$ and the
plane at infinity $\Delta_{\infty}^{1,\ldots ,n-1}$ defined by the affine coordinates $x_1, \ldots, x_{n-1}$,
where $z=0$ is fixed, is not invariant by the restrictions of either $\cald$ or $\fol$ to $\Delta_{\infty}$.
We are then able to apply the results of Section~4 to the foliation $\fol$. In particular,
the leaves of $\fol$ are equipped with the (singular) real foliations $\calh^{\theta}$ where $\theta$
is chosen in the interval $(-\pi/2 , \pi/2)$. For the rest of this section, these foliations
will be denoted by $\calh_{\fol}$ (resp. $\calh^{\theta}_{\fol}$).
To define a suitable version of these
real trajectories in the leaves of $\cald$ we proceed as follows. Given a point $p=
(x_1^0, \ldots, x_{n-1}^0, z^0)$ with $z^0 \neq 0$, let $L_p$ denote the leaf of $\cald$ through
$p$. To define the foliation $\calh_{\cald}$ at $p$, we consider the function
$(x_1, \ldots, x_{n-1}, z) \mapsto \vert z \vert$ restricted to $L_p$. The tangent vector to
$\calh_{\cald}$ at $p$ is simply the negative of the gradient of the function in question. Once
$\calh_{\cald}$ is defined the foliations $\calh^{\theta}_{\cald}$ have an obvious definition since the leaves of $\cald$ have
natural conformal structures.

The next step in our construction consists of investigating the basic properties of
$\calh_{\cald}$ in analogy with the properties of $\calh_{\fol}$ considered in
Sections~3 and~4. Recalling that $\cald, \, \fol$ induce the same foliation $\fol_{\infty}$
on $\Delta_{\infty}$, consider a point $(x_1^0, \ldots ,x_{n-1}^0, 0) \in \Delta_{\infty}$ that is regular for
the restrictions of both $\cald, \, \fol$ to $\Delta_{\infty}$. Then we have:

\begin{lema}
\label{suzuki2.2}
The direction of $\calh_{\cald}$ at the point $(x_1^0, \ldots ,x_{n-1}^0, z)$ converges
uniformly to the direction of $\calh_{\fol}$ at $(x_1^0, \ldots ,x_{n-1}^0, 0)$. In particular
the foliation $\calh_{\cald}$ can be extended to the regular part of $\cald$ in $\Delta_{\infty}$ and
this extended foliation coincides with $\calh_{\fol}$ on $\Delta_{\infty}$.
\end{lema}

\begin{proof}
Since the behavior of $\cald$ near $(x_1^0, \ldots ,x_{n-1}^0, 0)$ is dominated by the
component $X_d$ of $X$, it suffices to check that the trajectories of $\calh_{\fol}$
admit a definition analogous to the one given above for the trajectories of $\calh_{\cald}$.
In other words, it suffices to prove that the direction of $\calh_{\fol}$ at
$(x_1^0, \ldots ,x_{n-1}^0, z)$ coincides with the gradient of the function
$(x_1, \ldots, x_{n-1}, z) \mapsto \vert z \vert$ restricted to the leaf of $\fol$ through
$(x_1^0, \ldots ,x_{n-1}^0, z)$. This is, however, an immediate consequence of
Formula~(\ref{omega1}). The lemma is proved.
\end{proof}

\begin{obs}
\label{remarkboundedgeometry}
{\rm {\bf The trajectories $\calh_{\fol}$ and $\calh_{\cald}$ as geodesics for a foliated flat structure}. By building
on the proof of Lemma~\ref{suzuki2.2}, we can provide further and more accurate information on the discussion
about ``separating curves'' and ``flat structure with bounded geometry'' conducted in the Introduction
(after the statement of Theorem~A'). In the case of homogeneous foliations, such as $\fol$, every leaf $L$ of
$\fol$ is equipped with the $1$-form $\omega_1$ defined in Section~3. These include the leaves of $\fol$
contained in $\Delta_{\infty}$ (or $\Delta_0$) so that $\omega_1$ is a foliated $1$-form on the compact manifold $M$
which, in turn, gives $\omega_1$ its ``bounded geometry'' nature. Since the leaves of $\fol$ are Riemann surfaces,
the restriction of $\omega_1$ to one of these leaves $L$ can equally well be seen as singular flat structure
on $L$. It is an elementary fact that, in a local coordinate for the flat structure in question, the trajectories of
$\calh_{\fol}^{\theta}$ (including $\calh_{\fol}, \, \calh_{\fol}^{\perp}$) are straight lines and hence geodesics for the flat structure itself.
In particular, the trajectories of, say, $\calh_{\fol}^{\perp}$ satisfies all conditions
in the statement of Theorem~A' to be a {\it separating curve}.
In view of the preceding, these curves are also geodesics for a suitable flat
structure with bounded geometry on the corresponding leaf of $\fol$.

To extend the construction of the $1$-form $\omega_1$ to non-homogeneous foliations such as $\cald$, we proceed as
follows. Consider a leaf $L$ of $\cald$. We want to construct a $1$-form
$\omega_{1,\cald}$ on $L$ for which $\calh_{\cald}$ is the real foliation. For this, suppose
that $L$ is not contained in $\Delta_{\infty}$ and consider a regular point
$p \in L$ along with a vector $v \in T_p L$ pointed in the oriented direction of $\calh_{\cald}$. In particular,
all directions for $\calh_{\cald}^{\theta}$ and $\calh_{\cald}^{\perp}$ are immediately defined through the conformal
structure of $L$. To define $\omega_{1,\cald}$ at $p$, we just need to ``associate'' a complex number to these
real foliations. The complex number in question is simply the derivative of the holomorphic function $(x_1, \ldots ,
x_{n-1} ,z) \mapsto z$ restricted to $L$. This construction equips every leaf $L$ of $\cald$ not contained in $\Delta_{\infty}$
with an abelian form $\omega_{1,\cald}$ and hence with a singular flat structure. Finally, if $L$ is contained in
$\Delta_{\infty}$, then we pose $\omega_{1,\cald} = \omega_1$ since the foliation $\cald$ and $\fol$ coincide on
$\Delta_{\infty}$. The previous discussion show that $\omega_{1,\cald}$ is a foliated $1$-form for $\cald$ defined on
all of the compact manifold $M$. In this sense, the argument used in the case of $\omega_1$ can be repeated here to show
that the foliated flat structure arising from $\omega_{1,\cald}$ has bounded geometry. Moreover, the trajectories of
$\calh_{\cald}^{\perp}$ define separating curves in the sense of Theorem~A'.}
\end{obs}

To help us to explain how to derive properties of $X, \, \cald$ from properties of
$X_d, \, \fol$, it is convenient to consider a small neighborhood $V$ (in the $n$-dimensional ambient space)
of $({\rm Sing}\, (\cald) \cap \Delta_{\infty}) \cup {\rm Sing}\, (X)$, where ${\rm Sing}\, (\cald)$ (resp. ${\rm Sing}\, (X)$)
stands for the singular set of $\cald$ (resp. $X$). Next, denote by $U$ a neighborhood of $\Delta_{\infty} \setminus V$.
Also, in order to keep a ``uniform contractive character'' over trajectories of
$\calh_{\fol}^{\theta}$, we fix some (small) $\epsilon >0$ and consider only those
values of $\theta$ belonging to the interval $(-\pi/2 + \epsilon, \pi/2 - \epsilon)$.

It follows from our general construction that the endpoints belonging to $U$ for trajectories of $\calh_{\fol}^{\theta}$
are situated over $\Delta_{\infty}^{1,\ldots ,n-1}$. In particular a trajectory of $\calh_{\fol}^{\theta}$
through an affine point $(x_1^0, \ldots ,x_{n-1}^0, 0) \in U \cap \Delta_{\infty}$ will never approach
$\Delta_{\infty}^{1,\ldots ,n-1}$ unless it first enters $V$. Modulo choosing the neighborhood
$U$ sufficiently narrow, the restriction to $U$ of the foliation $\cald$ is very close to the (restriction to $U$ of the)
foliation $\fol$. A similar statement holds for the foliations $\calh_{\cald}$ and $\calh_{\fol}$ thanks to
Lemma~\ref{suzuki2.2}. In particular, we obtain the following:

\begin{lema}
\label{suzuki3.3}
Let $\epsilon >0$ be fixed.
Consider a point $p=(x_1^0, \ldots ,x_{n-1}^0, z^0) \in U$ and denote by $l_p^{+,\theta}$ the semi-trajectory of
$\calh_{\cald}^{\theta}$ initiated at $p$ for $\theta \in [-\frac{\pi}{2} + \epsilon, \frac{\pi}{2} - \epsilon]$.
Consider a path $c :[0,1] \rightarrow  l_p^{+,\theta} \cap U$, with $c(0) =p$, and set
$c(1) = (x_1^1, \ldots ,x_{n-1}^1, z^1)$. Then there is a constant ${\rm Cte}$ depending solely on
$\epsilon$ such that the following condition is always verified: whenever the length of $c$ exceeds
${\rm Cte}$, we have the estimate $\vert z^1 \vert < \vert z^0 \vert /2$.
\end{lema}

\begin{proof}
It follows immediately from the proof of Theorem~\ref{introducelabel4} concerning the foliation $\fol$. More precisely,
the statement was shown for $\calh_{\fol}$ but the adaptations needed for the foliations $\calh_{\fol}^{\theta}$ are clear.
The present statement follows from the fact that inside $U$ the foliation $\calh_{\fol}^{\theta}$ is ``very close''
to $\calh_{\cald}^{\theta}$.
\end{proof}

One last ingredient is still needed for the proof of Theorem~A. Let $l_p^{+,\theta}$ be a trajectory as in Lemma~\ref{suzuki3.3} and denote by
$L_p$ the leaf of $\cald$ containing $l_p^{+,\theta}$. The idea behind the statement of Theorem~A consists of estimating the integral
of $dT_L$ over $l_p^{+,\theta}$ where $dT_L$ stands for the time-form induced by $X$ on $L_p$. In Section~4 suitable estimates
for this type of integral were obtained in the case of homogeneous vector fields. The estimate is based on the ``renormalized time-form'' induced on
$\Delta_{\infty}$ by the vector field and on the evolution of the distance of the points to $\Delta_{\infty}$ (the ``height'' of the points).
As to the height of points, the preceding lemma provides a suitable control of their evolution over trajectories of $\calh_{\cald}^{\theta}$
in the case of non-homogeneous polynomial vector fields. Finally we recall that the foliations induced by $X$ and by $X_d$ on $\Delta_{\infty}$
turn out to coincide and the same holds for the ``renormalized time-forms'' induced on $\Delta_{\infty}$ by $X$ and by $X_d$.

We are now ready to prove Theorem~A.

\begin{proof}[Proof of Theorem~A]
Consider the foliation $\cald$ induced by $X$ on $\C P(n)$ and let $\Delta_{\infty}$ be as above. Let $V$ denote the
given neighborhood of $({\rm Sing}\, (\cald) \cap \Delta_{\infty}) \cup {\rm Sing}\, (X)$ and fix
$\epsilon > 0$. Next, choose a neighborhood $U$ of $\Delta_{\infty} \setminus V$ so that the statement of Lemma~\ref{suzuki3.3} holds.
It is sufficient to prove the theorem for the foliation $\calh_{\cald}$ since the adaptations needed to the general case of the foliations
$\calh_{\cald}^{\theta}$, $\theta \in (-\pi/2 + \epsilon , \pi/2 -\epsilon)$, are clear.

Consider a point $p=(x_1^0, \ldots ,x_{n-1}^0, z^0) \in U \setminus V$. Denote by
$l_p^+$ (resp. $L_p$) the semi-trajectory of $\calh_{\cald}$ initiated at $p$ (resp. the leaf of $\cald$ through $p$).
Suppose first that $l_p^+$ is entirely contained in $U$. To explain the
structure of the proof of our theorem, we shall first prove that the preceding assumption contradicts the fact that the vector
field $X$ is complete. For this, we are going to show that the integral of the time-form $dT_L$ induced by $X$ on $L_p$ over
$l_p^+$ converges. Since it clearly accumulates on $\Delta_{\infty}$ (in particular $l_p^+$ leaves every compact set contained in $L_p$)
the convergence of the mentioned integral contradicts the completeness of $X$. Let us also point out that our claim reduces to
Theorem~\ref{introducelabel4} in the case of homogeneous vector fields.

To adapt the proof of Theorem~\ref{introducelabel4} to the present case where $X$ is not homogeneous we proceed as follows.
The choice of the coordinates $(x_1, \ldots, x_{n-1}, z) = (\underline{x}, z)$ allows us to write $X_d$ in the form
$$
X_d = z^{1-d} [F_1 (\underline{x}) \partial /\partial x_1 + \cdots + F_{n-1} (\underline{x}) \partial /\partial x_{n-1} + z H (\underline{x}) \partial /\partial z]
$$
whereas the vector field $X$ becomes
$$
X =  z^{1-d}  [ F_1^{\ast} (\underline{x}, z) \partial /\partial x_1 + \cdots + F_{n-1}^{\ast} (\underline{x}, z) \partial /\partial x_{n-1} +
z H^{\ast} (\underline{x}, z) \partial /\partial z \!]  .
$$
Besides the coefficients $F_i, \, F_i^{\ast}$, $i=1, \cdots , n-1$, (resp. $H_i, \, H_i^{\ast}$) are related by the formulas
$$
F_i^{\ast} (x_1, \ldots , x_{n-1}, z) - F_i (x_1, \ldots , x_{n-1}) = z P_i (x_1, \ldots , x_{n-1}, z)
$$
(resp. $H^{\ast} (x_1, \ldots , x_{n-1}, z) - H (x_1, \ldots , x_{n-1}) = z Q (x_1, \ldots , x_{n-1}, z)$)
where $Q, \, P_i$ are polynomials in the variables in question. Next note that the time-form $dT_L$ is given by
$$
dT_L = \frac{z^{d-1} }{F_1^{\ast} (x_1, \ldots , x_{n-1}, z)} \, dx_1 = \cdots = \frac{z^{d-1} }{F_{n-1}^{\ast} (x_1, \ldots , x_{n-1}, z)} \, dx_{n-1} \, .
$$
Now since $U$ does not intersect the singular set of $X$ (or $\cald$), we can suppose without loss of generality that
$F_1^{\ast} (x_1, \ldots , x_{n-1}, z)$ is bounded from below by a positive constant $\beta$ in $U$, otherwise we replace $F_1^{\ast}$ by a suitable
$F_i^{\ast}$. This last estimate combined to Lemma~\ref{suzuki3.3} then shows that the integral of $dT_L$ over $l_p^+$ is bounded by simply repeating
the calculations performed in the proof of Theorem~\ref{introducelabel4}.

We are then led to conclude that the semi-trajectory $l_p^+$ must intersect the neighborhood $V$ of
$({\rm Sing}\, (\cald) \cap \Delta_{\infty}) \cup {\rm Sing}\, (X)$ regardless of how small is $V$. In particular,
it may happen that $l_p^+$ accumulates on singular points of $\cald$ lying in $\Delta_{\infty}$.
In this case the integral of $dT_L$ over $l_p^+$ cannot be bounded without additional conditions. Fortunately, in order
to establish Theorem~A, we do not need to keep track of the amount of ``time'' that $l_p^+$ spends {\it inside}\, $V$ but rather of the amount
of time that $l_p^+$ spends away from $V$. To be more precise, let us prove the following:

\vspace{0.1cm}

\noindent {\it Claim}. The distance between the trajectory $l_p^+$ and the hyperplane $\Delta_{\infty}$ cannot have a minimum unless this
minimum is {\it zero}. Besides, when the latter case happens,
the intersection point $l_p^+ \cap \Delta_{\infty}$ is never reached by the flow of $X$.

Before starting the proof of the claim, it is convenient to make some general comments regarding the possibility of having a point $q$ in
$l_p^+ \cap \Delta_{\infty}$. A first case where this may happen arises from the definition of ``leaf'' given in Section~2 and borrowed from
\cite{marco2}. According to this definition, the leaf $L_p$ of $\cald$ may contain a singular point of $\cald$ lying in $\Delta_{\infty}$.
In fact, in this case, a local branch of $L_p$ about $q$ defines an irreducible separatrix for $\cald$ at $q$. It is then natural to think
of $q$ as belonging to $l_p^+$. More generally, it may happen that $l_p^+$ converges to a point $q$ lying in $\Delta_{\infty}$ whether or
not $q$ belongs to $L_p$. With a small abuse of notation, the point $q$ may be thought of as belonging to $l_p^+$. In all these cases the
statement of Theorem~A is clear: the completeness of $X$ implies that the integral of $dT_L$ over a local branch of $l_p^+$ converging
to $q$ is infinite. So $l_p^+$ enters every given neighborhood of $q$ and remains inside ``forever''. The statement then follows from
observing that $q$ must be a singular point of $\cald$ since $\Delta_{\infty}$ is invariant by $\cald$.

A further reduction in the proof of Theorem~A is possible even though not strictly needed. Namely,
with the above notations, we can suppose that (a local branch of)
$l_p^+$ never converges to a point $q$ that is singular for $X$ (and in particular for $\cald$). In fact, if this point belongs
$\Delta_{\infty}$ then the theorem results immediately as previously seen. Similarly, if
$q \in {\rm Sing}\, (X) \setminus \Delta_{\infty}$, then the theorem follows from the standard results
on existence and uniqueness of solutions for regular ordinary differential equations.

\noindent {\it Proof of the Claim}.
Given what precedes, let us suppose for a contradiction that
$q$ is a point of minimum for the mentioned distance and that $q$ lies in $\C P(n) \setminus \Delta_{\infty}$.
First, we are going to prove that the point $q$ must belong to the domain of definition of the coordinates
$(x_1, \ldots , x_{n-1}, z)$. Since $X$ is not homogeneous and $q$ is not in $\Delta_{\infty}$, this assertion is not an immediate
consequence of Lemma~\ref{lineatinfinity}. Thus, in order to prove it, suppose that $c: [0,1) \rightarrow l_p^+$ is a local
parametrization of $l_p^+$ with $\lim_{t\rightarrow 1^-} c(t) =q$. Setting $c(t) = (x_1 (t), \ldots ,x_{n-1} (t), z(t))$,
it follows that $z(t)$ is locally bounded at $q$. If, in addition, $(x_1 (t), \ldots ,x_{n-1} (t))$ leaves the domain of
definition of coordinates $(x_1, \ldots , x_{n-1}, z)$, then by using standard coordinates of $\C P(n)$ whose domain
contains $\Delta_{\infty}$, it immediately follows from the bounded character of $z(t)$ that $\lim_{t\rightarrow 1^-} c(t) =q$
actually belongs to $\Delta_{\infty}$. As already shown, the statement of the theorem holds when the situation in question occurs.

Summarizing the above discussion, we can suppose that $q = (q_1, \ldots, q_n)$ is a regular point for $X, \, \cald$ belonging to the
domain of definition of the coordinates $(x_1, \ldots , x_{n-1}, z)$ and verifying $q_n \neq 0$. A final contradiction can now be
obtained as follows. Let $\Phi (T) = (\Phi_1 (T), \ldots , \Phi_n (T))$ be a local parametrization of $L_p$ about $q$ ($\Phi (0) =q$).
Since $q$ is a regular point for $X$, the holomorphic map $T \mapsto \Phi_n (T) \in \C$ is not constant and hence it
must be open what, in turn, contradicts the assumption that $\vert \Phi_n \vert$ has
a (positive) local minimum at $T=0$. The claim is proved.\qed

To finish the proof of Theorem~A consider now the semi-trajectory $l_p^+$. The above discussion shows that $l_p^+$ accumulates on $\Delta_{\infty}$,
in particular $l_p^+$ leaves compact sets of $L_p$. The completeness of $X$ then implies
$$
\int_{l_p^+} dT_L = \infty \; .
$$
Consider a decomposition $l_p^+ =  c_1 \ast c_2 \ast \cdots$ of $l_p^+$ in finitely or infinitely many paths such that $c_k$ is contained in $U$
for $k$ odd and $c_k$ is contained in $V$ for $k$ even. The statement is now reduced to prove that
$$
\sum_{k=0}^{\infty} \left[ \int_{c_{2k+1}} dT_L \right] < \infty \, .
$$
The last estimate however follows from the same argument employed above in the case where $l_p^+$ was entirely contained in $U$.
It suffices to observe that the claim guarantees that $\vert c_{2(k+1) +1} (0) \vert < \vert c_{2k+1} (1) \vert$. The theorem is proved.
\end{proof}

We can now prove Theorem~A'.

\begin{proof}[Proof of Theorem~A']
Consider again a fixed point $p$ and let $\Phi_p : \C \rightarrow L_p$ be given by $\Phi_p (T) = \Phi (T, p)$
where $L_p$ stands for the leaf of $\cald$ through $p$. In the affine coordinates $(x_1, \ldots ,x_{n-1}, z)$,
the map $\Phi_p$ becomes $(\Phi_1 (T), \ldots , \Phi_n (T))$. In particular, this allows us to define the Abelian form
$\eta$ on $\C$ by letting $\eta = - \Phi_n^{'} dT / \Phi_n$. Next, if the oriented
foliation $\calh$ is restricted to $L_p$, then we can consider the corresponding pulled-back oriented foliation
$\Phi_p^{\ast} \calh$ on $\C$.

\vspace{0.1cm}

\noindent {\it Claim 1}: The oriented foliation $\Phi_p^{\ast} \calh$ coincides with the real (positive) foliation induced by $\eta$.

\noindent {\it Proof of Claim 1}. Consider a point $q = \Phi_p (T_0) \in L_p$ that is regular for $\calh$. The direction of $\calh$ at $q$
is determined by the negative of the gradient of the ``height'' function $(x_1, \ldots ,x_{n-1}, z) \mapsto \Vert z \Vert$ restricted to $L_p$.
In the coordinate $T$ this function is simply $T \mapsto \Vert \Phi_n (T) \Vert$. The gradient direction of this latter function is determined
by the condition that $\Phi_n^{'} (T_0) (T-T_0)$ must be aligned with $\Phi_n (T_0)$. This amounts to saying that the direction of
$\Phi_p^{\ast} \calh$ at $T_0$ is nothing but the positive real direction of $\eta$.\qed

\vspace{0.1cm}

To abridge notations the foliation $\Phi_p^{\ast} \calh$ will be denoted by $\{ {\rm Arg}\, \eta =0\}$. More generally, the pull-back by
$\Phi_p$ of the foliations $\calh^{\theta}$ coincide with $\{ {\rm Arg}\, \eta =\theta\}$, in particular $\{ {\rm Arg}\, \eta =\pi/2\}$
is the foliation orthogonal to $\{ {\rm Arg}\, \eta =0\} = \Phi_p^{\ast} \calh$.

The separating curve $c_0$ to be chosen is given by the trajectory of $\{ {\rm Arg}\, \eta =\pi/2\}$ through $T_0$
i.e. a trajectory of $\calh^{\perp}$. Geometrically,
$\Phi_p (c_0)$ is the curve determined in $L_p$ by the intersection of $L_p$ itself with the real hypersurface $\{ \Vert z \Vert =
\vert \Phi_n (T_0) \vert \}$. This curve may be closed. Next, we choose the component $\mathcal{U}^+$ of $\C \setminus c_0$ that
corresponds to the saturated of $T_0$ by the spray of trajectories of $\{ {\rm Arg}\, \eta =\theta\}$ issued from $T_0$ with
$\theta \in (-\pi/2 ,\pi/2)$. To check that $\mathcal{U}^+$ is unbounded just notice that a trajectory $l_p^{+, \theta} \subset L_p$,
$\theta \in (-\pi/2 ,\pi/2)$, issued from $p$ will, by construction, accumulate on $\Delta_{\infty}$
unless it accumulates on a singularity of $X$ lying in $\C^n$.
The statement being clear in the latter case, let us consider that $l_p$ accumulates on $\Delta_{\infty}$ so that
it leaves every compact set contained in $L_p$. Since $X$ is complete, it results that the integral of the
corresponding time-form over $l_p^{+, \theta}$ is unbounded. Next, note that the pre-image of $l_p^{+, \theta}$ by $\Phi_p$
is the trajectory of $\{ {\rm Arg}\, \eta =\theta\}$ issued from $T_0$. Furthermore, the pre-image by $\Phi_p$ of the time-form
induced on $L_p$ is nothing but the canonical form $dT$ on $\C$.
Thus the integral of the time-form over segments of $l_p^{+, \theta}$ is equal to the integral
of the form $dT$ over corresponding segments of the mentioned trajectory. It then follows that the trajectory in question
must leave every compact set contained in $\C$ what shows that $\mathcal{U}^+$ is unbounded.

Summarizing to show that $\mathcal{U}^+$ satisfies all the conditions in the statement it only remains to check that
\begin{equation}
\lim_{r \rightarrow \infty} \frac{ {\rm Meas}\, (\mathcal{T}_V \cap B_r)}
{{\rm Meas}\, (\mathcal{U}^+ \cap B_r)} =1 \, . \label{thelastone}
\end{equation}

To begin with, note that $\eta$ is holomorphic in $\mathcal{U}^+$ since $\Phi_n (T)$ never reaches $0 \in \C$.
Furthermore the trajectories of $\calh^{\theta}$, $\theta \in (-\pi/2 ,\pi/2)$, approach
$\{ z=0\}$. These trajectories, in fact, remain in a compact part of the domain of definition
of the coordinates $(x_1, \ldots ,x_{n-1}, z)$ since the ``infinity'' of
$\{ z=0\}$ consists of poles with residue equal to~$1$ for the abelian form $\omega_1$ in Section~3, cf. Lemma~\ref{lineatinfinity}.
Hence the same thing happens for $\omega_{1,\cald}$ since these forms coincide on $\Delta_{\infty}$.
In other words, on a neighborhood of $\Delta_{\infty}^{1,\ldots ,n-1}$, the trajectories of
$\{ {\rm Arg}\, \eta =\pi/2\}$ are closed curves while the trajectories of
$\calh^{\theta}$ for $\theta \in (-\pi/2 ,\pi/2)$ point outward these closed curves. The preceding claim then becomes clear.
As a consequence of this, we conclude that the absolute value of the coefficient of $\eta$ is uniformly bounded in
$\mathcal{U}^+ \setminus \mathcal{T}_V$ since away from $\mathcal{T}_V$ the form $\omega_{1,\cald}$ is bounded from
below by a positive constant. As already explained in Remark~\ref{remarkboundedgeometry},
$\eta$ defines a singular flat structure on $\mathcal{U}^+$ for which the trajectories of
$\{ {\rm Arg}\, \eta =\theta\}$ are geodesics (``straight lines''). This leads us to

\vspace{0.1cm}

\noindent {\it Claim 2}: Given $\varepsilon >0$, there is $\delta >0$ such that the saturated $\mathcal{U}_{\delta}^+$ of $T_0$ by trajectories of
$\{ {\rm Arg}\, \eta =\theta\}$ for $\theta \in (-\pi/2 +\delta ,\pi/2 - \delta)$ verifies
$$
\liminf_{r \rightarrow \infty}
\frac{ {\rm Meas}\, [(\mathcal{U}_{\delta}^+ \cap B_r) \cup (\mathcal{T}_V \cap (\mathcal{U}^+ \setminus \mathcal{U}^+_{\delta}))]}
{{\rm Meas}\, (\mathcal{U}^+ \cap B_r)} > 1- \varepsilon \, .
$$

\vspace{0.1cm}

\noindent {\it Proof of Claim 2}. The statement would be clear if the flat structure induced by $\eta$ were the standard flat
structure of $\C$. More generally, suppose that $\eta$ has no singular points and consider an arc of circle $S_{r_0}$
(about $T_0$) of radius $r_0$ whose interior
contains no singular points of $\eta$. Consider also two trajectories $l_{\theta_1}, \, l_{\theta_2}$ issued from $T_0$.
Since $X$ is complete these trajectories
intersect $S_{r_0}$ at points $T_1, T_2$.  Because $\eta$ is closed, the integral of $\eta$ over the boundary of the triangle
whose sides are the segments of $l_{\theta_1}, \, l_{\theta_2}$ delimited by $T_0$ and $T_1, T_2$ and the corresponding arc
of $S_{r_0}$ determined by $T_1, T_2$ equals zero. Finally, since the coefficient of $\eta$ is uniformly bounded (and bounded
from below if we stay away from its singular points), we conclude that the length of the arc of $S_{r_0}$ determined by
$T_1, T_2$ is bounded by ${\rm Const} r_0 \vert \theta_1 - \theta_2 \vert$ for every pair $\theta_1, \theta_2$. The
desired estimate follows from this since  $S_{r_0}$ contains no singularities of $\eta$ in its interior.

To finish the proof of the claim, we need to consider the effect of the singularities of $\eta$.
These singularities are of saddle type since $\eta$ is holomorphic on $\mathcal{U}^+$. For every singular point
of $\eta$, we consider a disc of uniform radius about the corresponding point in $L_p$. In the complement of the union
of these discs, the form $\omega_1$, and hence $\eta$ in the coordinate $T$, is bounded from below by a positive constant.
The claim will be proved if the union of these discs in the coordinate $T$ has area less than $\varepsilon r/2$ for $r$ large.
In fact, in the complement of this union $\eta$ is bounded from below by a positive constant and from above
by the previous constant so that the preceding argument can be applied in finitely many regions of a ball $B_r$.
Finally, to justify the previous claim note that, in order to prove the desired estimate, we only need to consider those
discs about points in $L_p$ that lie in the complement of $V$. Therefore the norm
of $X$ is bounded from below by a positive constant in these discs what, in turn,
ensures that their size in the coordinate $T$ is uniformly bounded.
Besides, the distance in the leaf $L_p$ between every two discs as before is bounded from below by a positive constant.
Though this property is not directly reflected in the coordinate $T$ since the norm of $X$ increases
(i.e. $X$ becomes ``faster''), the size of the corresponding neighborhoods reduces in the same proportion as the norm of
$X$ increases. This quickly leads to the desired conclusion and establishes the claim.\qed

In view of Claim~2 to finish the proof of Theorem~A' it suffices to show that
$$
\lim_{r \rightarrow \infty} \frac{ {\rm Meas}\, (\mathcal{T}_V \cap B_r \cap \mathcal{U}_{\delta}^+)}
{{\rm Meas}\, (\mathcal{U}_{\delta}^+ \cap B_r)} =1 \, ,
$$
for fixed positive $\delta$. To do this, consider $r$ given. Next, note that every for $\theta \in (-\pi/2 +\delta ,\pi/2 - \delta)$
the corresponding trajectory $l_{\theta}$ of $\{ {\rm Arg}\, \eta =\theta\}$
issued from $T_0$ intersects the boundary $\partial B_r$ of $B_r$ since $X$ is complete. Let $T_{\theta, r}$ be this intersection
point and denote by $l_{\theta, r}$ the segment of $l_{\theta}$ delimited by $T_0$ and $T_{\theta, r}$. According to Theorem~A,
there is uniform constant ${\rm Cte}$ (depending neither on $\theta$ nor on $r$) such that the length of the segments of
$l_{\theta, r}$ corresponding to those instants where $\Phi (T)$ remains away from $V$ is bounded by ${\rm Cte}$
whereas the length of $l_{\theta, r}$ goes to infinity as $r \rightarrow \infty$. The statement of Theorem~A'
now results from a standard application of Fubini's theorem.
\end{proof}

\subsection{Complete polynomial vector fields on $\C^n$ with simple singularities at infinity}

In this section we shall give an application of Theorem~\ref{maintheo} that cannot be obtained from Theorem~\ref{introducelabel4}
alone. Let $X$ be a complete polynomial vector field on $\C^n$ and denote by $\cald$ its associated foliation. Recall that we make
no distinction between $\cald$ viewed as a foliation on $\C^n$ and $\cald$ viewed as a foliation on $\C P(n)$. Again
$X_d$ denotes the homogeneous component of highest degree ($d$) of $X$ and $d$ is supposed to be
at least~$2$. The foliation associated to $X_d$ is denoted
by $\fol$ and can also be viewed as a foliation on both $\C^n$ or $\C P(n)$. Recall that the singularities of $\cald$ in the
hyperplane at infinity $\Delta_{\infty}$ are ``simple'' in the sense of Conditions~1 and~2
given in the Introduction (just before the statement of Theorem~B). There follows that these singularities are
isolated inside $\Delta_{\infty}$ (but maybe not inside $\C P(n)$). Furthermore, if $X_d$ is divisible by a non-constant
polynomial $P$, then $P$ must also divide $X$. Otherwise, the curve induced on $\Delta_{\infty}$
by $\{ P=0\}$ would contain singularities of $\cald$ whose linear part is degenerate: this is impossible
since the singularities of $\cald$ are supposed to be simple. In turn, the last observation implies that
the singular sets of  $\cald$ and of $\fol$ coincide on $\Delta_{\infty}$. Finally, $X_d$ must have a singular set of codimension
at least~$2$ since the singular set of $X$ has codimension two or greater.

The reader is reminded that the restriction of $\cald$ to $\Delta_{\infty}$ coincides with the foliation $\fol_{\infty}$ induced
by $\fol$ on $\Delta_{\infty}$. Besides, if
 $q\in \Delta_{\infty}$ is a (necessarily common) singular point of $\cald, \, \fol$, then the corresponding
linear parts of these foliations at $q$ turn out to coincide as well.

\begin{lema}
\label{Cstar1.1}
With the definition of leaf from Section~$2$, the leaves of $\fol_{\infty}$ are either rational curves
or quotients of $\C$.
\end{lema}

\begin{proof}
We need to show that the leaves of $\fol_{\infty}$ cannot be hyperbolic Riemann surfaces. Since $\C P(n)$ has a K\"ahler
structure, it follows from the main result of \cite{marco2} that the set of parabolic leaves of $\cald$ is a pluri-polar set unless it coincides
with the whole space. Since the leaves of $\cald$ contained in $\C^n$ are clearly parabolic, we conclude that the leaves of $\cald$
contained in $\Delta_{\infty}$ must also be parabolic (or rational). The latter leaves, however, are nothing but the leaves of $\fol_{\infty}$.
\end{proof}

The next step is to consider the foliation $\tilf$ induced on the manifold $M$ by $\fol$. In particular, the
foliation $\tilf_{\infty}$ is naturally identified with $\fol_{\infty}$. In view of the existence of the projections $\calp_0, \, \calp_{\infty}$
introduced in Section~3, Lemma~\ref{Cstar1.1} implies that no leaf of $\tilf$ is hyperbolic. Note that this conclusion cannot directly be
derived from the vector field $X_d$ since $X_d$ need not be complete (it is only semi-complete).

The next proposition relies heavily on Theorem~\ref{maintheo} and it will play a crucial role in the proof of Theorem~B.

\begin{prop}\label{hyperbolicleaves}
There exists a singularity of $\tilf_{\infty}$ providing a sink singularity
for $\calh$ (resp. $\calh^{\theta}$) restricted to a generic leaf of  $\tilf_{\infty}$.
\end{prop}

In the sequel, we shall prove Theorem~B taking for grant Proposition~\ref{hyperbolicleaves}.
To this purpose, Proposition~\ref{hyperbolicleaves} can be summarized by saying that there is
a singularity $q \in \Delta_{\infty}$ of $\tilf$ all of whose eigenvalues $\lambda^q_1, \ldots, \lambda^q_n$ belong
to $\R_+^{\ast}$. This assertion can, in turn, slightly be improved. Indeed, consider the vector field $X_d$ and
local coordinates as in Section~3, where $q \simeq 0$ and $\{x_n = 0\}$ is contained in the hyperplane at infinity
of $M$ (and hence identified to the coordinate ``$z$'' in the mentioned section). In this local coordinates,
the vector field $\tXX_d$ is written in the form
\[
\tXX_d = QY_1 + {\rm h.o.t}
\]
where $Q$ is a rational function and $Y_1$ is a linear vector field with real positive eigenvalues $\lambda^q_1,
\ldots , \lambda^q_n$. Since we have seen that $X^d$ has a singular set of codimension at least~$2$, $Q$
is a rational function possessing a (unique) polar component of degree~$d-1$ passing through $q$ (and coinciding
with $\Delta_{\infty}$) and empty zero divisor. Thus the foliation associated to $QY_1$ must have a
smooth separatrix transverse to $\Delta_{\infty}$. By restricting the vector field to this separatrix the
semi-complete condition implies that $d$ must be equal to~$2$.
Furthermore, the local holonomy of the separatrix in question must be trivial (cf. for example \cite{guillotThesis}). Therefore
each of the eigenvalues $\lambda^q_1 ,\ldots , \lambda^q_{n-1}$ is a multiple of the eigenvalue $\lambda^q_n$.
Hence we can set $\lambda^q_n =1$ and $\lambda^q_1 , \ldots , \lambda^q_{n-1} \in \Z_+$. This refined statement will
lead us to

\begin{lema}
\label{Cstar3.3}
The set formed by the separatrices of $\cald$ at $q$ contains non-trivial open sets of $\C^n$.
\end{lema}

\begin{proof}
Consider the vector field $X$ (resp. foliation $\cald$) on a neighborhood of the singularity $q \in \Delta_{\infty}$.
The reader is recalled that the foliations $\cald$ and $\fol$ have the same singularities in $\Delta_{\infty}$ and each
of these (common) singularities have the same eigenvalues. Thus what precedes implies that $\cald$ has
$n$ eigenvalues $\lambda^q_1, \ldots , \lambda^q_n$ $q$ which are strictly positive integers.
Therefore they belong to the Poincar\'e domain so that the corresponding local vector field is either linearizable or
it admits a Poincar\'e-Dulac normal form. Since the latter possibility was ruled out by assumption, our local vector field
must be linearizable at $q$. The statement follows now from the fact that
$\{ \lambda^q_1, \ldots , \lambda^q_n \} \subset \Z_+$.
\end{proof}

\begin{proof}[Proof of Theorem~B]
Consider a local leaf of $\cald$ defining a separatrix $\mathcal{S}$ for $\cald$ at $q$ as above and not contained in
$\Delta_{\infty}$. Since $\mathcal{S}$ may be singular at $q$, we also consider a local irreducible Puiseux parametrization
$\gamma (t)$ for $\mathcal{S}$ where $t$ is defined on a neighborhood of $0 \in \C$ and $\gamma (0) =q$. Since $\mathcal{S}$
is invariant by $X$, we consider the pull-back $\gamma^{\ast} X$ by $\gamma$ of the restriction of $X$ to $\mathcal{S}$.
The irreducible character of $\gamma$ ensures that $\gamma^{\ast} X$ is holomorphic and semi-complete on a neighborhood
of $0 \in \C$. Therefore the order of $\gamma^{\ast} X$ at $0 \in \C$ belongs to the finite set $\{ 0, 1, 2\}$. Furthermore,
we can actually exclude the case in which this order equals zero since otherwise points of $L$ would reach $q \in \Delta_{\infty}$
in finite time and this contradicts the fact that $X$ is complete on $\C^n$.

Suppose now that the order of $\gamma^{\ast} X$ is~$2$. Modulo adding the point $q$ to the global leaf $L$, it follows that
$L$ is a Riemann surface equipped with a complete holomorphic vector field having a quadratic singularity. Thus $L$ must
be the Riemann sphere (recall that our definition of leaf allows a leaf to go through singular points).
This being valid for set of leaves of $\cald$ containing open sets
of $\C^n$ (cf. Lemma~\ref{Cstar3.3}), it follows that $\cald$ defines a completely integrable rational
foliation on $\C^n$.

Finally, consider the case where the order of $\gamma^{\ast} X$ at $0 \in \C$ equals~$1$. In this case, the time-form induced on
$L$ by $X$ has a non-trivial period. Hence $L$ cannot be an orbit of type
$\C$. Furthermore, $L$ cannot be an elliptic curve either since the restriction of $X$ to the (normalization of) of $L$
contains a singular point of the vector field $X$. Thus, Lemma~\ref{Cstar3.3} and Suzuki's theorem \cite{suzuki} can now
be combined to ensure that the generic orbit of
$X$ has type $\C^{\ast}$. A theorem due to Brunella then guarantees the existence of a non-constant first integral $R$
for $\cald$, cf. \cite{marco1}.

To conclude that $X$ is completely integrable (with rational leaves), we shall proceed by induction on
the dimension~$n$. Indeed, it will suffice to check the theorem for $n=3$ since the induction step will then become clear.
Therefore let us fix a point $p \in \C^n$ that is regular for $X$. Denote by $L_p$ the leaf of $\cald$ through $p$ and let $S$ stand
for the surface level of $R$ containing $p$.  The preceding argument also holds for the restriction of $X$ to $S$. In fact,
consider the foliation $\cald_S$ obtained on $S$ by restricting $\cald$.
The assumptions made about the structure of the singularities of $\cald$ lying in $\Delta_{\infty}$ automatically implies similar
assumptions for the singularities of $\cald_S$ lying in (the desingularization of) $S \cap \Delta_{\infty}$. In other words,
owing to the fact that Suzuki's results hold for Stein manifolds, we can assume
without loss of generality that the ``generic'' orbit of $X$ {\it restricted to the affine part of $S$}\, still is of type
$\C^{\ast}$. Similarly, the above mentioned theorem of Brunella also
holds for the restriction of $X$ to $S$. It means that $\cald_S$ also possesses
a non-constant rational first integral $R_2$. However, for $n=3$, $S$ is an algebraic surface of dimension~$2$. The existence
of $R_2$ then implies that the restriction of $\cald$ to $S$ is completely integrable. Thus the closure of $L_p$ is, in fact,
a (possibly singular) algebraic curve. Since the restriction of $X$ to this compact curve contains a singular point, this curve
must be rational (modulo normalization). Summarizing, we have still concluded that $L_p$ is a (possibly singular)
rational curve. The theorem follows for $n=3$. The general case can be settled with a simple induction argument.
\end{proof}

Let us then close this section with the proof of Proposition~\ref{hyperbolicleaves}. Given a generic leaf $L$
of $\tilf$, we need to show the existence of a singular point $q \in \Delta_{\infty}$ providing a sink singularity for
the restriction of $\calh$ to $L$. In turn, this amounts to find a singular point $q \in \Delta_{\infty}$ at which
$\tilf$ has non-zero eigenvalues $\lambda^q_1, \ldots, \lambda^q_n$ in $\R_+^{\ast}$. Finally, once this singular
point exists, $\cald$ must have the same set of eigenvalues at $q$ so that it will follow from the discussion in Lemma~\ref{Cstar3.3}
that the union of the local separatrices of $\cald$ at $p$ contains
non-empty open sets of $\C^n$. Note also that, since we know {\it a priori}\, that the singularities of $\fol, \, \cald$
at infinity are simple, then the converse also holds, namely:
if a singular point of $\tilf$ in $\Delta_{\infty}$ is such that
the union of all its separatrices contains open sets in $\C^n$, then the corresponding eigenvalues $\lambda_1, \ldots, \lambda_n$
of $\cald$ lie all in $\Z_+^{\ast}$.

The above discussion is connected with an issue appearing in the proof of Proposition~\ref{hyperbolicleaves} and
related to the definition of leaf given in Section~2.2. Recall that a leaf $L$ of $\tilf$ is allowed to contain singular
points of $\cald$ if certain conditions are satisfied. For our purposes, it will be enough to notice that
if a singular point of $\cald$ is ``added'' to a local branch $L$ of a leaf of $\cald$, then the branch in question
defines a local separatrix for $\cald$ at the singular point in question. In particular, if the union of all local
separatrices for singular points of $\cald$ has empty interior then a ``generic'' leaf $L$ of
$\cald$ contains no singular point of $\cald$ (the reader will also note that the singular points of $\cald$ away
from $\Delta_{\infty}$ need not be isolated).

Consider now a singular point $p$ of $\cald$ lying in the complement of $\Delta_{\infty}$. Denote by ${\rm Sep} \, (p)$
the union of all local separatrices of $\cald$ at $p$. Next set
$$
{\rm Sep} \, (\cald) = \bigcup_{p \in {\rm Sing}\, (\cald) \\ p \not\in \Delta_{\infty}}  {\rm Sep} \, (p) \, .
$$
Finally let $\mathcal{U} \subset \C^n$ be defined as the saturated of ${\rm Sep}\, (\cald)$ by the foliation $\cald$. Now, we have:

\begin{lema}
\label{addingonemorelemma}
If the interior of $\mathcal{U}$ is not empty, then $\cald$ is a completely integrable rational foliation.
\end{lema}

\begin{proof}
Let $p$ be a singularity as in the statement and consider a leaf $L$ of $\cald$ defining a local separatrix
for $\cald$ at $p$. Modulo choosing $L$ ``generic'', the restriction of $X$ to $L$ does not vanish identically.
Let $\gamma$ denote an irreducible Puiseaux parametrization of a branch of $L$ through $p$. Restricting $X$ to
$L$ and pulling the resulting restriction back by $\gamma$ yields a local vector field $\gamma^{\ast} X$ defined
about $0 \in \C$. However, since $X$ is holomorphic and semi-complete about $p$, it follows that the order
of the vector field $\gamma^{\ast} X$ at $0 \in \C$ equals either~$1$ or~$2$ (the order cannot vanish since $p$
is singular for $\cald$).

If the mentioned order equals~$2$, then the restriction of $X$ to the (global) leaf $L$ is a complete
vector field on a Riemann surface possessing a singular point of order~$2$. The leaf $L$ is hence a rational
curve. If this condition holds for a $\cald$-saturated set of $\C^n$ possessing non-empty interior, then
$\cald$ must be a completely integrable rational foliation.

Similarly, if the order of $\gamma^{\ast} X$ at $0 \in \C$ equals~$1$, then $L$ has a non-trivial period.
This condition being verified by a set of leaves with non-empty interior implies that the generic
leaf of $\cald$ is of type $C^{\ast}$ in the sense of Suzuki. In turn, Brunella's theorem \cite{marco1}
yields a non-constant rational first integral for $\cald$. The rest of the lemma then follows by the same induction
argument used in the proof of Theorem~B.
\end{proof}

\begin{proof}[Proof of Proposition~\ref{hyperbolicleaves}]
Let us suppose for a contradiction that the statement is false. In view of the material in Section~5, and given that
the singularities of $\cald$  are isolated inside $\Delta_{\infty}$, we conclude that the trajectories of $\calh$ have no
future endpoint once they are contained in a generic leaf of $\cald$. In particular all the corresponding trajectories
have infinite length. This condition will be exploited in the sequel.

Similarly, owing to Lemma~\ref{addingonemorelemma}, we can assume that a ``generic leaf'' of $\cald$ does not
contain any singular point of $\cald$. Besides, ``generic'' leaves are not contained in rational curves, otherwise
Theorem~B would follow at once (and the foliation $\calh$ would have
sink singularities when restricted to a generic leaf of $\tilf_{\infty}$).

Our strategy to obtain a contradiction will consist of showing that a ``generic'' leaf of $\tilf_{\infty}$ cannot be
a ``parabolic'' (i.e. non-hyperbolic) Riemann surface. For this, consider a leaf $L_{\infty} \subseteq \Delta_{\infty}$ and
points $p \in L_{\infty}$ and $q \in \mathcal{P}^{-1}_{\infty}(p)$ as in the statement of Theorem~\ref{maintheo}. Denote
by $L_q$ the leaf of $\tilf$ through $q$ so that $\calp_{\infty} (L_q) =L_{\infty}$. Besides, we can assume without
loss of generality that $L_q$ is a ``generic'' leaf of $\tilf$ in the above sense. Now, note that
$\tXX_d$ is regular over $L_q$. Since $\tXX_d$ is semi-complete, this leaf is recovered by an open set ${\bf U}
\subseteq \C$ which is the domain of definition of the semi-complete flow on $L_q$. In fact, the assumption of
having a semi-complete vector field is equivalent to saying that ${\bf U}$ is a maximal domain of definition for
the solution $\phi$ of the equation associated to this vector field with initial condition at $q$. It then follows that the lift
of $l_q$ in ${\bf U}$ has finite length for the natural euclidean metric of $\C$. Indeed, this length is nothing
but the integral $\int_{l_q} dT$. Hence the mentioned lift converges to a point $t_0$ in the boundary of ${\bf U}$.

Recalling that $L_q$ must be a non-hyperbolic Riemann surface, and noting that it cannot be contained in a rational
curve, it follows that ${\bf U} = \C \setminus \{ t_0\}$. This conclusion actually uses our assumption that $L_q$
contains no singular point of $\cald$. Otherwise, $L_q$ might contain also the point ``$t=\infty$'' what would open
the possibility of having ${\bf U}$ isomorphic to $\C$ minus two points.

In particular, $L_q$ cannot be compactified in an elliptic
curve. Thus $L_q$ is actually not contained in any compact curve. Consider now the plane $\C$ equipped with the
coordinate $t$ so that the canonical form $dt$ coincides with the pull-back by $\phi$ of the time-form induced
on $L_q$ by $X$. Also, denote by $\calh_t$ the lift of the foliation $\calh$ (restricted to $L_q$) to the $\C$-plane.
The foliation $\calh_t$ is naturally a foliation defined on ${\bf U} = \C \setminus \{ t_0\}$. Yet we have the following claim whose
meaning is fully explained below.

\vspace{0.1cm}

\noindent {\it Claim}. The point $t_0$ represents a ``sink singularity'' for $\calh_t$.

\noindent {\it Proof of the Claim}. Clearly $L_q$ possesses a cylindrical end. In other words, the map $\phi$ allows us to realize
a punctured neighborhood of $t_0 \in \C$ as an end of $L_q$. Next consider a point $q' = \phi (t')$ for $t'$ near $t_0$. The trajectory
of $\calh$ through $q'$ is infinite and thus the integral of the corresponding time-form over this trajectory must converge again to a point
lying in the boundary of ${\bf U} = \C \setminus \{ t_0\}$. We then conclude that the mentioned integral converges to $t_0$.
In the $\C$-plane equipped with the coordinate $t$, the preceding translates
into the fact that the integral of the canonical form $dt$ over the lift (by $\phi$) of the $\calh$-trajectory through $q'$ converges to $t_0$.
This lift, however, is nothing but the trajectory of $\calh_t$ through $t'$, it then follows that the trajectory in question converges to $t_0$
and this is the contents of the claim.\qed

\vspace{0.1cm}

The above proof needs further comments. As mentioned $\calh_t$ is the foliation induced on the plane $\C$ by the $1$-form
$\phi^{\ast} \omega_1$ which is meromorphic on ${\bf U} = \C \setminus \{ t_0\}$ but may have an essential
singular point at $t_0$. In our previous discussion, the behavior of $\calh$ was worked out on a neighborhood of points
that are zeros or poles for the $1$-form in question and this led to the singular points of type sink, source and saddle.
Yet, nothing was mentioned about the behavior of $\calh$ about
an essential singular point for the corresponding $1$-form. This is why the
expression sink singularity was written between quotes in the statement of the claim. The proof of the above claim, however,
shows that $t_0$ looks like a sink for $\calh_t$ from a topological viewpoint. In fact, the $\calh_t$-trajectories through
points ``close to $t_0$'' converge to $t_0$. To avoid confusion, this type of situation will be referred
to from now on as constituting an {\it improper sink}\,
for $\calh_t$. Thus, in the ``time coordinate'' $t$, ``improper sinks'' such as $t_0$ are viewed as ``topological sinks'' for
the foliation $\calh_t$ in the above mentioned. The reader will also note that, in this improper sink situation, whereas the
$\calh$-trajectories viewed in the coordinate $t$ appear to converge towards a point $t_0$, when these trajectories are viewed
as they were defined in the actual leaf $L_q$, they turn out to have infinite length.

Summarizing what precedes, the foliation $\calh_t$ on the plane $\C$ has a unique improper sink singularity, corresponding
to $t=t_0$, and no (ordinary) sink singularity. Furthermore all trajectories of $\calh_t$ converge to $t_0$. Indeed, the
integral of the time-form over every trajectory of $\calh$ converges to a point in the boundary of ${\bf U} = \C \setminus
\{ t_0\}$ and hence to $t_0$ itself.

To finish the proof of the proposition, we are going to show that the situation described above cannot happen. In fact, let
$t^{\ast}_1$ be a {\it source singularity of $\calh_t$}. Note that $t^{\ast}_1$ exists since this of type singularity is produced
by the intersections of $\calp_{\infty} (L_q)$ with the hyperplane at infinity $\Delta_{\infty}^{(x,y)}$ in the coordinates
$(x,y)$ (i.e. the hyperplane at infinity corresponding to affine coordinates for $\Delta_{\infty}$). Since $\Delta_{\infty}^{(x,y)}$
neither is invariant by $\tilf_{\infty}$ nor there are singularities of $\tilf_{\infty}$ lying in $\Delta_{\infty}^{(x,y)}$,
it follows that every leaf of $\tilf_{\infty}$ intersects non-trivially $\Delta_{\infty}^{(x,y)}$. Furthermore
these intersections produce then source singularities for $\calh_t$. These source singularities also have residue equal to~$1$.

Next, note that away from saddle-singularities of $\calh_t$, this foliation can be given a
structure of transverse riemannian foliation: just parameterize the trajectories of the orthogonal foliation
by the integral of $\phi^{\ast} \omega_1$. Thus there is a region $I_1^{\ast}$ on a small loop $c$ about
$t_0$ such that the integral of $\phi^{\ast} \omega_1$ over $I_1^{\ast}$ equals~$-1$ (the negative of the
residue of $\phi^{\ast} \omega_1$ at $t^{\ast}_1$). In fact, $I_1^{\ast}$ is obtained by the intersections
with $c$ of the trajectories emanated from $t^{\ast}_1$ (recalling that they all converge to $t_0$). However,
$c$ can be chosen so that $\phi^{\ast} \omega_1$ is bounded on a neighborhood of $c$ and hence the integral
of $\phi^{\ast} \omega_1$ over $c$ is well-defined. To derive a final contradiction, it suffices to observe
that in the complement of the small disc bounded by $c$, there are infinitely many source singularities
$t^{\ast}_1, t^{\ast}_2, \ldots$ for $\calh_t$. In fact, in the above construction, each of these singularities
determine as above a region $I_k^{\ast}$ over which the integral of $\phi^{\ast} \omega_1$ equals~$-1$.
Finally all these regions $I_1^{\ast}, I_2^{\ast}, \cdots$ are clearly pairwise disjoint what immediately
leads to a contradiction with the existence of a bound for $\phi^{\ast} \omega_1$ on a neighborhood of $c$.

Finally, it only remains to check the existence of infinitely many source singularities $t^{\ast}_1, t^{\ast}_2, \ldots$ in the
complement of the disc bounded by $c$ (and containing $t_0$). Let us then suppose that the number
of the mentioned source singularities is finite. It then follows that the solution $\phi$ is bounded
in the complement of a compact region of the plane $\C$. Hence this solution can be extended to a neighborhood
of the infinity in $\C$. Now the image of $\infty$ in $\C \cup \{ \infty \} \setminus \{ t_0 \}$ under the
extension of $\phi$ can be nothing but a singular point of the complete vector field $X$ on the affine $\C^3$.
Modulo using the standard  Remmert-Stein theorem, it follows that $L_q$ contains a local separatrix for an
affine singularity $p$ of $X$. In other words, a generic leaf of $\tilf$ defines a local separatrix of $\tilf$
at one of its singular points: this is impossible since we are treating the case where a generic leaf of $\tilf$
does not define a local separatrix for any of the singularities of $\tilf$ (the case where this situation takes
place having already been settled). This finishes our proof.
\end{proof}

\begin{obs}
\label{tobeincluded}
{\rm Throughout this paper, the affine coordinates $(x,y,z)$ (as well as their higher dimensional versions) used in the
construction of $M$ where generically chosen in the sense that $\Delta_{\infty}^{(x,y)}$ remained away from the singularities
of $\tilf_{\infty}$. In particular, in the context of Theorem~B, this was useful to establish Proposition~\ref{hyperbolicleaves}
guaranteing the existence of a singular point all of whose eigenvalues are strictly positive. This argument will be repeated
in the context of Theorem~C to be proved in Section~7.1. The point we want to make here is that further information, such as
the existence of a second singularity with similar property, may be obtained by choosing the coordinates $(x,y,z)$ so
as to make $\Delta_{\infty}^{(x,y)}$ to pass through the first singularity having only strictly positive eigenvalues.

The above mentioned choice of coordinates $(x,y,z)$ is particularly useful because, under very general conditions,
the ``source character'' of points lying in $\Delta_{\infty}^{(x,y)}$ may cancel out the ``sink character'' of the
singularity in question. In other words, we can avoid the singularity fixed in the beginning of producing ``future endpoints''
for the trajectories of $\calh$. If we know that the leaves of $\tilf_{\infty}$ are parabolic Riemann surfaces, it then
follows the existence of some other point in $\Delta_{\infty}$ yielding a sink singularity for $\calh$}
\end{obs}

\section{Theorem~C and Halphen vector fields}

This section contains the proof of Theorem~C along with a discussion of some of the results obtained in \cite{guillotIHES}
with proofs obtained through the methods developed in the course of this work. The results presented here about
Halphen vector fields are not new and often less sharp than the theorems due to A. Guillot so that they can
be regarded simply as further illustrations of our techniques. Nonetheless the setting introduced here makes
sense for much more general vector fields transverse to singular fibrations and conceivable similar ideas can be applied to
certain Painlev\'e equations as well as to other classical equations including several Chazy's equations.

\subsection{Guillot's lattices of quadratic vector fields : examples and results}

This paragraph concerns an application of our techniques to the work of A. Guillot in \cite{guillotFourier}, \cite{guillotIHES}
(cf. also \cite{guillotThesis} for an introduction to both papers). Let us place ourselves in the context of \cite{guillotFourier}.
Hence $X^2$ stands for a semi-complete homogeneous quadratic vector field with isolated singularities on $\C^n$. We assume
the foliation $\tilf$ associated to $X^2$ on $M$ to leave $\Delta_{\infty}$ (resp. $\Delta_0$) invariant. The pole divisor
(resp. zero divisor) of the lift to $M$ of $X^2$ consists of $\Delta_{\infty}$ (resp. $\Delta_0$) with multiplicity exactly~$1$.
Finally we suppose that $\tilf$ has exactly $2^n -1$ singularities on $\Delta_{\infty}$ (resp. $\Delta_0$) and all of them possess
$n$ eigenvalues different from {\it zero}.

Let us denote by $p_1, \ldots , p_{2^n -1}$ (resp. $q_1, \ldots , q_{2^n -1}$) the singularities of $\tilf$ in $\Delta_0$ (resp. their dual
singularities in $\Delta_{\infty}$). Following \cite{guillotFourier}, \cite{guillotThesis} convention, the eigenvalues of
$\tilf$ at $p_i$ are $1, u_1^i , \ldots , u_{n-1}^i$ where~$1$ is the eigenvalue corresponding to the radial direction. According to Guillot,
for a semi-complete vector field $X^2$ as above, $u_1^i, \ldots , u_n^i$ are all integers. Besides setting $\xi_i = u_1^i \times \cdots \times u_{n-1}^i$,
$X^2$ yields a solution for the equation
\begin{equation}
\sum_{i=1}^{2^n-1} \frac{1}{\xi_i} = (-1)^{n+1} \, . \label{egyptianfraction}
\end{equation}
In particular the problem of precisely classifying vector fields as above is naturally related to these egyptian fractions. In general this problem is
quite intricate as attested by the multitude of interesting examples presented in the above mentioned works. Our contribution to this problem begins
with the following lemma:

\begin{lema}
\label{Guillotlattice1.1}
The leaves of a vector field in Guillot lattice having no dicritical singularity at infinity are hyperbolic Riemann surfaces
(where by dicritical it is meant that all the $n$ eigenvalues of the singularity are positive real).
\end{lema}

\begin{proof}
If there is no dicritical singularity in $\Delta_{\infty}$, there cannot exist a future endpoint for the trajectories of
$\calh, \, \calh^{\theta}$, $\theta \in (-\pi/2, \pi/2)$ for a ``generic'' leaf. In other words, the trajectories of the
foliations $\mathcal{H}^{\theta}$ must be infinite. Thus we can apply Theorem~\ref{maintheo} and the argument of
Proposition~\ref{hyperbolicleaves} to conclude the statement.
\end{proof}

The converse of Lemma~\ref{Guillotlattice1.1} does not hold. A particularly interesting example being provided by Halphen
vector fields (defined on $\C^3$). The dynamics and geometry of these vector fields are beautifully described
in \cite{guillotIHES}, the reader is referred to this paper for background information. From our point of view, these vector
fields provide an interesting example illustrating some aspects of our discussion.

First we note that Halphen vector fields possess dicritical singularities in $\Delta_{\infty}$ (as well as in $\Delta_0$). Yet,
under natural additional conditions, they are semi-complete and possess leaves that are isomorphic (as Riemann surface)
to the unit disc. This deserves some further comments. It was
seen that the existence of dicritical singularities is a necessary condition for the leaves of $\fol, \, \tilf$ to be non-hyperbolic
Riemann surfaces. Indeed a more precise statement holds: in order to have non-hyperbolic leaves, ``most'' $\calh$ trajectories induced
on the leaves of $\fol_{\infty}$ must be of finite length. To have finite length implies that the leaf must have both ``future and
past'' ends and, in the present context, the existence of ``future ends'' implies the existence of dicritical singular points.
In the case of Halphen vector fields, it can be shown by using the ideas of \cite{guillotIHES} that the trajectories of $\calh$
are, indeed, finite and that they tile the corresponding leaves of $\fol$. This might suggest that the leaves of $\fol$
ought to be ``parabolic Riemann surfaces'' what is not the case. Explanation for their hyperbolic character is provided
in \cite{guillotIHES} by carefully exploiting the intimate connection between these equations and the Lie algebra of
${\rm PSL}\, (2, \C)$. Whereas we can hardly improve on \cite{guillotIHES}, we shall provide an explanation
of this fact based on the ideas introduced in this paper. Nonetheless, let us first establish Theorem~C.

\begin{proof}[Proof of Theorem~C]
The fact that the singularity produced by a vector field of special type having no dicritical singularity at infinity cannot
be realized in a K\"ahler manifold follows from the combination of Lemma~\ref{Guillotlattice1.1} with Brunella's theorem as
used in the proof of Lemma~\ref{Cstar1.1}. In fact, if there were a K\"ahler manifold equipped with a holomorphic vector field
$X$ exhibiting this type of singularity at a point $p$, the blow-up of $X$ at $p$ would endow the exceptional divisor with a
foliation whose leaves are hyperbolic Riemann surfaces. Since the remaining leaves associated to the orbits of $X$ must be
parabolic, because $X$ is complete as a globally defined vector field on a compact manifold, a contradiction follows.

The same argument applies to the case of (hyperbolic) Halphen vector fields since these vector fields are known to have
leaves that are isomorphic to the unit disc, cf. \cite{guillotIHES} or the discussion carried out below.

Consider now the dual singularity of vector field of special type having no dicritical singularity at infinity. This means
that we have a neighborhood of $\Delta_{\infty}$ embedded in a manifold $N$. The argument of Theorem~\ref{maintheo} that an
orbit of the vector field entering this neighborhood will contain trajectories converging to the pole locus and such that
the integral of the corresponding time-form is finite. This contradicts the completeness of the vector field in the complement
of its pole locus. Again the argument for Halphen vector fields is totally analogous.
\end{proof}

\subsection{Poincar\'e-type series and Halphen vector fields}

As a matter of fact, Halphen vector fields constitute a particularly remarkable example of semi-complete vector field
belonging to Guillot's lattice whose geometry and dynamics is nicely described in \cite{guillotIHES}. As mentioned
in this section we are going to discuss some aspects of these vector fields from an alternate point of view. Needless to
say that the reader is referred to \cite{guillotIHES} for a fuller discussion.

In the sequel, we shall work on $\C^3$. Let $\cale$ be the radial vector field
$\cale = x\partial /\partial x + y \partial /\partial y + z
\partial /\partial z$ and set $Y = \partial /\partial x +  \partial /\partial y + \partial /\partial z$.
A quadratic homogeneous vector field $X$ is said to be a Halphen vector field if it satisfies $[Y,X] = 2\cale$. It then
follows that the triplet $Y, \cale, X$ form a Lie algebra isomorphic to the Lie algebra of ${\rm PSL}\, (2, \C)$. If $Y,X$
and $\cale$ are identified to the vector fields they induce on $M$ as in Sections~3 and~4 and expressed in the corresponding
coordinates $(x,y,z)$, then $\cale$ becomes $z^{-1} \partial /\partial z$. In other words $X,Y$ commute up
to a vertical vector field. Thus $X$ preserves the projection on $\Delta_{\infty}$ of the foliation induced by $Y$. In other
words, the foliation $\fol_{\infty}^X$ induced by $X$ on $\Delta_{\infty}$ is transverse to the foliation induced by $Y$ on
$\Delta_{\infty}$. Since $Y$ was a constant vector field, the foliation it induced on $\Delta_{\infty}$ is simply a linear pencil
of rational curves. Summarizing $\fol_{\infty}^X$ is transverse to a linear pencil of rational curves.

Once the above observation is made, it is easy to work out the structure of $\fol_{\infty}^X$. It leaves exactly $3$ projective lines $C_1, C_2, C_3$
invariant and these $3$ lines intersect mutually at a radial singularity $P \in \Delta_{\infty}$. Indeed the eigenvalues of $\fol_{\infty}^X$ at $P$ are $1,1$ whereas the
eigenvalues of $\fol^X$ (the foliation induced by $X$ on all of $M$) at $P$ are $1,1,-1$ (the $-1$ eigenvalue being associated to the direction ``$z$'').
Also, if $P' \in \Delta_0$ is the singularity of $\fol^X$ ``dual'' to $P$, then the eigenvalues of $\fol^X$ at $P'$ are $1,1,1$. For $i=1,2,3$, let $p_i, q_i$ denote the
remaining two singularities of $\fol^X$ over $C_i$. We assume that $X$ is semi-complete though this is not really indispensable
in what follows. According to Halphen's results revisited by Guillot, it easily follows that the eigenvalues of $\fol^X$ at $p_i$ (resp. $q_i$)
have the form $-1,-1, m_i$ (resp. $-1, -1, -m_i$), with $m_i \in \N^{\ast}$. The converse also holds though it is harder to prove, see  \cite{guillotIHES}.
In any event it is also easy to check that $\fol^X$ is locally linearizable about $p_i$ (resp. $q_i$). Finally  note that the
convention used above concerning the order of the eigenvalues of $\fol^X$ at $p_i$ (resp. $q_i$) is such that
the first eigenvalue corresponds to the vertical direction ``$z$'', the second to the curve
$C_i$ and the third to a direction contained in $\Delta_{\infty}$ and transverse to $C_i$. This convention is slightly different from \cite{guillotIHES} since the eigenvalue associated to
the direction ``$z$'' is denoted by $-1$ rather than $1$. This change of sign is due to the fact that we consider singularities in $\Delta_{\infty}$ whereas
Guillot considers singularities in $\Delta_0$.

The dynamics of $\fol_{\infty}^X$ is fully encoded in its global holonomy group with respect to a fixed line $\overline{C}$ in the above
mentioned linear pencil that is transverse to $\fol_{\infty}^X$. The preceding also shows that this holonomy group coincides with the subgroup of ${\rm PSL}\, (2, \C)$ generated
by three elements $\xi_1, \xi_2, \xi_3$ which are associated to the local holonomy of each of the three invariant lines. In particular we have
$$
\xi_1 \xi_2 \xi_3 = \xi_1^{m_1} = \xi_2^{m_2}= \xi_3^{m_3} = {\rm id} \, .
$$
In other words, it is a triangle group whose dynamics on $S^2$ is well known: provided that
\begin{equation}
\frac{1}{m_1} + \frac{1}{m_2} + \frac{1}{m_3}  < 1 \, , \label{stillacrossreference}
\end{equation}
this group is conjugate to a subgroup of ${\rm PSL}\, (2, \C)$ and thus it leaves a circle $\Lambda_{\infty} \subset S^2
\simeq \overline{C}$ invariant. Besides each connected
component of $\overline{C} \setminus \Lambda_{\infty}$ is invariant by the action. In fact, on these components
the action is properly discontinuous whereas it is minimal when restricted to the circle $\Lambda_{\infty}$ itself. Similarly is the case
where $1/m_1 + 1/m_2 +1/m_3 =1$, the resulting groups are well-known groups of affine diffeomorphisms associated to special
tiles of the plane. When $1/m_1 + 1/m_2 +1/m_3 >1$ the resulting group is indeed finite and thus it is easy to see that all leaves are
compact.

\begin{obs}
{\rm A note of caution in what precedes concerns the fact the quasi-isometric type of the holonomy group does not {\it a priori}\, determine the quasi-isometric
type of the leaves of $\fol^X$ since the latter are not everywhere transverse to the associated fibration: besides the existence of singularities, there are $3$ fibers
of this fibration that are invariant under $\fol^X$. In particular it is not clear that the leaves of $\fol^X$ must be hyperbolic once Estimate~(\ref{stillacrossreference})
is verified.}
\end{obs}

A similar picture is valid for the foliation $\fol^X$ associated to $X$ on $M$. Clearly $\fol^X$ is transverse to the
codimension~$1$ foliation defined by the
``cone'' over the leaves of $\fol_{\infty}^Y$ and its dynamics is also encoded
in its corresponding holonomy group. This is still generated by the local holonomy {\it with respect to $M$}\, of each of the mentioned three invariant lines.
Each of the three generators is now realized as an automorphism $\Xi_i$ of $\Ff_1$, the line bundle over $\C P(1)$ with Chern class equal to~$1$. In our context
this line bundle is the cylinder over $\overline{C} \subset \Delta_{\infty}$. Besides $\Ff_1$ can be obtained by gluing together two copies of $\C \times \C$
with coordinates $(w,z)$ and $(w', z')$ according to the equation $(w', z') = (1/w , wz)$. The coordinate ``$z$''
of the first copy can be identified with the previous affine coordinate ``$z$'' for $\C^3$. The automorphism $\Xi_i$ fixes the null-section and thus it
can be expressed in the mentioned coordinates as $\Xi_i (w,z) = (\xi_i (w) , B_i(w) z)$
where $\xi_i (w)$ is a homography. Furthermore it is also know that $\Xi_i^{m_i} = {\rm id}$ on $\Ff_1$ and $\xi_i^{m_i} = {\rm id}$ on $\overline{C}$. Thus
\begin{equation}
B_i(w) \times B_i( \xi_i (w)) \times \cdots \times B_i (\xi_i^{m_i-1} (w)) = 1 \, . \label{product=1}
\end{equation}

Next recall that the M\"obius group has a natural extended action to $\Ff_1$ consisting of multiplying vectors in the fibers by the square root of its derivative.
In other words, if $\xi$ is a homography, then its extended action on a pair $(w,z) \in \Ff_1$ is simply
$$
\xi . (w,z) = (\xi (w) , \sqrt{\xi' (w)} \, z ) \, .
$$
It is to be noted that the square root of the derivative of a homography is well-defined so that the claim follows from observing that
$\Ff_1 \otimes \Ff_1$ is isomorphic to the tangent bundle of $\overline{C}$. In particular a more explicitly expression for $\Xi_i$ can be derived
as follows. Let $q_i, q_{i+1}$ ($\neq 0, \infty$) denote the two fixed points of $\xi_i$ in $\overline{C}$. The transformation
$$
\sigma_i (w) = \frac{w-q_i}{w-q_{i+1}} \; \; \; \;  ; \; \; \; \; \sigma_i^{-1} (w) =\frac{q_{i+1} w -q_i}{w-1}
$$
conjugates $\xi_i$ to a homography fixing $0, \infty$. In this coordinate $\overline{w}$, $\xi_i$ must take on the form
$\overline{w} \mapsto k_i^2 \overline{w}$ where $k_i^2 = e^{2\pi \sqrt{-1}/m_i}$ since $\xi_i^{m_i} = {\rm id}$ on $\overline{C}$. Furthermore, in the coordinate
$\overline{w}$ it is clear that $B_i (\overline{w})$ must be constant so as to allow $\Xi_i$ to have a holomorphic extension to a fibered neighborhood of
$\infty$. Setting $B_i (\overline{w}) =B_i$ for this particular choice of coordinates, it follows that the expression of $\Xi_i$ in the initial coordinate $w$ is given by
\begin{eqnarray*}
\Xi_i (w,z) &=& (\xi_i (w) , B_i (w) z) = \left( \sigma_i^{-1} (k_i^2 \sigma_i (w)) , \sqrt{(\sigma_i^{-1})'\vert_{k_i^2 \sigma_i (w)}} B_i \sqrt{\sigma_i' (w)} z \right) \\
& = & \left( \frac{(k_i^{-1} q_i -k_i q_{i+1}) w + q_i q_{i+1} (k_i -k_i^{-1})}{(k_i^{-1} -k_i) w + (k_i q_i -k_i^{-1} q_{i+1})} ,
\frac{k_i^{-1} B_i (q_i -q_{i+1}) z}{(k_i^{-1} -k_i) w + (k_i q_i -k_i^{-1} q_{i+1})} \right) \\
& = & (\xi_i (w) , k_i^{-1} B_i \sqrt{\xi_i' (w)} z ) \, .
\end{eqnarray*}
The above formulas are going to enable us to understand the solutions of Halphen vector fields from the point of view worked out in this work.
Let us first consider the special case where $1/m_1 + 1/m_2 +1/m_3 =1$. In this case the three homographies $\xi_1 ,\xi_2 , \xi_3$ share a common fixed
point. Choosing coordinates $(w,z)$ where this point is $\infty$ it follows that $\xi'(w)$ is constant and thus $\Xi_i (w,z) = (\xi_i (w) , A_i z)$
for certain constants $A_i$, $i=1,2, 3$. As a sort of converse to Theorem~\ref{maintheo}, we obtain the following:

\begin{prop}
\label{entireHalphen}
Suppose that $1/m_1 + 1/m_2 +1/m_3 =1$. Then all the leaves of $\fol^X$ are uniformized by $\C$ as Riemann surfaces.
\end{prop}

\begin{proof}
The proof is rather  simple. Let $L$ be a generic leaf of $\fol_X$ and let $\Ff_1$ be as above. Because all the $A_i$'s are constant,
the intersection points $p_1, p_2 \ldots$ between $L$ and $\Ff_1$ have their distance to $\Delta_{\infty}$ bounded from below by a positive constant.
Thus the time-form $dT_L$ induced by $X$ on $L$ is uniformly bounded on a neighborhood $W$ of $\{ p_1, p_2 \ldots \}$.
Consider now the maximal domain ${\bf U} \subseteq \C$ of a solution $\phi$ of $X$. Suppose for a contradiction that $\C \setminus {\bf U} \neq \emptyset$
and choose a point $T \in \C$ lying in the boundary $\partial {\bf U}$ of ${\bf U}$. Finally let $t_1, t_2, \ldots $ be a sequence of points in ${\bf U}$ converging to $T$.
Given the above mentioned structure of $\fol_X$ as a foliation transverse to a fibration, we can assume without loss of generality that $\phi (t_j)$ lies in $W$.
Since $dT_L$ is uniformly bounded on $W$, it follows that $\phi$ is defined on a disc of uniform (positive) radius about each $t_j$. This is however impossible
since $t_j \rightarrow T \in \partial {\bf U}$. The proposition is proved.
\end{proof}

From now on let us focus on the more interesting case where $\xi_1, \xi_2, \xi_3$ generate a hyperbolic triangle group,
i.e. in the case where $1/m_1 + 1/m_2 +1/m_3 < 1$. The fixed points of
$\xi_1, \xi_2, \xi_3$ are three (mutually different) points $q_1 ,q_2, q_3 \in \overline{C}$. By this we mean that for each $i \in \{ 1,2,3 \}$, the
homography $\xi_i$ fixes the points $q_i, q_{i+1}$ (where $q_{3+1} =q_1$). In the present case there is no coordinate $w$ where all the
$\xi_i$ become affine maps. Thus we shall need to work with the full information provided by the action of $\Xi_1, \Xi_2, \Xi_3$ on $\Ff_1$. We are going to
show that the geometry of the leaves is related to the Poincar\'e series with exponent $1/2$. Next let us denote by $\Gamma$ the group generated by
$\xi_1, \xi_2, \xi_3$ and consider its Cayley graph with respect to the generating set given by $\xi_1, \xi_2, \xi_3$ and their inverses. Choose a ``geodesic ray''
$\gamma_0 ={\rm id}, \gamma_1, \gamma_2 , \ldots$
in the Cayley graph going from the identity to an ``end'' of the graph. We have:

\begin{prop}
\label{hyperbolic-vs-entire}
Suppose the the holonomy group is a hyperbolic triangle group. Let $L$ denote a leaf of $\fol_X$ passing through a point
$(w_0, z_0)$. As a Riemann surface $L$ is hyperbolic provided that the series
\begin{equation}
S (w_0) = \sum_{j=0}^{\infty} \Vert \gamma_j (w_0) \Vert^{1/2} \label{georay}
\end{equation}
is convergent for more than one geodesic ray. If this series diverges for all geodesic rays as above then $L$ is a quotient of $\C$.
\end{prop}

\begin{proof}
The proof consists of elaborating further on the argument used in Proposition~\ref{entireHalphen}. Again denote by
${\bf U}$ the maximal domain of definition of the solution $\phi$ of $X$ satisfying $\phi (0) =(w_0, z_0)$, $z_0 \neq 0$.
Denote by $L_{\infty}$ the projection of $L$ on $\Delta_{\infty}$. By virtue of the structure of the foliation $\fol_{\infty}^X$ on
$\Delta_{\infty}$, we know that $L_{\infty}$ is a ramified covering of $\C P(1)$ where the ramified points sit over three
points of $\C P(1)$ (identified to the three invariant fibers of $\fol_{\infty}^X$). We can then think of $L_{\infty}$ as being the universal
covering of $\CP(1)$ minus $3$ points modulo adding to it the ramification points. In particular there is a natural sense in considering
{\it fundamental domains}\, in $L_{\infty}$. The leaf $L_{\infty}$ can then be considered as the union of the corresponding
fundamental domains $L_{\infty}^{(0)},  L_{\infty}^{(1)}, \ldots$ such that $L_{\infty}^{(j)} = \gamma_j (L_{\infty}^{(0)})$. These domains
have natural lifts to the leaf $L \subset M\setminus (\Delta_0 \cup \Delta_{\infty})$ which, modulo re-numeration, will be denoted by
$L^{(0)},  L^{(1)}, \ldots$ in such way that $L^{(j)} = \Gamma_j (L^{(0)})$ where $\Gamma_j$ is the automorphism of $\Ff_1$ corresponding
to the action of $\gamma_j$ on $\overline{C}$.

On $L$ (or on its universal covering if necessary), we define the map
$$
\mathcal{D}_L (p) = \int_{(w_0, z_0)}^p dT_L
$$
where $dT_L$ stands for the time-form induced on $L$. Since $X$ is semi-complete, $\mathcal{D}_L$ provides a diffeomorphism
from the (universal covering of) $L$ to ${\bf U}$. Let then $U_{(j)}$ be the image of $L^{(j)}$ by $\mathcal{D}_L$. The
above assertion implies that the set $U_{(j)}$ ``tile'' ${\bf U}$ without overlapping and modulo adding the image of ramification points
(involved in the preceding definition of the fundamental domains $L_{\infty}^{(0)},  L_{\infty}^{(1)}, \ldots$). Next recall that
the affine structure on $L_{\infty}$ is uniformly bounded (from below and by above). Combining this fact to the expression for
$dT_L$ arising from Formula~\ref{tXX}, with $d=2$, it follows the existence of sequences $\{ r_j\}$, $\{ R_j\}$,
$0<r_j < R_j$, $j =1, 2 ,\ldots$, satisfying the conditions below.
\begin{enumerate}
\item There are constants $0<c<C$, independent of $j$, such that
$$
c \Vert \pi_2 (\Xi_j (w_0, z_0)) \Vert  r_j < R_j < C \Vert \pi_2 (\Xi_j (w_0, z_0)) \Vert  r_j \,   ,
$$
where $\pi_2$ stands for the projection on the second coordinate (i.e. the fiber of $\Ff_1$).

\item The image of $U_{(j)}$ under $\mathcal{D}_L$ contains a ball of radius $r_j$ about $\mathcal{D}_L (\Xi_j (w_0, z_0))$. Similarly
the same image is contained in ball of radius $R_j$ about $\mathcal{D}_L (\Xi_j (w_0, z_0))$.
\end{enumerate}

It then becomes clear that ${\bf U}$ must be the whole $\C$ provided that the series $\sum_{j=0}^{\infty} \Vert \pi_2 (\Xi_j (w_0, z_0) \Vert$
diverges for every geodesic ray. Conversely, if this series
is convergent, then we can easily construct a ``small'' piece of continuum contained in the boundary of ${\bf U} \subset \C$. Thus ${\bf U}$
must be a hyperbolic domain so that $L$ itself must be a Riemann surface covered by the unity disk. To conclude the proof of the proposition
it is therefore sufficient to check that the series $\sum_{j=0}^{\infty} \Vert \pi_2 (\Xi_j (w_0, z_0) \Vert$ converges (resp. diverges) if and only
if so does the ``reduced Poincar\'e series'' in the statement. This is however clear since $\vert k_i^{-1} B_i \vert =1$ as an
immediate consequence of Formula~(\ref{product=1}). The proof of the proposition is over.
\end{proof}

Next we are going to show that the series~(\ref{georay}) always converges provided that $w_0 \in \overline{C} \setminus
\Lambda_{\infty}$. Some indications concerning the behavior of this series in the case $w_0 \in \Lambda_{\infty}$ will also be provided. As mentioned
the series in question differs from the usual Poincar\'e series since the sum is not carried over the entire group but only over those elements
belonging to a chosen ``geodesic ray''. Indeed the ``full'' Poincar\'e series of $\Gamma$ with exponent $1/2$ diverges as it follows from well-known
results due mainly to Sullivan (see \cite{bulletin} for an overview of the standard theory).

\begin{lema}
\label{alwaysconverging}
The series~(\ref{georay}) converges provided that $w_0 \in \overline{C} \setminus \Lambda_{\infty}$.
\end{lema}

\begin{proof}
The argument is rather simple. Consider the action of $\Gamma$ in the $w$-plane. This action preserves a circle identified to $\Lambda_{\infty}$.
In particular $\Gamma$ is realized as a Fuchsian group, i.e. a discrete group of automorphisms of the hyperbolic disc. Next, since $\Gamma$
acts on the hyperbolic disc, it is easy to see that the convergence of this series does not depend on $w_0$. In other words, the series~(\ref{georay})
converges for $w_0$ if and only if it converges at $0$. It is then sufficient to check that the series converges for $w_0 =0$ (identified to the origin
of the disc). For this consider again the geodesic ray $\gamma_0 ={\rm id}, \gamma_1, \gamma_2 , \ldots$ in the Cayley graph of $\Gamma$ and set
$a_j = \gamma_j (0)$. In the mentioned Cayley graph the distance between the identity and $\gamma_j$ is obviously $j$. The existence of a quasi-isometry
between this graph and the hyperbolic disc implies that the hyperbolic distance $d_H (0, \gamma_j (0))$ between $0$ and $\gamma_j (0)$ satisfies
$c j < d_H (0, \gamma_j (0)) < C j$
for appropriate uniform constants $C > c >0$. The standard formula for the length of a minimizing geodesic in the hyperbolic unit disc joining $0$ to a
point $a$ of this disc (naturally satisfying $\Vert a \Vert <1$) yields
$$
\frac{e^{cj} -1}{e^{cj} + 1} \leq \Vert a_j \Vert \leq \frac{e^{Cj} -1}{e^{Cj} + 1} \, .
$$
On the other hand the coefficients of the hyperbolic metric at $0$ and at $a_j$ allow us to obtain a formula for the derivative of $\gamma_j$ at $0$. Combined
with the above estimates this formula gives
$$
\Vert \gamma_j' (0) \Vert \leq \frac{4e^{Cj}}{(1+ e^{Cj})^2}
$$
and thus $\Vert \gamma_j' (0) \Vert^{1/2} \leq 2 e^{Cj/2} / (e^{Cj} +1)$. The convergence of the mentioned series follows immediately.
\end{proof}

To close this discussion let us briefly indicate the behavior of the series~(\ref{georay}) for points $w_0$ lying
in $\Lambda_{\infty}$. Since $\gamma_j$ takes $0$ to $a_j$, it follows that $\gamma_j (w) = e^{2\pi i \theta} (w+a_j)/
(1 + \overline{a}_j w)$, for some $\theta \in [0,1)$. In particular
$$
\Vert \gamma_j' (w) \Vert = \frac{1 - \Vert a_j \Vert^2}{(1 + \overline{a}_j w)^2} \, .
$$
Because $\Gamma$ is discrete, it follows that $\Vert a_j \Vert \rightarrow 1$ as $j \rightarrow + \infty$. Set $a_j = \Vert a_j \Vert e^{i\theta_j}$ and
$w_0 = e^{-i\theta_j -\pi + \alpha_j}$ so that
\begin{equation}
\Vert \gamma_j' (w_0) \Vert = \frac{1 - \Vert a_j \Vert^2}{1 + \Vert a_j \Vert^2 -2\Vert a_j \Vert \cos (\alpha_j)} \, . \label{justtobequoted}
\end{equation}
Next note that $\Vert \gamma_j' (w_0) \Vert > 1/2$ as long as $\cos (\alpha_j) > \Vert a_j \Vert$. In particular if there are infinitely many indices $j$
satisfying this condition the corresponding series will diverge. In this case there is a
subsequence of $\{a_j\}$ converging ``almost radially'' for $-w_0$. Conversely if the denominator in~(\ref{justtobequoted}) is bounded from below by
some positive constant, then the argument used in Lemma~\ref{alwaysconverging} ensures again the convergence of series~(\ref{georay}). In general
we are led to a finer analysis taking into account the ``conic approximation'' of $-z_0$ by the sequence $a_j =\gamma_j (0)$. Not surprisingly the behavior of
series~(\ref{georay}) on $\Lambda_{\infty}$ depends on the initial point $w_0$: for some values of $w_0$ it diverges whereas for others it is convergent.

Most properties of Halphen vector fields become encoded in the extended dynamics of the group generated by
$\Xi_1, \Xi_2, \Xi_3$ on $\Ff_1$. For example the
study of first integrals for Halphen vector fields amounts to searching functions that are invariant by this action.
In particular on each connected component of $\overline{C} \setminus \Lambda_{\infty}$ it is not hard to construct
``automorphic functions'' for this group so that on the corresponding open sets on $M$ the Halphen vector field possesses
a holomorphic first integral. It is also not very hard to check that the Halphen vector field does not admit a holomorphic
(or meromorphic) first integral on the set corresponding to $\Lambda_{\infty}$ (which has real dimension equal to~$5$).
Yet on the latter set, there is a real-valued first integral for the equation that is actually globally defined on $M$.
We shall not pursue this type of discussion here not only because Halphen vector fields were detailed studied in \cite{guillotIHES}
but also because the corresponding issues will no longer be in line with the main ideas of this paper.

\section{Appendix: some problems}

In closing we would like to suggest some problems for which the method developed in this work may provide some
insight.

\noindent $\bullet$ {\bf Equations with Painlev\'e property}. It is clear that our methods are well designed for investigating these equations, specially
if meromorphic solutions defined on all of $\C$ are targeted. Since Nevanlinna theory has become an important tool in these questions,
it would be interesting to compare both approaches which, as suggested by Theorem~A', may have some non-trivial points of contact.
Here it may also be worth reminding the reader that our methods apply equally well to rational, as opposed to polynomial, vector fields.

Also many special equations including Painlev\'e's equations admit a formulation involving associated transverse fibrations. The domain of
definition of their solutions can then be studied by adapting the argument used in Section~7.2 to handle the case of Halphen equations.
For example, consider the case of P-$VI$ which contains the remaining Painlev\'e equations as particular cases.
The standard Hamiltonian formulation of P-$VI$ allows one to consider its {\it holonomy group at infinity}\, as in \cite{lubo} which is shown
to be virtually abelian in the same paper. By construction, this group consists of holomorphic diffeomorphisms of $\C^2$. The study
of the corresponding dynamics should then recover, in particular, the domain of definition of the equations in question as well as to yield
new insight into the way they bifurcate as the parameters vary. This certainly relates to several other aspects of Painlev\'e's equations
such as global linearization and Riemann-Hilbert correspondence.

\noindent $\bullet$ {\bf Confinement questions for real equations}. These problems were brought to our attention by F. Cano and C. Roche. Let us
consider the case of singular points of real analytic vector fields as in \cite{cano1}, \cite{cano2}. Several important issues appearing in the
study of these singular points, such as existence of iterated tangents, depend strongly on having suitable estimates on the ``speed'' with which
a solution may converge to the origin. Though these problems are real rather than complex, our method can provide information on certain cases.
Let us explain the main adaptations needed for it.

The central issue is that the exceptional divisor in this context is obtained by a {\it real blow-up}\, and the foliation induced on it has
real dimension~$1$ as opposed to complex dimension~$1$. Thus there is no freedom to choose any ``steepest descent direction'', indeed, we can only follow the
given trajectory in a chosen orientation. However, under some assumptions concerning the first non-zero homogeneous component of the vector
field, it is possible to guarantee a ``normal contractive holonomy'' as the one considered in Section~3.

Once this contraction is established, the fact that we are dealing with real solutions, rather than with complex ones, becomes an advantage
for two main reasons which can more transparently be appreciated in dimension~$3$. The first one has to do with Lemma~\ref{lineatinfinity}. Recall
that $\Delta_{\infty}^{(x,y)}$ consists of source singularities for $\calh$ whereas $\Delta_{0}^{(x,y)}$ is constituted by sink singular points of $\calh$. This
is the main reason why we have worked with $\Delta_{\infty}^{(x,y)}$ rather than with $\Delta_{0}^{(x,y)}$: the fact that $\calh$ has source
singularities at $\Delta_{\infty}^{(x,y)}$ is fundamental for the proof of Lemma~\ref{remainingcompact} and hence for Theorem~\ref{introducelabel4}.
Nonetheless, in the real case, it is very often that these trajectories remain in a compact part of the ``line at infinity'' contained in $\Delta_0$.
Therefore, modulo appropriate assumptions satisfied by many foliations, the statement of Theorem~\ref{introducelabel4} remains valid
in the present setting. When working in dimension~$3$,
the situation becomes much better: dealing with real equations the divisor playing the role of $\Delta_0$ is no longer the complex projective plane but the
{\it real projective plane}. The asymptotic behavior of the real foliation induced on ``$\Delta_0$'' is therefore easily described by the standard
Poincar\'e-Bendixson theory. It is then reasonable to expect it to provide detailed information on the structure of the singular point in question.
It would be nice to know if the theory put forward by F. Cano, R. Moussu and their collaborators can be furthered by this type of analysis.

Let us also mention that the study of complete real vector fields on $\R^3$ may be approached from this point of view as pointed out to us by A. Guillot.
In particular this type of idea apply to Lorenz equation where the presence of a very special ``saddle-node singularity'' in $\Delta_{\infty}$ seems to
organize much of the information concerning the corresponding dynamics. Similar ideas concerning Lorenz systems were also considered by
X. Gomez-Mont and, independently, by J.-P. Ramis.

\noindent $\bullet$ {\bf Actions of ${\rm SL}\, (2, \C)$}. In \cite{guillotIHES} A. Guillot classified complex $3$-folds that are quasi-homogeneous under an
action of ${\rm SL}\, (2, \C)$. The fundamental point of describing non-locally free actions of ${\rm SL}\, (2, \C)$ arises from considering, necessarily
homogeneous, vector fields $X$ verifying the equation $[Y,X] = 2\cale$ where $\cale = x\partial /\partial x + y \partial /\partial y + z
\partial /\partial z$ and $Y = \partial /\partial x +  \partial /\partial y + \partial /\partial z$. These vector fields $X$ are by definition Halphen vector fields.
The same equation can be considered in every dimension to yield higher dimensional analogues of Halphen vector fields that are naturally
related to ${\rm SL}\, (2, \C)$-actions. Let $\cale$ and $Y$ still denote the constant and the radial vector fields in dimension~$n$.
The previous discussion, shows that $X$ still induces a foliation on $\Delta_{\infty}$ transverse to the linear pencil obtained
from $Y$. The interesting novelty appearing in dimensions~$\geq 4$ is that the basis of this pencil is isomorphic to
$\C P (n-2)$ which inherits of a non-trivial foliation induced by $X$ since $n \geq 4$. We do not know how ``wild'' may this latter foliation be, but an
understanding of its dynamics would allow us to generalize the arguments given in Section~7.2 and extend the results of Halphen and Guillot.

\noindent $\bullet$ {\bf Singular points of vector fields on complex K\"ahler $3$-folds}. Let $X$ be a holomorphic vector field defined on a
compact K\"ahler manifold $M$ of dimension~$3$ and consider a singular point $p \in M$ of $X$.
Denote by $X_d$ the first non-zero homogeneous component of $X$ at $p$. It is known
that $X_d$ is a (homogeneous) semi-complete vector field on all of $\C^3$. The problem is then to classify all possible models for $X_d$. As already
explained, this is equivalent to finding all possible normal forms for the {\it top-degree}\, homogeneous component of a complete polynomial
vector field on $\C^3$.

Let us consider the foliation $\tilf_{\infty}$ induced by $X_d$ on $\Delta_{\infty}, \, \Delta_0$ and we assume once and for all that
$\tilf_{\infty}$ is not a pencil.
Since all leaves of $\tilf_{\infty}$ are parabolic Riemann surfaces, we can apply McQuillan theorem, as formulated in \cite{book}, to conclude that
$\tilf_{\infty}$ is transverse to a pencil of genus either~$0$ or~$1$. The idea is then to resort again to arguments similar to those developed in Section~7.2
to work out the structure of $\tilf_{\infty}$. The first thing to be proved is that the global holonomy arising from $\tilf_{\infty}$ is conjugate to a subgroup of the
affine group of $\C$. In particular it has a fixed point corresponding to an algebraic curve $C$ invariant by $\tilf_{\infty}$.
Next we should consider the ``extended'' holonomy of $\tilf$ taking values on the group of automorphisms of $F_1$ as previously done.
The resulting group will still be elementary and it must be ``compatible''
with the affine structure induced on $C$ which is necessary uniformizable.
Finally the cases in which the genus of $C$ is equal to~$0$ and to~$1$ must separately be considered. The most interesting case corresponds to genus~$0$
since there are more possibilities for the affine structure on $C$. Besides ``rational orbits'', we must expect to find the ``elliptic'' orbifolds associated
to the triangle groups $(2,2, \infty)$, $(2,3,6)$, $(2,4,4)$ and $(3,3,3)$ in addition to the orbifold $(2,2,2,2)$ which appear in connexion with the classical
integrable equations of the Euler top spin.

\begin{flushleft}
{\sc Julio Rebelo} \\
Institut de Math\'ematiques de Toulouse\\
118 Route de Narbonne\\
F-31062 Toulouse, FRANCE.\\
rebelo@math.univ-toulouse.fr

\end{flushleft}

\begin{flushleft}
{\sc Helena Reis} \\
Centro de Matem\'atica da Universidade do Porto, \\
Faculdade de Economia da Universidade do Porto, \\
Portugal\\
hreis@fep.up.pt \\

\end{flushleft}

\end{document}